\documentclass[11pt]{article}
\usepackage[english]{babel}
\usepackage[utf8]{inputenc}
\usepackage[T1]{fontenc}
\usepackage{amsmath,amsthm,amsfonts,amssymb,bbm,epsfig,verbatim,enumerate,stackrel}
\usepackage{graphicx}
\usepackage{fullpage}
\usepackage{hyperref}
\usepackage{color}

\newtheorem{theorem}{Theorem}
\newtheorem{lemma}[theorem]{Lemma}
\newtheorem{proposition}[theorem]{Proposition}
\newtheorem{corollary}[theorem]{Corollary}
\newtheorem{definition}[theorem]{Definition}

\newtheorem{remark}[theorem]{Remark}
\newtheorem{assumption}{Assumption}

\numberwithin{theorem}{section}
\numberwithin{equation}{section}
\numberwithin{figure}{section}

\DeclareMathOperator{\rad}{rad}

\newcommand{\PP}{\mathbb{P}}
\newcommand{\EE}{\mathbb{E}}
\newcommand{\RR}{\mathbb{R}}
\newcommand{\CC}{\mathbb{C}}
\newcommand{\ZZ}{\mathbb{Z}}

\newcommand{\TT}{\mathbb{T}}
\newcommand{\calE}{\mathcal{E}}
\newcommand{\calF}{\mathcal{F}}
\newcommand{\calH}{\mathcal{H}}
\newcommand{\calN}{\mathcal{N}}
\newcommand{\calS}{\mathcal{S}}
\newcommand{\ve}{\varepsilon}
\newcommand{\ind}{\mathbbm{1}}
\newcommand{\lra}{\leftrightarrow}

\newcommand{\ol}{\overline}
\newcommand{\ul}{\underline}
\newcommand{\dol}[1]{\ol{\ol{#1}}}
\newcommand{\dul}[1]{\ul{\ul{#1}}}
\newcommand{\doul}[1]{\ul{\ol{#1}}}

\newcommand{\Ann}{A} 
\newcommand{\Ball}{B} 
\newcommand{\arm}{\mathcal{A}} 
\newcommand{\circuit}{\mathcal{C}} 
\newcommand{\circuitevent}{\mathcal{O}} 
\newcommand{\circuitarm}{\mathcal{OA}} 
\newcommand{\cluster}{\mathcal{C}} 
\newcommand{\Cinf}{\cluster_\infty} 
\newcommand{\Ch}{\mathcal{C}_H} 
\newcommand{\Cv}{\mathcal{C}_V} 
\newcommand{\colorseq}{\mathfrak{S}} 
\newcommand{\lclus}{\cluster^{\textnormal{max}}} 
\newcommand{\lclusbar}{\ol{\cluster}^{\textnormal{max}}} 
\newcommand{\din}{\partial^{\textnormal{in}}} 
\newcommand{\dout}{\partial^{\textnormal{out}}} 
\newcommand{\de}{\partial^{\textnormal{e}}} 
\newcommand{\next}{\Psi}
\newcommand{\aval}{\mathfrak{a}}
\newcommand{\avalnb}{\mathfrak{n}}
\newcommand{\PPh}{\overline{\mathbb{P}}}

\begin{document}

\title{Near-critical avalanches\\ in 2D frozen percolation and forest fires}

\author{Wai-Kit Lam\footnote{University of Minnesota}, Pierre Nolin\footnote{City University of Hong Kong. Partially supported by a GRF grant from the Research Grants Council of the Hong Kong SAR (project CityU11306719).}}

\date{}

\maketitle

\begin{abstract}
We study two closely related processes on the triangular lattice: frozen percolation, where connected components of occupied vertices freeze (they stop growing) as soon as they contain at least $N$ vertices, and forest fire processes, where connected components burn (they become entirely vacant) at rate $\zeta > 0$. In this paper, we prove that when the density of occupied sites approaches the critical threshold for Bernoulli percolation, both processes display a striking phenomenon: the appearance of near-critical ``avalanches''.

More specifically, we analyze the avalanches, all the way up to the natural characteristic scale of each model, which constitutes an important step toward understanding the self-organized critical behavior of such processes. For frozen percolation, we show in particular that the number of frozen clusters surrounding a given vertex is asymptotically equivalent to $(\log(96/5))^{-1} \log \log N$ as $N \to \infty$. A similar mechanism underlies forest fires, enabling us to obtain an analogous result for these processes, but with substantially more work: the number of burnt clusters is equivalent to $(\log(96/41))^{-1} \log \log (\zeta^{-1})$ as $\zeta \searrow 0$. Moreover, almost all of these clusters have a volume $\zeta^{- 91/55 + o(1)}$.

For forest fires, the percolation process with impurities introduced in \cite{BN2018} plays a crucial role in our proofs, and we extend the results in that paper, up to a positive density of impurities. In addition, we develop a novel exploration procedure to couple full-plane forest fires with processes in finite but large enough (compared to the characteristic scale) domains.

\bigskip


\smallskip

\end{abstract}

\tableofcontents

\section{Introduction}

\subsection{Frozen percolation and forest fires} \label{sec:intro_intro}

This paper is concerned with two families of processes, \emph{frozen percolation} and \emph{forest fires}, defined on a simple graph $G = (V,E)$ (where as usual, $V$ and $E$ contain the vertices and the edges of $G$, respectively). We describe them in an informal way now, and refer the reader to Section~\ref{sec:def_processes} for precise definitions. They are all constructed from an underlying \emph{birth process} on $G$, in which the vertices can be in two states, that we interpret as ``containing a particle'' (\emph{occupied}), e.g. a tree, or ``being empty'' (\emph{vacant}). In this process, all vertices are initially vacant (at time $t = 0$), and then become occupied at rate $1$, independently of each other. Hence, at time $t \geq 0$, each vertex is occupied with probability $p(t) = 1 - e^{-t}$, and vacant otherwise. Occupied vertices can be grouped into maximal connected components, called occupied clusters.

First, we consider the class of growth processes known as frozen percolation. More specifically, volume-frozen percolation (FP) with parameter $N \geq 1$, or \textbf{$N$-frozen percolation} for short, is obtained by modifying the dynamics of the birth process in the following way. Each time a vertex tries to change its state from vacant to occupied, it is not always allowed to do so: it becomes occupied only if it is not adjacent to an occupied cluster containing at least $N$ vertices. In other words, we let each occupied cluster $\cluster$ grow as long as its volume, i.e. the number $|\cluster|$ of vertices that it contains, is at most $N-1$. If this volume happens to cross the threshold $N$ at some time $t$, then the cluster stops growing: we say that it \emph{freezes}, and the vertices along its outer boundary, which are all vacant at time $t$, remain in this state forever. Occupied vertices belonging to such a cluster are said to be frozen (at all times $t' \geq t$). Note that a given cluster may never reach volume $N$, if it happens to be ``trapped'' by frozen clusters inside a region with volume smaller than $N$. The probability measure governing $N$-frozen percolation on $G$ is denoted by $\PP_N^{(G)}$.

We also analyze forest fire processes on $G$, for some given ignition rate $\zeta > 0$. Vertices again turn from vacant to occupied at rate $1$. In addition, each vertex $v \in V$ is hit by lightning at rate $\zeta$: when this happens, all the vertices in the occupied cluster containing $v$ become vacant instantaneously, while nothing happens if $v$ is vacant. This process corresponds to the Drossel-Schwabl model \cite{DrSc1992}. It has garnered a lot of attention since its introduction in 1992, but it is still not well understood. We also consider a variant of this process where vertices, once burnt, stay in this state forever, and cannot become occupied at a later time (so that a vertex can be in three possible states: vacant, occupied, or burnt). We refer to this modified version as forest fire without recovery, abbreviated as \textbf{FFWoR}, while the original process is called forest fire with recovery (\textbf{FFWR}). When studying these processes, we use the notations $\PP_\zeta^{(G)}$ and $\ol{\PP}_\zeta^{(G)}$, respectively.

The first version of frozen percolation was introduced by Aldous \cite{Al2000}, inspired by sol-gel transitions \cite{St1943}. In that paper, two graphs $G$ are considered: the infinite $3$-regular tree, and the planted binary tree (in which the root vertex has degree $1$, but all other vertices have degree $3$). More precisely, an edge version of frozen percolation is studied, where connected components freeze as soon as they become infinite. In this case, corresponding formally to the value $N = \infty$, existence is far from clear, and it is established thanks to explicit computations allowed by the tree structure\footnote{Soon after \cite{Al2000}, Benjamini and Schramm pointed out that this process is not well-defined on two-dimensional lattices such as $\ZZ^2$ (see also Remark~(i) after Theorem~1 in \cite{BT2001}).}. Furthermore, the process in \cite{Al2000} is shown to display an exact form of self-organized criticality, a fascinating phenomenon studied extensively in statistical physics (see e.g. \cite{Ba1996, Je1998, Pr2012}, and the references therein). Here, the (near-) critical regime of Bernoulli percolation arises ``spontaneously'', which was also observed, for example, in the case of invasion percolation \cite{WW1983, DSV2009, GPS2018b}.

Percolation theory, which was initiated by Broadbent and Hammersley \cite{BH1957} in 1957, provides key tools to analyze such processes, where connectivity plays a central role. For example, it will allow us to understand how far fires can spread, through large-scale connections in the forest. More specifically, we make use of Bernoulli site percolation on $G$, where vertices are independently occupied or vacant, with respective probabilities $p$ and $1-p$, for some parameter $p \in [0,1]$. In the following, the graph $G$ is either the full (two-dimensional) triangular lattice $\TT = (V_{\TT}, E_{\TT})$, embedded in a natural way into $\RR^2$ (with a vertex at the origin $0$, each face being a triangle with sides of length $1$, one of which parallel to the $x$-axis), or subgraphs of it. As we explain briefly in Section~\ref{sec:outline}, we restrict to $\TT$ because it is the planar lattice where the most sophisticated properties, used extensively in our proofs, are known rigorously. We will mostly consider balls around $0$, i.e. the graph with set of vertices $\Ball_n := [-n,n]^2 \cap V_{\TT}$, and set of edges induced by $E_{\TT}$. In the case of Bernoulli site percolation on $\TT$, there exists almost surely (a.s.) an infinite occupied cluster \emph{if and only if} $p > p_c^{\textrm{site}}(\TT)$, and it is a celebrated result \cite{Ke1980} that $p_c^{\textrm{site}}(\TT) = \frac{1}{2}$ (see also Section~3.4 in \cite{Ke1982}). For the birth process discussed in the beginning of this section, this means that a transition occurs at the critical time $t_c := \log 2$, such that $p(t_c) = p_c^{\textrm{site}}(\TT)$: for all $t \leq t_c$, there is a.s. no infinite occupied cluster, while for all $t > t_c$, there exists a.s. such a cluster (which, moreover, is known to be unique).

Infinite clusters are obviously prevented from forming in $N$-frozen percolation. For forest fires, it is also known that such clusters do not arise (a.s.), although this is much less obvious \cite{Du2006} (but even if they did, they would have to disappear instantaneously since $\zeta > 0$). These two processes have a very peculiar behavior as time crosses the threshold $t_c$, as we explain in Section~\ref{sec:background}. In forest fires for instance, large-scale connections start to appear, helping fires to spread. Hence, large burnt areas may be created, which then act as ``fire lines'' hindering the emergence of new large clusters. The main goal of this paper is to follow closely the emergence of frozen / burnt clusters surrounding a given vertex of the graph: the successive times at which they appear, and how big they are (in terms of diameter and volume). In both cases, we highlight the emergence of what we call \emph{near-critical avalanches}, which, although described by different sets of exponents, are of a similar nature.

\subsection{Background} \label{sec:background}

We now discuss earlier works on frozen percolation and forest fires, in particular the existence of exceptional scales, which are a clear indication of the challenges ahead when trying to study the processes rigorously. These scales are a manifest symptom of the ``non-monotone'' nature of the underlying mechanisms: for example, increasing the number of trees in a forest makes it more connected, but potential fires may then destroy larger regions, making the forest less connected eventually.

\subsubsection{Frozen percolation: exceptional scales, deconcentration}

The volume-frozen percolation process considered here was first studied in \cite{BN2015}, where the existence of a remarkable sequence of functions $m_k(N)$, $k \geq 1$, was uncovered, with $\sqrt{N} \asymp m_1(N) \ll m_2(N) \ll m_3(N) \ll \ldots$ as $N \to \infty$. These were called \emph{exceptional scales} because of the following dichotomy as $N \to \infty$, established in Theorems~1 and 2 of \cite{BN2015}\footnote{In this paper, time is indexed by $[0,\infty)$ in order to achieve a unified treatment with forest fires. We want to underline the fact that earlier works on frozen percolation often use time indexed by $[0,1]$ instead, following the notations of Aldous \cite{Al2000}. In particular, we rephrase the results of \cite{BN2015} and \cite{BKN2015} accordingly in this section.}. Let $M(N)$ be a function of $N$, and consider $N$-frozen percolation in the ball $\Ball_{M(N)}$.
\begin{itemize}
\item[(i)] If $M(N) \asymp m_k(N)$ for some $k \geq 2$, then
$$\liminf_{N \to \infty} \PP_N^{(\Ball_{M(N)})} \big( 0 \text{ is frozen at time } \infty \big) > 0.$$

\item[(ii)] If $m_k(N) \ll M(N) \ll m_{k+1}(N)$ for some $k \geq 1$, then
$$\lim_{N \to \infty} \PP_N^{(\Ball_{M(N)})} \big( 0 \text{ is frozen at time } \infty \big) = 0.$$
\end{itemize}
In the first case, with high probability (w.h.p.) the cluster of $0$ in the final configuration is either macroscopic (with a volume of order $N$, frozen or non-frozen), or microscopic (containing $O(1)$ vertices). In the second situation, $0$ lies in a ``mesoscopic'' cluster w.h.p., with volume $\gg 1$ but $\ll N$ (so non-frozen, in particular).

For each $k \geq 1$, $m_{k+1}$ is obtained from $m_k$ via the relation
\begin{equation} \label{eq:excep_scales_FP}
m_{k+1}(N) = \sqrt{\frac{N}{\pi_1 \big( m_k(N) \big)}}
\end{equation}
(up to multiplicative constants), where $\pi_1(n)$ denotes the one-arm probability for Bernoulli percolation at criticality, i.e. at $p = p_c$, which is the probability that the occupied cluster of $0$ has a radius at least $n$ (see Section \ref{sec:notations} for precise definition of this and other quantities). Hidden behind \eqref{eq:excep_scales_FP} is a transformation, defined roughly as
\begin{equation} \label{eq:map_FP}
\next_N: R \mapsto t \in (t_c, \infty) \text{ s.t. } R^2 \theta(t) = N
\end{equation}
(see \eqref{eq:def_next_FP}). It is used to predict when the first (macroscopic) freezing event occurs in the ball $\Ball_R$. In this definition, $\theta(t)$ denotes the probability for $0$ to lie in an infinite cluster at time $t$, which we know is $> 0$ \emph{iff} $t > t_c$. This map happens to induce an approximate ``fixed point'', in a sense made precise in Section~\ref{sec:successive}. We denote it by $m_{\infty}(N)$, and one can prove that it satisfies
$$m_{\infty}(N)^2 \pi_1(m_{\infty}(N)) \asymp N.$$
In this relation, the l.h.s. gives the order of magnitude, for Bernoulli percolation at criticality, of the volume of the largest cluster in $\Ball_{m_{\infty}(N)}$. In the case of the triangular lattice, it is known that $\pi_1(n) = n^{-\frac{5}{48} + o(1)}$ as $n \to \infty$, so that $m_{\infty}(N) = N^{\frac{48}{91} + o(1)}$ as $N \to \infty$. We often think of $m_{\infty}$ as a ``characteristic length'' naturally associated with frozen percolation.

The dichotomy above tells only part of the story, since for every $k \geq 1$, $m_k(N) \ll m_{\infty}(N)$ as $N \to \infty$. More work is required for boxes with a bigger side length, and for the full-lattice process. In \cite{BKN2015} it was shown in particular (Theorem~1.1) that for $N$-frozen percolation on $\TT$, the density of frozen vertices vanishes as $N \to \infty$:
$$\lim_{N \to \infty} \PP_N^{(\TT)} \big( 0 \text{ is frozen at time } \infty \big) = 0.$$
One can also extract from the proofs that a typical point lies in a mesoscopic cluster, and that the number of frozen clusters surrounding $0$ tends to $\infty$ in probability. In addition, the same conclusions hold for the process in $\Ball_{M(N)}$, if $M(N) \gg m_k(N)$ for all $k \geq 1$ (see Theorem~1.2 of \cite{BKN2015}). This is true in particular if $M(N)$ is at least of order $m_{\infty}(N)$ (or even $\frac{m_{\infty}(N)}{\upsilon(N)}$, as long as $\upsilon(N) = o(N^{\ve})$ for each $\ve > 0$). For these results, the aforementioned map \eqref{eq:map_FP} plays a central role, allowing one to follow the dynamics around a given vertex. The condition $M(N) \gg m_k(N)$ is used in \cite{BKN2015} to ensure that at least $k$ clusters encircling $0$ freeze successively, closer and closer to it, yielding a ``deconcentration'' property for the cluster of $0$ (its diameter) as the number of steps $k \to \infty$.

\subsubsection{Forest fires: exceptional scales, percolation with impurities} \label{sec:background_FF}

The FFWoR process mentioned in Section~\ref{sec:intro_intro} was shown in \cite{BN2018} to display a similar dichotomy as the ignition rate $\zeta \searrow 0$, for a sequence of exceptional scales\footnote{For the sake of readability, we often use the same notations as for frozen percolation, since it will always be clear from the context which process we are referring to.} $m_k(\zeta)$, $k \geq 1$, satisfying $\frac{1}{\sqrt{\zeta}} \asymp m_1(\zeta) \ll m_2(\zeta) \ll m_3(\zeta) \ll \ldots$ as $\zeta \searrow 0$. However, we get for this process a different relation between $m_k$ and $m_{k+1}$, namely
\begin{equation} \label{eq:excep_scales_FF}
m_{k+1}(\zeta) = \frac{m_k(\zeta)}{\sqrt{\zeta}} \sqrt{\frac{\pi_4 \big( m_k(\zeta) \big)}{\pi_1 \big( m_k(\zeta) \big)}}
\end{equation}
(again, up to multiplicative constants). In this formula, $\pi_4(n)$ ($= n^{-\frac{5}{4} + o(1)}$ as $n \to \infty$) is the so-called four-arm probability for critical percolation (see Section \ref{sec:notations}): it describes for example the probability that there exist two disjoint occupied clusters, each containing a neighbor of $0$ and a vertex outside $\Ball_n$.

Consider the FFWoR process with rate $\zeta$ in $\Ball_{M(\zeta)}$, for a function $M(\zeta)$. The following holds true (see \cite{BN2018}, Theorem~1.3).
\begin{itemize}
\item[(i)] If $M(\zeta) \asymp m_k(\zeta)$ for some $k \geq 1$, then for all $t > t_c$,
$$\liminf_{\zeta \searrow 0} \PP_\zeta^{(\Ball_{M(\zeta)})} \big( 0 \text{ burns before time } t \big) > 0.$$

\item[(ii)] If $m_k(\zeta) \ll M(\zeta) \ll m_{k+1}(\zeta)$ for some $k \geq 1$, then for all $t > 0$,
$$\lim_{\zeta \searrow 0} \PP_\zeta^{(\Ball_{M(\zeta)})} \big( 0 \text{ burns before time } t \big) = 0.$$
\end{itemize}
In the first case, $0$ lies in a cluster with volume of order $\frac{1}{\zeta}$ or $1$, while in the second one, the cluster of $0$ has a volume $\gg 1$ and $\ll \frac{1}{\zeta}$ (w.h.p.).

Again, \eqref{eq:excep_scales_FF} comes from a map $\next_{\zeta}: R \mapsto t \in (t_c, \infty)$, where now $t$ satisfies $R^2 \theta(t) (t - t_c) \asymp \frac{1}{\zeta}$ (see \eqref{eq:def_next_FF}), and gives the approximate time when the first large cluster burns in $\Ball_R$. It gives rise to a fixed point $m_{\infty}(\zeta) = \zeta^{- \frac{48}{55} + o(1)}$, which can also be guessed heuristically, and can be interpreted as a characteristic length for the FFWoR process.

However, proofs in this case are more complicated, due to fires occurring all over the lattice. Even though only tiny clusters burn in the beginning, larger and larger ones get ignited as time approaches $t_c$, which are far from microscopic. This motivated the introduction in \cite{BN2018} of a class of percolation models with \emph{heavy-tailed impurities}. These models were used to understand the cumulative effect of fires on the connectivity of the lattice, all the way into the near-critical window. They are defined as follows, for a parameter $m \to \infty$, and two sequences $(\pi^{(m)})$ and $(\rho^{(m)})$. First, for each vertex $v$, there is an impurity $H_v$ centered on it with probability $\pi^{(m)}$. In this case, the radius $r_v$ (with respect to the $L^{\infty}$ norm) of $H_v$ is drawn randomly according to the distribution $\rho^{(m)}$ on $[0,+\infty)$, i.e. all vertices in $H_v = v + \Ball_{r_v}$ are declared vacant. The sequences $\pi^{(m)}$ and $\rho^{(m)}$ are essentially chosen in the following way: for some constants $c_1, c_2, c_3 > 0$, and $\alpha < \beta < 2$,
\begin{equation} \label{eq:ass_impurities}
\pi^{(m)} = c_1 m^{- \beta}, \quad \text{and for all } r \geq 1, \:\: \rho^{(m)} \big( [r,+\infty) \big) = c_2 r^{\alpha-2} e^{-c_3 \frac{r}{m}}.
\end{equation}
Finally, in the complement of the union of all the impurities, we consider Bernoulli percolation with a parameter $p$.

The particular case where for all $m$, $\rho^{(m)} = \delta_0$ (the Dirac mass at $0$, i.e. all impurities consist of a single site) simply corresponds to Bernoulli percolation with a slightly lower parameter. In this case, it is known that if $\beta < \frac{3}{4}$ and $p = p_c^{\textrm{site}}(\TT)$, then the impurities make the configuration subcritical, even though they have a vanishing density (for example, some clusters in $\Ball_m$ have a diameter of order $m$ in critical Bernoulli percolation, but all of them have a diameter $\ll m$ once the impurities are added, w.h.p.). In our situation, the resulting configuration will still be approximately critical, thanks to an additional hypothesis, but as observed in \cite{BN2018}, this comes from a subtle balance: we have to show that the impurities do not substantially hinder the formation of occupied connections.

\subsection{Statement of results: near-critical avalanches} \label{sec:statements}

We now describe our results, first for frozen percolation, and then forest fires. For all $t \geq 0$, we denote by $\calF_t$ the set of frozen / burnt clusters (depending on the process) surrounding the origin $0$ at time $t$ (that is, clusters $\cluster$ such that either $0 \in \cluster$ or the connected component of $0$ in $V_{\TT} \setminus \cluster$ is finite). We are interested in the set $\calF := \calF_{\infty}$ of such clusters in the final configuration in the FP and FFWoR processes\footnote{For the FFWoR process, we count two adjacent clusters as distinct if they burn at different times (there is no ambiguity for frozen percolation since frozen clusters cannot touch, due to the layer of vacant sites along their outer boundaries).} (note that in these two cases, $\calF_t$ is clearly increasing in $t$ as a set). In both cases, we study avalanches at scales of order $m_{\infty}$ or below.

In particular, we are able to determine the precise order of magnitude of $|\mathcal{F}|$, which is, respectively, $\log \log N$ or $\log \log \frac{1}{\zeta}$. Additionally, it is even possible to derive the exact constants in front, and prove limit theorems, which was surprising to us. In our opinion, this is an unexpected and remarkable phenomenon, which seems to be specific to such processes.

\subsubsection{Frozen percolation}

Let
$$\avalnb^{\textrm{FP}} := \frac{1}{\log \big( \frac{96}{5} \big)}.$$
Our main result for frozen percolation is the following.

\begin{theorem} \label{thm:mainFP}
Let $K > 0$, and consider $N$-frozen percolation in $\Ball_{K m_\infty(N)}$. For all $\ve > 0$,
$$\PP_N^{(\Ball_{K m_\infty(N)})} \bigg( \frac{|\calF|}{\log \log N} \in (\avalnb^{\textnormal{FP}} - \ve, \avalnb^{\textnormal{FP}} + \ve) \bigg) \stackrel[N \to \infty]{}{\longrightarrow} 1.$$
\end{theorem}

In the beginning of the proof of Theorem~\ref{thm:mainFP}, we use a result from Section~7 of \cite{BKN2015}, where it is explained how to compare frozen percolation on the whole lattice $\TT$ to the process in finite, large enough as a function of $N$, domains (with a radius, seen from $0$, a little smaller than $m_\infty(N)$). More precisely, we observe that in this result from \cite{BKN2015}, we can replace the full-lattice process by frozen percolation in $\Ball_{K m_\infty(N)}$, for any $K > 0$. For similar reasons, the following corollary holds, where for each $t \geq 0$, $\calF_t^{(\Ball_{K m_\infty(N)})}$ denotes the set of clusters in $\calF_t$ which are entirely contained in $\Ball_{K m_\infty(N)}$.

\begin{corollary} \label{cor:main}
Let $K > 0$. For all $\ve > 0$,
$$\PP_N^{(\TT)} \bigg( \frac{|\calF^{(\Ball_{K m_\infty(N)})}|}{\log \log N} \in (\avalnb^{\textnormal{FP}} - \ve, \avalnb^{\textnormal{FP}} + \ve) \bigg) \stackrel[N \to \infty]{}{\longrightarrow} 1.$$
\end{corollary}

Our proofs proceed by iterating a construction similar to the ones in \cite{BN2015}, for the existence of exceptional scales. However, all the results in \cite{BN2015} involve only a finite number of iterations, which means that it is not problematic to ``lose a little bit'' at every step. At first sight, the proofs seem to be really too crude when the number of successive freezings tends to $\infty$. One of our main contributions is to show that it is possible to quantify the error made at every step in order to control all freezings taking place starting from the scale $m_{\infty}$, at which the ``avalanche'' is expected to begin, even though the number of such freezings tends to $\infty$ as $N \to \infty$.

It requires estimating very precisely the ``uncertainty'' produced at every step of the iteration. A first key observation is that the number of such steps still grows very slowly, like $ \log \log N$. In particular, it is smaller than the number of disjoint occupied circuits (or clusters) in $\Ball_{K m_\infty(N)}$ surrounding the origin in critical Bernoulli percolation, which can be shown to be of order $\log N$, from the Russo-Seymour-Welsh lemma (see \eqref{eq:RSW}) -- or in the whole lattice $\TT$, the number of such circuits with volume $< N$.

On the other hand, it is possible to prove that the number of disjoint frozen \emph{circuits} grows at least as a power law in $N$, so much faster than $\log N$.

\begin{proposition} \label{prop:circuits}
Let $C_{\calF}$ denote the maximal number of disjoint frozen circuits surrounding $0$, in the final configuration. There exists a universal constant $\xi > 0$ such that: for all $K > 0$,
$$\PP_N^{(\Ball_{K m_\infty(N)})} \big( C_{\calF} \geq N^{\xi} \big) \stackrel[N \to \infty]{}{\longrightarrow} 1.$$
Moreover, the same is true for $N$-frozen percolation on the full lattice $\TT$.
\end{proposition}

As a side contribution, we explain in Section~\ref{sec:scale_minf_FP} how to handle the process on scales of order $m_{\infty}(N)$ in the case when the so-called \emph{boundary rules} are modified: if vertices along the outer boundary of a frozen cluster are no longer kept vacant forever, and are allowed to become occupied, and possibly freeze, at later times. We believe that this observation is of independent interest, and may be useful to study related processes such as diameter-frozen percolation (see Section~\ref{sec:other_proc}). Indeed, the macroscopic behavior of this latter process is known to be significantly affected by the boundary rules \cite{BN2017}, leading to the creation of highly supercritical regions, which did not exist in the original process. Furthermore, this idea is then used in Section~\ref{sec:scale_minf_FF} to analyze the FFWoR process around scale $m_{\infty}(\zeta)$.

\subsubsection{Forest fires}

The FFWoR process displays a similar behavior, but with a different constant in front:
$$\avalnb^{\textrm{FF}} := \frac{1}{\log \big( \frac{96}{41} \big)}.$$
Our main result in this case is the following analog of Theorem~\ref{thm:mainFP}.

\begin{theorem} \label{thm:mainFF}
Let $K > 0$, and consider the FFWoR process with ignition rate $\zeta$ in $\Ball_{K m_\infty(\zeta)}$. For all $\ve > 0$,
$$\PP_\zeta^{(\Ball_{K m_\infty(\zeta)})} \bigg( \frac{|\calF|}{\log \log \frac{1}{\zeta}} \in (\avalnb^{\textnormal{FF}} - \ve, \avalnb^{\textnormal{FF}} + \ve) \bigg) \stackrel[\zeta \searrow 0]{}{\longrightarrow} 1.$$
\end{theorem}

Analyzing avalanches for this process requires significantly more work than for frozen percolation. Again, the proof proceeds by quantifying very accurately the error created at every step of an iterative procedure, and uses that the number of such steps grows like $\log \log \frac{1}{\zeta}$. However, much more convoluted constructions are necessary, in order not too lose too much along the way. Moreover, additional technical difficulties arise, related to the process with impurities from \cite{BN2018}, discussed in Section~\ref{sec:background_FF}. This process is used in an instrumental way in the iteration, and we will explain in Section~\ref{sec:perc_impurities} that some non-trivial modifications are required in our setting, especially to estimate the volume of the largest cluster in a box. The stochastic minorant derived in \cite{BN2018} for forest fires is only suitable on scales below the $k$th exceptional scale $m_k(\zeta)$, for each fixed $k \geq 1$, while here we need a good approximation of the FFWoR process up to the characteristic scale $m_{\infty}(\zeta) = \zeta^{- \frac{48}{55} + o(1)}$, leading to study configurations where fires have destroyed a non-negligible fraction of the lattice.

More specifically, we replace the two assumptions from \cite{BN2018}, mentioned (roughly) in \eqref{eq:ass_impurities}, by a single hypothesis on the product $\pi^{(m)} \cdot \rho^{(m)}$, see Assumption~\ref{ass:impurities} in Section~\ref{sec:perc_impurities}. This allows us to extend stability results from that paper up to a positive density of impurities, while in the original setting this density decreased at least as a power law in $m$. In addition, the comparison with forest fires uses the (easy, but crucial) observation that the ignition rate $\zeta$ can be expressed in terms of $m_{\infty}$ without logarithmic corrections: as shown in \eqref{eq:estim_zeta}, we have
\begin{equation} \label{eq:zeta_minf}
\zeta \asymp \frac{\pi_4(m_{\infty})}{\pi_1(m_{\infty})} \quad \text{as $\zeta \searrow 0$.}
\end{equation}

Contrary to the case of frozen percolation, no result seems to be readily available in the literature to connect full-plane forest fires with the corresponding processes in suitable finite domains, i.e. with a radius a bit below $m_{\infty}$. We thus have to develop an analogous approach to the one in Section~7 of \cite{BKN2015}, as a first step in the proof of Theorem~\ref{thm:mainFF}. This allows us to move from the scale $K m_\infty$ down to a somewhat smaller scale, and start to study the avalanche. Because of ignitions taking place all over the lattice and at all times, we were not able to construct exploration procedures which are as neat as in \cite{BKN2015}, producing ``stopping sets'' (which are a natural analog of stopping times). On the other hand, the fact that ignitions follow a Poisson process provides some valuable spatial independence, which allows us to simplify significantly some of the arguments. Altogether, these reasonings are quite involved. We believe that they will be useful in the future when studying forest fire processes on the full lattice, and that they thus represent a noteworthy contribution of this paper.

Here we decided to state our main result in terms of the number of burnt clusters surrounding a given vertex, but the proof can be used to derive much more information about the near-critical behavior of the FFWoR process. As a by-product, we can obtain for example that as $\zeta \searrow 0$, most clusters burning around $0$ have a volume of order, roughly,
\begin{equation} \label{eq:def_Vinfty}
V_{\infty}(\zeta) := m_{\infty}(\zeta)^2 \pi_1(m_{\infty}(\zeta)).
\end{equation}

\begin{proposition} \label{prop:large_burnt_clusters}
For all $\ve > 0$, there exists $\xi = \xi(\ve) \in (0,\frac{1}{2})$ such that the following holds. For all $K > 0$,
$$\PP_\zeta^{(\Ball_{K m_\infty(\zeta)})} \bigg( \bigg| \bigg\{ \cluster \in \calF \: : \: \frac{V_{\infty}(\zeta)}{e^{( \log \frac{1}{\zeta} )^{1 - \xi}}} \leq |\cluster| \leq \frac{V_{\infty}(\zeta)}{e^{( \log \frac{1}{\zeta} )^{\xi}}} \bigg\} \bigg| \geq \big( \avalnb^{\textnormal{FF}} - \ve \big) \log \log \frac{1}{\zeta} \bigg) \stackrel[\zeta \searrow 0]{}{\longrightarrow} 1.$$
\end{proposition}

One has $V_{\infty}(\zeta) = \zeta^{- \frac{91}{55} + o(1)}$ as $\zeta \searrow 0$, so the clusters appearing in this result have, in particular, a volume $\gg \frac{1}{\zeta}$.

\begin{remark}
There does not seem to be a natural counterpart of this observation for frozen percolation, since all frozen clusters have a volume of order $N$ (more precisely, between $N$ and $3N-2$, any such cluster being formed from at most three clusters, each with a volume $\leq N-1$). On the other hand, a result similar to Proposition~\ref{prop:circuits} holds for the FFWoR process with rate $\zeta$ in $\Ball_{K m_\infty(\zeta)}$, for all $K > 0$, although we do not include a detailed proof in this paper for the sake of conciseness. Analogs of Corollary~\ref{cor:main} and Proposition~\ref{prop:circuits} for the full-lattice FFWoR process can also be derived without any additional difficulty. But strictly speaking, this would require to check first the existence of the process in this case, which can be done thanks to the reasonings in \cite{Du2006}.
\end{remark}

\subsubsection{Outline of proofs} \label{sec:outline}

First, we want to mention that our results use, as a key ingredient, a detailed understanding of near-critical site percolation on the triangular lattice, coming in large part from the groundbreaking work \cite{Ke1987}. They rely indirectly on the SLE (Schramm-Loewner Evolution) technology \cite{LSW_2001a, LSW_2001b} and the conformal invariance property of critical percolation in the scaling limit \cite{Sm2001}: these works provide the exact value of several critical exponents \cite{LSW2002, SW2001} (see also \cite{We2009}), and this is a crucial input in our proofs. Indeed, this is what enables us to obtain sharp limit theorems for $|\calF|$ for both frozen percolation and the FFWoR process.

As a matter of fact, one could hope for even better asymptotic estimates on $|\calF|$. For frozen percolation for instance, it is tempting to guess that $(|\calF| - \avalnb^{\textrm{FP}} \log \log N)$ is tight as $N \to \infty$ (and similarly for the FFWoR process). However, obtaining such strengthened results from our strategy of proof would require improved estimates for critical Bernoulli percolation. More precisely, the potential existence of logarithmic corrections in arm probabilities, as in \eqref{eq:arm_exp}, would need to be ruled out, which, to the best of our knowledge, is still an open problem at the moment (it is not clear how to obtain it from the fact that $SLE_6$ processes describe boundaries of clusters in the scaling limit). Actually, such logarithmic terms turn out to be a substantial source of technical difficulties in our proofs. A key idea to circumvent this issue is to compare, in a systematic way, all scales to the characteristic scale $m_\infty$ (and in the case of the FFWoR process, use the expression \eqref{eq:zeta_minf} of $\zeta$ in terms of $m_\infty$).

The proofs of our main results, Theorems~\ref{thm:mainFP} and \ref{thm:mainFF}, can each be decomposed into three consecutive steps, roughly. Let $K > 0$, and consider frozen percolation or the FFWoR process in $\Ball_{K m_\infty}$. In a first step, we show that the number of frozen / burnt clusters surrounding $0$ with a diameter of order $m_\infty$ is essentially tight. This allows us to move to scales slightly below $m_\infty$ (typically $m_\infty / (\log *)^{\alpha}$, for some exponent $\alpha = \alpha(K)$ which must be taken sufficiently small, where $* = N$ or $\frac{1}{\zeta}$ depending on the model). In a second step which is the core of the proof, we analyze iteratively the successive frozen / burnt clusters around $0$. Finally, in a third (and easy) step, we explain how to terminate the iteration scheme. Informally speaking, there is ``too much randomness'' at scale $m_\infty$, and this is why we have to start the second stage on smaller scales (this ensures that a property of ``separation of scales'', between successive frozen / burnt clusters, remains valid during the whole procedure). Our arguments to control the number of clusters around scale $m_\infty$ are quite crude: they proceed by comparing the processes to critical percolation, thanks to ad-hoc near-critical parameter scales\footnote{We want to emphasize that we were able to obtain quite strong limit theorems only because such an intermediate scale $m_\infty / (\log *)^{\alpha}$ exists, which is both small enough to start the iteration, and large enough so that not too many near-critical circuits can form between this scale and the scale $K m_\infty$. As it can be seen from the proofs, there is a little freedom in the choice of this scale, but not much.}. This is good enough for our purpose, but this would constitute another roadblock to bypass in order to obtain the tightness property mentioned above.

For frozen percolation on other two-dimensional lattices, such as the square lattice $\ZZ^2$, or variants defined in terms of bond percolation, it is possible to obtain partial results, as we point out in Remark~\ref{rem:other_lattices3}. In particular, we can still prove that $|\calF|$ grows like $\log \log N$, even though an asymptotic result as precise as Theorem \ref{thm:mainFP} seems to be out of reach. The situation is more problematic with forest fire processes, since even more detailed knowledge about near-critical percolation is required in this case, in order to derive stability results for the process with impurities discussed earlier (i.e. to show that the burnt regions do not disconnect too much the lattice, even as $t$ approaches $t_c$). For instance, the inequality $\alpha_2 \geq \alpha_4 - 1$ between the so-called two- and four-arm exponents (see \eqref{eq:arm_exp} and below) is used.

Our methods could also be used to study the joint avalanches around a finite number of vertices, for both frozen percolation and the FFWoR process. The ``branching structure'' of such avalanches depends on how the pair-wise distances between the vertices compare to $m_{\infty}$. In the present paper, we do not attempt to describe the processes on scales significantly above $m_{\infty}$, even though we expect, for each of these models, that clusters around $0$ with a diameter much larger than $m_{\infty}$ are very unlikely to emerge. In particular, we believe that the conclusion of Corollary~\ref{cor:main} remains true for \emph{all} frozen clusters surrounding $0$, i.e. with $\calF$ ($= \calF^{(\TT)}$) instead of $\calF^{(\Ball_{K m_\infty(N)})}$, and an analogous result should hold for the FFWoR process. Analyzing the dynamics of the processes on such scales requires a different strategy, and we plan to tackle this issue in a future work.

\subsection{Discussion: extensions and related works}

\subsubsection{Processes with recovery}

Our results should also give insight into processes with recovery. In two dimensions, the behavior near $t_c$ of such processes was studied in particular in \cite{BB2004} and \cite{BB2006}. There, an important question about a process called ``self-destructive percolation'' was asked. Roughly speaking, the authors conjectured that after a macroscopic fire, a uniformly positive time is needed for the forest to recover, i.e. for new large-scale connections to emerge. Conditionally on this conjecture, several remarkable consequences were established in these papers. But the question itself remained unanswered for almost ten years, and was finally settled (affirmatively) in \cite{KMS2015}.

Based on the main result in \cite{KMS2015}, Theorem~4, we can expect the FFWR process to display the same near-critical avalanches as the FFWoR process (without recovery) studied in the present paper. Moreover, such a result should still hold, approximately, up to a supercritical time $t_c + \delta$, where $\delta > 0$ is a universal constant. But in order to establish this, an analog of Theorem~4 of \cite{KMS2015} would be needed, incorporating the Poisson ignitions all over the lattice, see the discussion in Section~8 of \cite{BN2018}. In addition, some difficulties arise in our analysis around scale $m_\infty$, as we explain in Remark~\ref{rem:FFWR_minf}. On the other hand, we believe that our reasonings for frozen percolation (in particular the proof of Theorem~\ref{thm:mainFP}) can be adapted to a frozen percolation model ``with recovery'', where frozen sites are allowed to become occupied again (in other words, a forest fire process where clusters burn when they reach a volume $N$). For this process, the results of \cite{KMS2015} should be more or less directly applicable.

\subsubsection{Other processes} \label{sec:other_proc}

Still on two-dimensional lattices, a variant of frozen percolation was introduced in \cite{BLN2012}, where connected components stop growing when their diameter (for the $L^{\infty}$ norm) reaches $N \geq 1$, the parameter of the model. This diameter-frozen version was then further studied in \cite{Ki2015}, yielding in particular the following properties for the full-lattice process. For any given $K > 1$, only a finite number of frozen clusters intersect $\Ball_{KN}$ (the probability that there are at most $k$ of them tends to $1$ as $k \to \infty$, uniformly in $N$), and they all freeze in a near-critical window around $t_c$, of width $N^{-\frac{3}{4} + o(1)}$. All these clusters have a vanishing density as $N \to \infty$, and w.h.p., $0$ belongs to a macroscopic non-frozen cluster (i.e. with a diameter smaller than $N$, but of the same order of magnitude).

Frozen percolation and forest fires have also been considered on various other graphs. As mentioned earlier, rather explicit quantitative results can be derived for frozen percolation when the graph $G$ is a tree \cite{Al2000}, and this special case was further studied in \cite{BKN2012} and \cite{RST2019}, in particular. Related processes have also been analyzed on the one-dimensional lattice (see e.g. \cite{BF2013}), and on the complete graph (see \cite{RT2009, Ra2009}). The reader can consult Section~1.7 of \cite{RST2019} for an extensive list of references, containing a brief summary of each paper.


\subsection{Organization of the paper}

In Section \ref{sec:prelim}, we first describe the setting and set notation for site percolation in two dimensions. We then state results about the critical and near-critical behavior of this process, which are later used repeatedly in our proofs, and we present the above-mentioned percolation model with ``large'' impurities. In Section \ref{sec:FP_FF}, we define precisely the processes under consideration, and we introduce the transformation which yields the successive freezing / burning times. We also state precisely the stochastic domination provided by percolation with impurities, which is key to analyze forest fire processes near the critical time. In Section \ref{sec:scale_minf}, we study the frozen / burnt clusters surrounding $0$ with a diameter of order $m_\infty$, for both frozen percolation and forest fires. To this end, we associate a ``near-critical'' parameter scale with each of the two processes. We then establish the results for frozen percolation and forest fires (stated in Section~\ref{sec:statements}) in Sections \ref{sec:proof_FP} and \ref{sec:proof_FF}, respectively. Finally, we provide proofs for some of the auxiliary results in Appendix.

\section{Preliminaries on near-critical 2D percolation} \label{sec:prelim}

In this preliminary section, we discuss the behavior of Bernoulli site percolation on the triangular lattice $\TT$. First, we define properly site percolation in Section~\ref{sec:notations}, and we set notation. We then collect in Section~\ref{sec:classical} classical results pertaining to its critical and near-critical regimes on $\TT$. Finally, Section~\ref{sec:perc_impurities} is specific to the study of forest fire processes. It is devoted to the percolation model with impurities, which is a useful stochastic minorant of forest fires. We define this model precisely, and explain how to extend the results of \cite{BN2018} to our setting.

\subsection{Setting and notations} \label{sec:notations}

In the whole paper, we consider the triangular lattice $\TT = (V_{\TT}, E_{\TT})$. This is the lattice having vertex set
$$V = V_{\TT} := \big\{ x + y e^{i \pi / 3} \in \CC \: : \: x, y \in \ZZ \big\},$$
and edge set $E = E_{\TT} := \{ \{v, v'\} \: : \: v, v' \in V \text{ with } |v - v'| = 1 \}$. A path of length $k$ ($k \geq 0$) is a sequence of vertices $v_0 \sim v_1 \sim \ldots \sim v_k$, where for $v, v' \in V$, $v \sim v'$ means that $\{v, v'\} \in E$ ($v$ and $v'$ are connected by an edge). Such a path is said to be (vertex-) self-avoiding if it does not use twice the same vertex: $v_i \neq v_j$ for all $i, j \in \{0, \ldots, k\}$ with $i \neq j$. A circuit is a path of length $k$, for some $k \geq 0$, which satisfies $v_k = v_0$ but is otherwise self-avoiding (i.e. the vertices $v_0, \ldots, v_{k-1}$ are distinct).

Consider a subset of vertices $A \subseteq V$. The number of vertices that it contains is denoted by $|A|$, and called volume of $A$. The outer and inner vertex boundaries of $A$ are defined, respectively, as
$$\dout A := \{ v \in V \setminus A \: : \: \exists v' \in A \text{ with } v' \sim v \}$$
and $\din A := \dout(V \setminus A)$ ($\subseteq A$), while the edge boundary of $A$ is
$$\de A := \{ \{v,v'\} \in E \: : \: v \in A \text{ and } v' \in V \setminus A\}.$$
In the particular case when $A$ is finite, we also consider its ``external'' outer and inner boundaries: $\dout_{\infty} A := (\dout A) \cap (V \setminus A)_{\infty}$ and $\din_{\infty} A := \dout (V \setminus A)_{\infty}$, where $(V \setminus A)_{\infty}$ is the unique connected component of $V \setminus A$ which is infinite. The external edge boundary $\de_{\infty} A$ is defined correspondingly.

Bernoulli site percolation on the triangular lattice is obtained by tossing a biased coin for each vertex $v \in V$: for some given percolation parameter $p \in [0,1]$, $v$ is occupied with probability $p$ and vacant with probability $1-p$, independently of the other vertices. The corresponding product probability measure (on configurations $(\omega_v)_{v \in V} \in \{0,1\}^V =: \Omega$) is denoted by $\PP_p$.

Two vertices $v, v' \in V$ are connected (denoted by $v \lra v'$) if there exists an occupied path of length $k$ from $v$ to $v'$, for some $k \geq 0$, i.e. a path containing only occupied sites (in particular, $v$ and $v'$ have to be occupied). For a vertex $v \in V$, we can consider the maximal occupied connected component $\cluster(v)$ containing $v$, that we call (occupied) cluster of $v$ (if $v$ is vacant, we simply set $\cluster(v) = \emptyset$). The event that $|\cluster(v)| = \infty$, i.e. $v$ lies in an infinite cluster, is denoted by $v \lra \infty$, and we write $\theta(p) := \PP_p(0 \lra \infty)$. If $A, B \subseteq V$, we use the notation $A \lra B$ for the existence of $v \in A$ and $v' \in B$ such that $v \lra v'$. Similarly, $A \lra \infty$ means that $v \lra \infty$ for some $v \in A$.

A phase transition occurs for Bernoulli site percolation on $\TT$ at the percolation threshold $p_c = p_c^{\textrm{site}}(\TT) = \frac{1}{2}$ \cite{Ke1980}. More precisely, for each $p \leq p_c = \frac{1}{2}$, there is almost surely no infinite cluster, while for $p > \frac{1}{2}$, there is almost surely a unique such cluster. For more background about percolation theory, we refer the reader to the classical references \cite{Ke1982, Gr1999}.

The existence of a horizontal crossing in a rectangle $R = [x_1,x_2] \times [y_1,y_2]$ ($x_1 < x_2$, $y_1 < y_2$) (i.e. an occupied path connecting two vertices adjacent to the left and right sides, respectively) is denoted by $\Ch(R)$, and the existence of a vertical crossing by $\Cv(R)$. We also use the notation $\Ch^*(R)$ and $\Cv^*(R)$ for the corresponding events with paths of vacant vertices. We denote by $\Ball_n := [-n,n]^2$ the ball of radius $n \geq 0$ around $0$ for the $L^{\infty}$ norm $\|.\| = \|.\|_{\infty}$, and for $0 \leq n_1 < n_2$, by $\Ann_{n_1,n_2} := \Ball_{n_2} \setminus \Ball_{n_1}$ the annulus centered at $0$ with radii $n_1$ and $n_2$. We sometimes write, for $z \in \CC$, $\Ball_n(z) := z + \Ball_n$ and $\Ann_{n_1, n_2}(z) := z + \Ann_{n_1,n_2}$.

For an annulus $A= \Ann_{n_1,n_2}(z)$, we denote by $\circuitevent(A)$ the existence of an occupied circuit in $A$ (and similarly $\circuitevent^*(A)$ for a vacant circuit). We also introduce the arm events $\arm_{\sigma}(A)$, for $k \geq 1$ and $\sigma \in \colorseq_k := \{o,v\}^k$ (where $o$ and $v$ stand for ``occupied'' and ``vacant'', respectively). The event $\arm_{\sigma}(A)$ corresponds to the existence of $k$ disjoint paths $(\gamma_i)_{1 \leq i \leq k}$ in $A$, in counter-clockwise order, each connecting $\dout \Ball_{n_1}$ and $\din \Ball_{n_2}$, with type prescribed by $\sigma_i$ (i.e. occupied or vacant). We let
\begin{equation}
\pi_{\sigma}(n_1,n_2) := \PP_{p_c}\big( \arm_{\sigma}(\Ann_{n_1,n_2}) \big),
\end{equation}
and $\pi_{\sigma}(n) := \pi_{\sigma}(1,n)$. Finally, we use the notation $\arm_1$, $\pi_1$ when $\sigma = (o)$, and $\arm_4$, $\pi_4$ when $\sigma = (ovov)$.

\subsection{Classical results} \label{sec:classical}

Our proofs rely heavily on precise properties of near-critical Bernoulli percolation in two dimensions, that we recall now. This detailed description of the phase transition comes, to a large extent, from the pioneering work \cite{Ke1987}.

First, the characteristic length $L$ is defined by:
\begin{equation} \label{eq:def_L}
\text{for $p < p_c = \frac{1}{2}$,} \quad L(p) := \min \big\{ n \geq 1 \: : \: \PP_p \big( \Cv( [0,2n] \times [0,n] ) \big) \leq 0.001 \big\},
\end{equation}
and $L(p) = L(1-p)$ for $p > p_c$. Since $L(p) \to \infty$ as $p \to p_c$, we also set $L(p_c) := \infty$.

The function $L$ is discontinuous and piece-wise constant, so we rather work with a regularized version $\tilde L$, defined as follows. First, we set $\tilde{L}(0) = \tilde{L}(1) = 0$, and $\tilde L(p) = L(p)$ at each point of discontinuity $p \in(0,p_c) \cup (p_c,1)$ of $L$. We then extend linearly $\tilde L$ to $[0,1] \setminus \{p_c\}$. The function $\tilde L$ has the additional property of being continuous on $[0,1] \setminus \{p_c\}$, as well as strictly increasing (resp. strictly decreasing) on $[0,p_c)$ (resp. $(p_c,1]$). In particular, it is a bijection from $[0,p_c)$ (resp. $(p_c,1]$) to $[0,\infty)$. From now on, we write $L$ instead of $\tilde L$ to simplify notation.

\begin{enumerate}[(i)]
	\item \emph{Russo-Seymour-Welsh bounds.} For all $K \geq 1$, there exists $\delta_4 = \delta_4(K) > 0$ such that: for all $p \in (0,1)$ and $n \leq K L(p)$,
	\begin{equation} \label{eq:RSW}
	\PP_p \big( \Ch( [0,4n] \times [0,n] ) \big) \geq \delta_4 \quad \text{and} \quad \PP_p \big( \Ch^*( [0,4n] \times [0,n] ) \big) \geq \delta_4.
	\end{equation}
	
	\item \emph{Exponential decay property.} For some universal constants $\kappa_1, \kappa_2 > 0$, we have: for all $p > p_c$ and $n \geq 1$,
	\begin{equation} \label{eq:exp_decay}
	\PP_p \big( \Ch( [0, 4n] \times [0,n] ) \big) \geq 1 - \kappa_1 e^{- \kappa_2 \frac{n}{L(p)}}
	\end{equation}
	(see Lemma 39 in \cite{No2008}). An analogous statement holds for $p < p_c$ and the existence of a horizontal vacant crossing.
	
	Using a standard construction (e.g. with a sequence of overlapping rectangles), it follows from \eqref{eq:exp_decay} that for some $\kappa'_1, \kappa'_2 > 0$: for all $p > p_c$ and $n \geq 1$,
	\begin{equation} \label{eq:connection_infinity}
	\PP_p \big( \dout \Ball_n \lra \infty \big) \geq 1 - \kappa'_1 e^{- \kappa'_2 \frac{n}{L(p)}}.
	\end{equation}
	
	\item \emph{Arm exponents at criticality.} For all $k \geq 1$ and $\sigma \in \colorseq_k$, there exists an ``arm exponent'' $\alpha_{\sigma} > 0$ such that
	\begin{equation} \label{eq:arm_exp}
	\pi_{\sigma}(k,n) = n^{- \alpha_{\sigma} + o(1)} \quad \text{as $n \to \infty$.}
	\end{equation}
	Moreover, the value of $\alpha_{\sigma}$ is known is the following cases.
	\begin{itemize}
		\item For $k=1$, $\alpha_{\sigma} = \frac{5}{48}$ \cite{LSW2002}.
		
		\item For all $k \geq 2$, and $\sigma \in \colorseq_k$ containing both types, $\alpha_{\sigma} = \frac{k^2-1}{12}$ \cite{SW2001}.
	\end{itemize}
	In our proofs we use the following, more ``uniform'', version (see Lemma 2.6 in \cite{BN2018}). For all $\ve > 0$, there exist $0 < C_1 < C_2$ (depending on $\sigma$ and $\ve$) such that: for all $1 \leq n_1 < n_2$,
	\begin{equation} \label{eq:uniform_arm_exp}
	C_1 \bigg( \frac{n_1}{n_2} \bigg)^{\alpha_{\sigma} + \ve} \leq \pi_{\sigma}(n_1, n_2) \leq C_2 \bigg( \frac{n_1}{n_2} \bigg)^{\alpha_{\sigma} - \ve}.
	\end{equation}
	
	\item \emph{Quasi-multiplicativity of arm events.} For all $k \geq 1$ and $\sigma \in \colorseq_k$, there exist $C_1, C_2 > 0$ (depending on $\sigma$) such that: for all $0 \leq n_1 < n_2 < n_3$,
	\begin{equation} \label{eq:quasi_mult}
	C_1 \pi_{\sigma}(n_1, n_3) \leq \pi_{\sigma}(n_1, n_2) \pi_{\sigma}(n_2, n_3) \leq C_2 \pi_{\sigma}(n_1, n_3)
	\end{equation}
	(see Proposition 17 in \cite{No2008}).
	
	\item \emph{Near-critical stability of arm events.} For all $k \geq 1$, $\sigma \in \colorseq_k$, and $K \geq 1$, there exist constants $C_1, C_2 > 0$ (depending on $\sigma$ and $K$) such that: for all $p \in (0,1)$ and $0 \leq n_1 < n_2 \leq K L(p)$,
	\begin{equation} \label{eq:near_critical_arm}
	C_1 \pi_{\sigma}(n_1, n_2) \leq \PP_p \big( \arm_{\sigma}(\Ann_{n_1, n_2}) \big) \leq C_2 \pi_{\sigma}(n_1, n_2)
	\end{equation}
	(see Theorem 27 in \cite{No2008}).
	
	\item \emph{Asymptotic equivalences for $\theta$ and $L$.} As $p$ approaches $p_c$, the following estimates hold true:
	\begin{equation} \label{eq:equiv_theta}
	\theta(p) \asymp \pi_1(L(p)) \quad \text{as $p \searrow p_c$}
	\end{equation}
	(see Theorem 2 in \cite{Ke1987}, or (7.25) in \cite{No2008}), and
	\begin{equation} \label{eq:equiv_L}
	\big| p - p_c \big| L(p)^2 \pi_4 \big( L(p) \big) \asymp 1 \quad \text{as $p \to p_c$}
	\end{equation}
	(see (4.5) in \cite{Ke1987}, or Proposition 34 in \cite{No2008}).
	
	Note that combining \eqref{eq:equiv_theta} and \eqref{eq:equiv_L} with \eqref{eq:arm_exp} (for $\sigma = (o)$ and $\sigma = (ovov)$, respectively) yields
	\begin{equation} \label{eq:exp_theta_L}
	\theta(p) = (p - p_c)^{\frac{5}{36} + o(1)} \text{ as } p \searrow p_c \quad \text{and} \quad L(p) = |p - p_c|^{-\frac{4}{3} + o(1)} \text{ as } p \to p_c.
	\end{equation}
	
	\item \emph{Volume estimates.} Let $(p_k)_{k \geq 1}$ satisfying $p_c < p_k < 1$. If $(n_k)_{k \geq 1}$ is a sequence of integers such that $n_k \to \infty$ and $L(p_k) \ll n_k$ as $k \to \infty$, then
	\begin{equation} \label{eq:largest_cluster}
	\text{for all $\ve \in (0,1)$,} \quad \PP_{p_k} \left( \frac{|\lclus_{\Ball_{n_k}}|}{\theta(p_k) |\Ball_{n_k}|} \notin (1 - \ve, 1 + \ve) \right) \stackrel[k \to \infty]{}{\longrightarrow} 0,
	\end{equation}
	where $|\lclus_{\Ball_{n_k}}|$ is the volume of the largest occupied cluster in $\Ball_{n_k}$ (see Theorem 3.2 in \cite{BCKS2001}).
	
	We will also need the following upper bound for the existence of abnormally large clusters at criticality. There exist universal constants $C_1, C_2 > 0$ such that: for all $n \geq 1$ and $x \geq 0$,
	\begin{equation} \label{eq:largest_cluster_exp_tail}
	\PP_{p_c} \big( |\lclus_{\Ball_n}| \geq x n^2 \pi_1(n) \big) \leq C_1 e^{-C_2 x}
	\end{equation}
	(see Proposition 6.3 (i) in \cite{BCKS1999}).
\end{enumerate}

We will actually use the following more quantitative version of \eqref{eq:largest_cluster}, which can be obtained by adapting the reasoning in \cite{BCKS2001}. For the reader's convenience, we include a proof in Appendix (see Section \ref{sec:proof_largest}).
\begin{lemma} \label{lem:largest_cluster_quant}
	For all $\ve \in (0,1)$, there exists $C = C(\ve)$ such that: for all $p > p_c$ and $n \geq 1$,
	\begin{equation} \label{eq:largest_cluster_quant}
	\PP_p \bigg( \bigg\{ \frac{|\lclus_{\Ball_n}|}{\theta(p) |\Ball_n|} \in (1 - \ve, 1 + \ve) \bigg\} \cap \circuitarm \Big( \Ann_{\frac{1}{2} n, n} \, \big| \, \lclus_{\Ball_n} \Big) \bigg) \geq 1 - C \, \frac{L(p)}{n},
	\end{equation}
	where $\circuitarm ( \Ann_{\frac{1}{2} n, n} \, | \, \lclus_{\Ball_n} )$ denotes the event that $\lclus_{\Ball_n}$ contains an occupied circuit and an occupied crossing (i.e. arm) in $\Ann_{\frac{1}{2} n, n}$.
\end{lemma}
Moreover, note that in annuli of the form $\Ann_{\eta n, n}$, $\eta \in (0,1)$ fixed, an analogous result (for $|\lclus_{\Ann_{\eta n, n}}|$) holds, for similar reasons (with some $C = C(\ve,\eta)$).

\begin{remark} \label{rem:other_lattices}
All results above remain valid for other two-dimensional lattices with enough symmetry, and also for bond percolation, except the precise derivation of arm exponents (\eqref{eq:arm_exp} and below, which rely on Smirnov's proof of conformal invariance \cite{Sm2001} mentioned earlier), and the immediately related properties \eqref{eq:uniform_arm_exp} and \eqref{eq:exp_theta_L}. Indeed, the article \cite{Ke1987} is written in a more general setting, which includes in particular site percolation and bond percolation on the square lattice (we refer the reader to the introduction of that paper for more details). In such a case, we have at our disposal the following a-priori bounds on the one- and four-arm events: there exist $\alpha, \alpha' > 0$ and $c_1, c_2, c'_1, c'_2 > 0$ such that for all $1 \leq n_1 < n_2$,
\begin{equation} \label{eq:apriori_arm}
c_1 \bigg( \frac{n_1}{n_2} \bigg)^{\frac{1}{2}} \leq \pi_1(n_1, n_2) \leq c_2 \bigg( \frac{n_1}{n_2} \bigg)^{\alpha} \quad \text{and} \quad c'_1 \bigg( \frac{n_1}{n_2} \bigg)^{2 - \alpha} \leq \pi_4(n_1, n_2) \leq c'_2 \bigg( \frac{n_1}{n_2} \bigg)^{1 + \alpha'}
\end{equation}
(see e.g. the explanations below (2.12) and (2.13) in \cite{BN2018}).

These bounds are enough to obtain partial results for frozen percolation, as explained in Remark~\ref{rem:other_lattices3} below. However, in the case of forest fires, our proofs rely crucially on a comparison to a percolation process with impurities, as we explain in Section~\ref{sec:perc_impurities}. In order to study this latter process, the precise knowledge of $\alpha_4$, as well as of $\alpha_{(oo)}$, is required (for the time being). More specifically, the results in \cite{BN2018} use the inequality $\alpha_{(oo)} > \alpha_4 - 1$ to estimate the impact of impurities on four-arm events, see the discussion in Section~4.3 of that paper.
\end{remark}

\subsection{Near-critical percolation with impurities} \label{sec:perc_impurities}

We now present tools to analyze the forest fire processes described informally in the Introduction (we define them precisely in Section~\ref{sec:def_processes}). It was shown in \cite{BN2018} that the intricate behavior of these processes, when the density of trees approaches $p_c$, could be well understood thanks to an auxiliary percolation process, where additional ``impurities'' (also called ``holes'' later) are created on the lattice. The impurities appear in an independent fashion, their size following a well-chosen distribution related to forest fires, which is ``heavy-tailed'' in some sense. This process, introduced in \cite{BN2018}, is instrumental in that paper to estimate the joint effect of fires (the precise comparison with forest fires is explained in Section~\ref{sec:stoch_domin_FF}). Hence, it plays a crucial role in our proofs as well, in Sections~\ref{sec:scale_minf_FF} and \ref{sec:proof_FF}.

We want to emphasize that in principle, it could well be the case that the obstacles created by the burnt areas affect significantly the phase transition of percolation, even when they do not strongly disconnect the lattice. Already in the case of usual Bernoulli configuration, consider a configuration at criticality in $\Ball_n$: if every site is switched from occupied to vacant with probability $n^{-\xi}$, for some exponent $\xi > 0$ (thus creating ``single-site impurities''), then the connectivity properties of the resulting configuration, as $n \to \infty$, depend heavily on the value of $\xi$. If $\xi > \xi_c := \frac{3}{4}$ ($= \frac{1}{\nu}$, where $\nu$ is the critical exponent associated with $L$, see \eqref{eq:exp_theta_L}), then the configuration in $\Ball_n$ remains comparable to critical percolation (and as shown in \cite{NW2009}, the scaling limit as $n \to \infty$ is just the same as in the critical regime), while if $\xi \in (0,\xi_c)$, it behaves like subcritical percolation: there are ``too many impurities'', even though their density tends to $0$ and they do not disconnect the lattice.

First, let us define the process precisely. It is parametrized by a parameter $m \geq 1$, and we are interested in its behavior as $m \to \infty$ (which intuitively corresponds to $t \nearrow t_c$ for forest fire processes). We follow the notation from Section~3 of \cite{BN2018}. For all $m \geq 1$, we consider a distribution $\rho^{(m)}$ on $[0,+\infty)$, and a family $(\pi_v^{(m)})_{v \in V}$ of probabilities indexed by the vertices of the lattice (typically very small, as $m \to \infty$). First, we perform Bernoulli percolation with a parameter $p \in [0,1]$ on the lattice. For each $v\in V$, independently of other vertices, we put a hole centered at $v$ with probability $\pi_v^{(m)}$, and the radius $r_v$ of the hole is distributed according to $\rho^{(m)}$: all vertices in the hole $H_v := \Ball_{r_v}(v)$ are turned vacant.

If we denote by $I_v$ the indicator function that there is a hole centered at $v$, the families $(r_v)_{v \in V}$ and $(I_v)_{v \in V}$ are always assumed to be independent. For notational convenience, we set $H_v := \emptyset$ if there is no hole centered at $v$. We denote by $\PPh_p^{(m)} := \PPh_p^{\pi^{(m)},\rho^{(m)}}$ the corresponding probability measure, on percolation configurations together with holes. When we are considering events involving only the configuration of holes, we stress it by using the notation $\PPh^{(m)}$. Finally, we mention that the process with impurities satisfies the FKG inequality, as observed in Remark~3.1 of \cite{BN2018}.

The setting introduced in \cite{BN2018} does not extend to scales close to $m_\infty(\zeta)$ (the characteristic scale mentioned in the Introduction, and defined precisely in Section~\ref{sec:successive}). Indeed, it only applies up to $m_\infty(\zeta) \cdot \zeta^{\delta}$, for any given $\delta > 0$ (arbitrarily small). This is enough for the applications discussed in \cite{BN2018}, in particular to establish the existence of exceptional scales, but in our situation, we need to consider a generalization. We make the following hypothesis on the distribution $\rho^{(m)}$ and the family $(\pi_v^{(m)})_{v \in V}$.
\begin{assumption} \label{ass:impurities}
There exist $c>0$, $\gamma \in (1,2)$, and a function $\upsilon(m)$ ($m \geq 1$) such that
\begin{equation} \label{eq:impurities_assumption}
\text{for all $m \geq 1$, $r \geq 0$, and $v \in V$,} \quad \pi_v^{(m)} \cdot \rho^{(m)}([r, +\infty)) \leq \frac{\upsilon(m)}{m^2} \left(\frac{r \vee 1}{m}\right)^{\gamma - 2} e^{-c\frac{r}{m}}.
\end{equation}
\end{assumption}
The quantity $\upsilon(m)$ should be thought of as measuring, up to a constant factor, the maximum ``density'' of impurities (that is, the probability $\PPh_p^{(m)}(\exists v \in V \: : \: 0 \in H_v)$ that $0$ belongs to at least one impurity, as can be seen from a short computation), and we will have to make an additional assumption on it for our results (see also Remark~\ref{rem:large_density} below). The precise condition will vary, but we typically require the holes to cover a fraction of the lattice which is rather small, but possibly bounded away from $0$, i.e. that $\upsilon(m) \leq \delta_0$ for some $\delta_0 > 0$ (which has to be chosen sufficiently small, depending on the context).

\begin{remark}
Note that in the setting of \cite{BN2018}, with some exponents $\alpha < 2$ and $\beta > 0$, Assumptions~1 and 2 in that paper imply our assumption \eqref{eq:impurities_assumption} in the case $\alpha > 1$, with, in our notation, $\gamma = \alpha$ and $\upsilon(m) = m^{\alpha-\beta}$ ($\to 0$ as $m \to \infty$, since $\beta > \alpha$ from Assumption~2). The condition considered here is thus more general for $\alpha > 1$, but on the other hand, we do not address the case $\alpha \in (\frac{3}{4},1]$. Values in this range are also studied in \cite{BN2018}, where a complete ``phase diagram'' is obtained. They require more care, but they are not needed for forest fires (which correspond, roughly speaking, to values $\alpha$ close to $\frac{55}{48} > 1$).
\end{remark}

In most applications to forest fires later on, $\upsilon(m) \to 0$ as $m \to \infty$, but not necessarily as a power law in $m$: we also need to allow e.g. $\upsilon(m) \leq (\log m)^{-\delta}$, which creates additional difficulties. Moreover, in order to study avalanches all the way up to $m_{\infty}(\zeta)$, we also have to include the case when $\upsilon(m)$ is small, but remains non-negligible. We need results analogous to some of the properties listed in Section~\ref{sec:classical}, that we state now.

\subsubsection*{Crossing holes}

For an annulus $A= \Ann_{n_1, n_2}(z)$ ($z\in \mathbb{C}, 1\leq n_1<n_2$), consider the event
$$\calH(A) :=\{\exists v\in V \: : \: H_v \cap \dout \Ball_{n_1}(z) \neq \emptyset \text{ and } H_v \cap \din \Ball_{n_2}(z) \neq \emptyset\}$$
that $A$ is crossed by a hole.
\begin{lemma} \label{lem:crossing_hole}
There exist constants $C, C'>0$ (depending only on $c$ and $\gamma$) such that the following holds. For all $m\geq 1$, for all annuli $A= A_{n_1, n_2}(z)$ with $z\in V$ and $1\leq n_1\leq \frac{n_2}{2}$,
$$\PPh^{(m)} \left(\calH(A)\right) \leq C \upsilon(m) e^{-C' \frac{n_1}{m}}.$$
\end{lemma}
This lemma is used repeatedly when establishing the results below, and for similar reasons, it is helpful in Section~\ref{sec:proof_FF} to produce some form of spatial independence.

\subsubsection*{Russo-Seymour-Welsh bounds}

We need to adapt the a-priori estimate on box-crossing probabilities \eqref{eq:RSW}.
\begin{proposition} \label{prop:crossing_impurities}
Let $K \geq 1$. There exists $C = C(c, \gamma, K) > 0$ such that: for all $p \in (0,1)$ and $1 \leq n \leq K (m \wedge L(p))$,
\begin{equation} \label{eq:RSW_impurities}
\PPh_p^{(m)} \big( \Ch([0,2n] \times [0,n]) \big) \geq \big( 1 - C \upsilon(m) \big) \, \PP_p \big( \Ch([0,2n] \times [0,n]) \big).
\end{equation}
\end{proposition}
In particular, note that $\upsilon(m)$ has to be small enough (in terms of $c, \gamma, K$) for the statement to be non-trivial.

\subsubsection*{Stretched exponential decay property}

We also make use of the following (slightly weaker) version of the exponential decay property \eqref{eq:exp_decay}.
\begin{proposition} \label{prop:exp_decay_impurities}
Let $K \geq 1$. There exist $\lambda_1, \lambda_2, \delta_0 > 0$ (depending only on $c$, $\gamma$, $K$) such that if we assume that $\upsilon(m) \leq \delta_0$ for all $m \geq 1$, then the following holds. For all $m \geq 1$, $n \geq 1$, and $p > p_c$ with $L(p) \leq K m$,
\begin{equation} \label{eq:stretched_exp_decay}
\PPh_p^{(m)} \big( \Ch([0,2n] \times [0,n]) \big) \geq 1 - \lambda_1 e^{- \lambda_2 ( \frac{n}{m} )^{\frac{1}{2}}}.
\end{equation}
\end{proposition}

\bigskip

The above results can all be obtained, to a great extent, from the corresponding proofs in \cite{BN2018}, through minor adjustments (in order to replace the use of Assumptions~1 and 2 from \cite{BN2018} by our Assumption~\ref{ass:impurities} just above). In Appendix (Section~\ref{sec:proof_impurities}), we highlight the non-trivial modifications which are required.

\subsubsection*{Volume estimates}

Finally, we derive an analog of the quantitative Lemma~\ref{lem:largest_cluster_quant}, on the size of the largest connected component in a box.
\begin{proposition} \label{prop:largest_cluster_impurities}
Let $K \geq 1$ and $\ve \in (0,1)$. There exist $C, \delta_0 > 0$ (depending on $c$, $\gamma$, $K$, $\ve$) such that if we assume that $\upsilon(m) \leq \delta_0$ for all $m \geq 1$, then the following holds. For all $m \geq 1$, $p > p_c$ with $L(p) \leq K m$ and $\upsilon(m) \log \big( \frac{m}{L(p)} \big) \leq 2 \delta_0$, and all $n \geq 1$,
\begin{equation} \label{eq:largest_cluster_impurities}
\PPh_p^{(m)} \bigg( \bigg\{ \frac{|\lclus_{\Ball_n}|}{\theta(p) |\Ball_n|} \in (1 - \ve, 1 + \ve) \bigg\} \cap \circuitarm \Big( \Ann_{\frac{1}{2} n, n} \, \big| \, \lclus_{\Ball_n} \Big) \bigg) \geq 1 - C \, \frac{m}{n}.
\end{equation}
Furthermore, the same conclusion holds if we include, into the l.h.s. of \eqref{eq:largest_cluster_impurities}, the additional property
$$\Bigg\{ \frac{\big|\lclus_{\Ball_n} \cap \Ball_{\frac{1}{2} n}\big|}{\theta(p) \big|\Ball_{\frac{1}{2} n}\big|} \in (1 - \ve, 1 + \ve) \Bigg\}.$$
\end{proposition}
Recall that the event $\circuitarm ( \Ann_{\frac{1}{2} n, n} \, | \, \lclus_{\Ball_n} )$ requires the existence, inside $\lclus_{\Ball_n}$, of an occupied circuit and an occupied crossing in $\Ann_{\frac{1}{2} n, n}$.

This result is the only one whose proof requires some extra work, compared with the corresponding property in \cite{BN2018}, Proposition~5.5. First, it is not difficult to deduce a quantitative statement of the form \eqref{eq:largest_cluster_impurities} from the proof in \cite{BN2018}, which is essentially based on estimating the expectation and the variance of a well-chosen quantity. However, the proof uses boxes with side lengths $n \gg m (\log m)^2$, and we need to get rid of this hypothesis for our applications. Indeed, when analyzing avalanches for forest fire processes in Section~\ref{sec:proof_FF}, we follow the successive burnt clusters starting from scales of order, roughly, $\frac{m_\infty(\zeta)}{(\log \frac{1}{\zeta})^{\alpha}}$, for arbitrarily small $\alpha > 0$. This leads us to apply the estimate above in boxes with a side length $n \approx m (\log m)^{\beta}$, where $\beta = \frac{55}{41} \, \alpha \to 0$ as $\alpha \to 0$ (see Remark~\ref{rem:hypothesis_box}). In Section~\ref{sec:proof_largest_imp}, we explain how to handle this issue by adapting the proof of Lemma~\ref{lem:largest_cluster_quant}, given in Section~\ref{sec:proof_largest}.

Moreover, Proposition~5.5 requires $L(p)$ and $m$ to remain comparable to each other, while in our iterative procedure, we have to consider $n \gg m \gg L(p)$. The additional condition $\upsilon(m) \log \big( \frac{m}{L(p)} \big) \leq 2 \delta_0$ is here to ensure that $m$ is not too much larger than $L(p)$ (in our setting, it will be clear that this requirement is satisfied). For this purpose, we need to revisit the proof of Proposition~\ref{prop:exp_decay_impurities}, in order to derive an improved lower bound in the case when $m \gg L(p)$.

\begin{remark} \label{rem:large_density}
It is somewhat remarkable that the stability results remain valid even for a positive, small enough, density of impurities. On the other hand, as we explain now, it is easy to see that if this density is too large, the impurities disconnect strongly the lattice as $m \to \infty$, so that the resulting configuration looks subcritical.

To fix ideas, let us choose some arbitrary $c>0$ and $\gamma \in (1,2)$, and assume, in this remark only, that we have: for some $\delta > 0$,
$$\text{for all } v \in V, \:\: \pi_v^{(m)} = \frac{\delta}{m^{\gamma}}, \quad \text{and for all } r \geq 1, \:\: \rho^{(m)} \big( [r,+\infty) \big) = r^{\gamma - 2} e^{-c\frac{r}{m}}.$$
Note that this makes sense, at least for all $m$ large enough (depending on $\delta$), and Assumption~\ref{ass:impurities} is clearly satisfied. For each of the (disjoint) balls $\Ball_{\frac{m}{10}}(m \mathsf{v})$, $\mathsf{v} \in V$, consider the event
$$E_{\mathsf{v}} := \Big\{ \exists v \in \Ball_{\frac{m}{10}}(m \mathsf{v}) \: : \: r_v \geq 3m \Big\}$$
that a ``large'' impurity has its center in this ball. The events $E_{\mathsf{v}}$, $\mathsf{v} \in V$, are independent, and a small computation shows that $\PPh^{(m)}(E_{\mathsf{v}})$ can be made arbitrarily close to $1$ by choosing $\delta$ large enough, uniformly in $m$ (sufficiently large). In particular, we can ensure that $\PPh^{(m)}(E_{\mathsf{v}}) > p_c$, so that a.s., there are infinitely many disjoint circuits of large impurities around $0$. Indeed, for any two vertices $\mathsf{v} \sim \mathsf{v}'$, if both $E_{\mathsf{v}}$ and $E_{\mathsf{v}'}$ occur, then the corresponding large impurities overlap.
\end{remark}

\section{Frozen percolation and forest fires} \label{sec:FP_FF}

In this section, we turn to the frozen percolation and forest fire processes studied in the present paper. First, we introduce them formally in Section~\ref{sec:def_processes}. We then consider, in Section~\ref{sec:successive}, the transformations $R \mapsto t \in (t_c, \infty)$ (alluded to in the Introduction), in each case. Finally, we state the stochastic domination of early fires by independent impurities in Section~\ref{sec:stoch_domin_FF}.

\subsection{Definition of the processes} \label{sec:def_processes}

Let the graph $G = (V_G, E_G)$ be either the triangular lattice $\TT = (V_{\TT}, E_{\TT})$, or a finite subgraph of it (i.e. $E_G$ contains all edges of $E_{\TT}$ connecting two vertices in $V_G$). We now introduce more notation for the processes on $G$ considered in the present paper: the pure birth process, volume-frozen percolation, and forest fires. These processes are all of the form $\omega = (\omega(t))_{t \geq 0}$, where $\omega(t) = (\omega_v(t))_{v \in V_G}$ is a vertex configuration for all $t \geq 0$. Initially, at time $t=0$, every vertex $v \in V_G$ is vacant ($\omega_v(0) = 0$) in each of the processes. It may then be occupied (state $\omega_v = 1$) at later times, or (for some of the processes) it can be in a third state $\omega_v = -1$, whose meaning we explain later on. In the following, we denote by $\cluster_t(v)$ the cluster of $v$ at any time $t \geq 0$, i.e. the occupied cluster $\cluster(v)$ in the configuration $\omega(t)$ (recall that it is empty if $v$ is not occupied).

\subsubsection*{Pure birth process}

First, in the pure birth process each vertex $v \in V_G$, independently of the other ones, becomes occupied ($\omega_v = 1$) at rate $1$, and then simply remains occupied. Obviously, at each given time $t \geq 0$ the vertex configuration $\omega(t) = (\omega_v(t))_{v \in V_G}$ is distributed as Bernoulli site percolation with parameter
\begin{equation}
p(t) := 1 - e^{-t}.
\end{equation}
For all $t \geq 0$, the configuration $\omega(t)$ belongs to $\Omega_G := \{0,1\}^{V_G}$. When referring to this process, we simply use the notation $\PP$.

\subsubsection*{Frozen percolation}

Next, we consider frozen percolation, with parameter $N \geq 1$. In this process, each vertex $v \in V_G$ again tries to become occupied ($\omega_v = 1$) at rate $1$, but it is prevented from doing so (and thus remains vacant) if one of its neighbors belongs to an occupied cluster with volume at least $N$. In other words, occupied connected components keep growing as long as they contain at most $N-1$ vertices: if such a component happens to reach a volume $N$ or more, its growth is immediately stopped, and the vertices along its outer vertex boundary, which are all vacant, stay in this state forever. In this process, we say that a vertex $v$ is frozen if it belongs to an occupied cluster of volume $\geq N$ (in particular, $v$ has to be occupied). Again, $(\omega_v(t))_{v \in V_G} \in \Omega_G = \{0,1\}^{V_G}$ for all $t \geq 0$. Recall that in the Introduction, we introduced the notation $\PP_N^{(G)}$ for the probability measure governing this process.

In Section~\ref{sec:scale_minf_FP} (and only in this section), we also consider temporarily a process called \emph{frozen percolation with modified boundary rules}, where now each vertex $v \in V_G$ can be in three states: vacant ($\omega_v = 0$), occupied ($\omega_v = 1$), or frozen ($\omega_v = -1$). We denote by $\ol \Omega_G := \{-1,0,1\}^{V_G}$ the corresponding set of vertex configurations. This process is defined in a similar way as the previous one, except that when an occupied cluster reaches a volume $N$ and freezes, all its vertices become frozen, while the vertices along its outer boundary remain unaffected: these boundary vertices stay vacant immediately after the freezing time, and they may become occupied (and then possibly freeze) at later times. More precisely, when a vacant vertex $v$ tries to change its state, say at time $t$, we consider the union of the occupied clusters adjacent to $v$ at time $t^-$:
$$\ol{\cluster}_{t^-}(v) := \bigcup_{\substack{v' \in V_G\\ v' \sim v}} \cluster_{t^-}(v').$$
If $|\{v\} \cup \ol{\cluster}_{t^-}(v)|\geq N$, then all vertices in $\{v\} \cup \ol{\cluster}_{t^-}(v)$ become frozen at time $t$. Otherwise, $v$ simply becomes occupied at time $t$ (and obviously, $|\cluster_t(v)| \leq N-1$).

\subsubsection*{Forest fires with / without recovery}

Finally, we introduce forest fire processes, with parameter $\zeta > 0$. We consider two variants, with or without recovery, abbreviated as FFWR and FFWoR respectively. Again, the vertex configuration at each time $t \geq 0$ belongs to $\ol \Omega_G$. In both processes, every vertex becomes occupied at rate $1$, and is hit by lightning at rate $\zeta$, and we assume that the corresponding Poisson processes (the birth and lightning processes) are independent. When a vertex $v \in V_G$ is hit at a time $t$, nothing happens if $v$ is vacant, while if $v$ is occupied, all vertices in its occupied cluster $\cluster_t(v)$ become vacant (i.e. with state $0$) for the forest fire process with recovery, or burnt (state $-1$) for the process without recovery, instantaneously. Burnt vertices remain so in the future, while vacant vertices become occupied at later birth times. We use the notations $\PP_\zeta^{(G)}$ and $\ol{\PP}_\zeta^{(G)}$ for the FFWoR and FFWR processes, respectively, and denote them by $\sigma = (\sigma(t))_{t \geq 0}$ and $\ol{\sigma} = (\ol{\sigma}(t))_{t \geq 0}$.

\subsubsection*{Additional comments}

Note that $N$-frozen percolation can be represented as a finite-range interacting particle system (each vertex interacts only with the vertices within a distance $N$ from it). Hence, this process can be constructed using the general theory of such systems (see e.g. \cite{Li2005}). For forest fire processes on $\TT$, existence is much less clear, and can be established using arguments by D\"urre \cite{Du2006}.

As far as processes on a finite subgraph $G$ of $\TT$ are concerned, they do not necessarily coincide with the restriction to $G$ of the corresponding full-lattice processes (except of course in the case of the pure birth process). In the following, we consider all processes above as being coupled, in an obvious way, via the same family of Poisson processes $(\varsigma_v(t))_{t \in [0,\infty)}$, $v \in V_{\TT}$, of birth times (each with intensity $1$).

\subsection{Successive freezings / burnings} \label{sec:successive}

Let $t_c := - \log(1 - p_c) = \log 2$ be the time at which the percolation parameter $p(t_c)$ in the pure birth process equals $p_c = \frac{1}{2}$. With a slight abuse of notation, we now write $\PP_t = \PP_{p(t)}$ when referring to events for the pure birth process. In particular, we write $\theta(t) = \theta(p(t))$ and $L(t) = L(p(t))$ for $t \in [0,\infty)$, and we set $\theta(\infty) = \lim_{t \to \infty} \theta(t) = 1$ and $L(\infty) = \lim_{t \to \infty} L(t) = 0$ (recall that we use the regularized version of the characteristic length $L$). We also let $c_{\TT} := \frac{2}{\sqrt{3}}$, such that $|\Ball_n| \sim c_{\TT} \, n^2$ as $n \to \infty$.

We introduce a transformation $t \in (t_c, \infty) \mapsto \hat{t} \in (t_c, \infty)$ for volume-frozen percolation and forest fire processes, which describes the successive freezings / burnings, and already appeared (essentially) in \cite{BKN2015, BN2018}. We need to distinguish frozen percolation and forest fires, since the precise definition differs between these two cases.

\subsubsection*{Frozen percolation}

Consider first volume-frozen percolation, and let $N \geq 1$ be given. Since $\theta$ is continuous and strictly increasing on $[t_c,\infty)$, it is a bijection from $[t_c,\infty)$ to $[0,1)$. Hence, for all $R > c_{\TT}^{-\frac{1}{2}} \sqrt{N}$, there exists a unique $t \in (t_c,\infty)$ satisfying
\begin{equation} \label{eq:def_next_FP}
c_{\TT} R^2 \theta(t) = N,
\end{equation}
which we denote by $\next_N(R)$. Roughly speaking, $\next_N(R)$ gives the (approximate) time around which we expect the first cluster to freeze in a box with side length $R$ (at least if $R \ll m_{\infty}$, where the scale $m_{\infty}$ is introduced below). Clearly $\next_N$ is strictly decreasing. For convenience, we also set $\next_N(R) = \infty$ for $R \in [0, c_{\TT}^{-\frac{1}{2}} \sqrt{N}]$.

We then define $t \mapsto \hat{t} = \hat{t}(t,N) > t_c$ by $\hat{t} = \next_N(L(t))$, i.e. via the relation
\begin{equation} \label{eq:def_t_hat_FP}
c_{\TT} L(t)^2 \theta(\hat{t}) = N.
\end{equation}
This time $\hat{t}$ is well-defined for all $t \in (t_c, L^{-1}(c_{\TT}^{-\frac{1}{2}} \sqrt{N}))$ ($L^{-1}$ denoting the inverse function of $L$, seen as a function of time, on $(t_c, \infty)$), and we set $\hat{t} = \infty$ for $t \geq L^{-1}(c_{\TT}^{-\frac{1}{2}} \sqrt{N})$.

For future reference, note that $t \mapsto \hat{t}$ is strictly increasing on $(t_c, L^{-1}(c_{\TT}^{-\frac{1}{2}} \sqrt{N}))$, with $\hat{t} \to t_c$ as $t \searrow t_c$, and $\hat{t} \to \infty$ as $t \nearrow L^{-1}(c_{\TT}^{-\frac{1}{2}} \sqrt{N})$. Moreover,
$$c_{\TT} L(t)^2 \theta(t) = (t-t_c)^{-2 \cdot \frac{4}{3} + \frac{5}{36} + o(1)} \to \infty \quad \text{as $t \searrow t_c$}$$
(from \eqref{eq:exp_theta_L}), and $c_{\TT} L(t)^2 \theta(t) < N$ for $t = L^{-1}(c_{\TT}^{-\frac{1}{2}} \sqrt{N})$, so the equation $c_{\TT} L(t)^2 \theta(t) = N$, i.e. $\hat{t} = t$, has at least one solution $t \in (t_c, L^{-1}(c_{\TT}^{-\frac{1}{2}} \sqrt{N}))$. We thus introduce
\begin{equation} \label{eq:def_t_infty_FP}
t_{\infty} = t_{\infty}(N) := \sup \{ t > t_c \: : \: \hat{t} = t \} \in (t_c, L^{-1}(c_{\TT}^{-\frac{1}{2}} \sqrt{N}))
\end{equation}
(note that there does not seem to be any reason why $t \mapsto L(t)^2 \theta(t)$ should be strictly decreasing, though it is ``essentially the case''), and $m_{\infty} := L(t_{\infty})$. It then follows immediately from \eqref{eq:exp_theta_L} that
\begin{equation} \label{eq:exp_t_infty_FP}
t_{\infty}(N) = t_c + N^{- \frac{36}{91} + o(1)} \quad \text{and} \quad m_{\infty}(N) = N^{\frac{48}{91} + o(1)}
\end{equation}
as $N \to \infty$. From the definition \eqref{eq:def_t_infty_FP} of $t_{\infty}$, we have: for all $t > t_{\infty}$, $\hat{t} > t$.

\subsubsection*{Forest fires}

For forest fire processes, we define $R \mapsto \next_{\zeta}(R)$, for each given $\zeta > 0$, as follows: for all $R > 0$, $\next_{\zeta}(R)$ is the unique $t > t_c$ satisfying
\begin{equation} \label{eq:def_next_FF}
c_{\TT} R^2 \theta(t) \big( t - t_c \big) = \zeta^{-1}.
\end{equation}
In this case, the time $\next_{\zeta}(R)$ is well-defined for all $t > t_c$ (since $t \mapsto \theta(t) (t - t_c)$ is an increasing bijection from $[t_c,\infty)$ to $[0,\infty)$). It has a similar interpretation as for volume-frozen percolation.

As before, we then introduce $t \mapsto \hat{t} = \hat{t}(t,\zeta) > t_c$ via the relation $\hat{t} = \next_{\zeta}(L(t))$, i.e.
\begin{equation} \label{eq:def_t_hat_FF}
c_{\TT} L(t)^2 \theta(\hat{t}) \big( \hat{t} - t_c \big) = \zeta^{-1}.
\end{equation}

Observe that $t \mapsto \hat{t}$ is strictly increasing on $(t_c, \infty)$, with $\hat{t} \to t_c$ as $t \searrow t_c$ and $\hat{t} \to \infty$ as $t \to \infty$. We can again deduce from \eqref{eq:exp_theta_L} that the equation $\hat{t} = t$ has at least one solution $t \in (t_c, \infty)$ (note that by definition, $L(t)$ tends to $0$ exponentially fast as $t \to \infty$), so we introduce
\begin{equation} \label{eq:def_t_infty_FF}
t_{\infty} = t_{\infty}(\zeta) := \sup \{ t > t_c \: : \: \hat{t} = t \} \in (t_c, \infty)
\end{equation}
and $m_{\infty} := L(t_{\infty})$. We can then get from \eqref{eq:exp_theta_L} that
\begin{equation} \label{eq:exp_t_infty_FF}
t_{\infty}(\zeta) = t_c + \zeta^{\frac{36}{55} + o(1)} \quad \text{and} \quad m_{\infty}(\zeta) = \zeta^{- \frac{48}{55} + o(1)}
\end{equation}
as $\zeta \searrow 0$. Again, the definition \eqref{eq:def_t_infty_FF} of $t_{\infty}$ implies that: for all $t > t_{\infty}$, $\hat{t} > t$.

\subsubsection*{Iteration exponent}

The next lemma studies the transformation $R \mapsto \next_{\cdot}(R)$, comparing both $L(\next_{\cdot}(R))$ and $R$ to ${m_{\infty}}$. It will be central in our reasonings, in order to control the speed at which one ``moves down'' the scales, starting from $m_{\infty}$.

\begin{lemma} \label{lem:one_iteration}
Consider $N$-volume-frozen percolation. For all $\ve > 0$, there exist $0 < C_1 < C_2$ (depending on $\ve$) such that: for all $N \geq 1$ and $2 c_{\TT}^{-\frac{1}{2}} \sqrt{N} \leq r \leq R$,
\begin{equation} \label{eq:ratioL}
C_1 \bigg( \frac{r}{R} \bigg)^{\aval + \ve} \leq \frac{L(\next_N(r))}{L(\next_N(R))} \leq C_2 \bigg( \frac{r}{R} \bigg)^{\aval - \ve},
\end{equation}
where $\aval = \aval^{\textnormal{FP}} = \frac{96}{5}$. Moreover, the same statement holds for forest fire processes, but with $\aval = \aval^{\textnormal{FF}} = \frac{96}{41}$ instead (i.e. for the corresponding map $\next_{\zeta}$, and for all $\zeta > 0$ and $\frac{1}{\sqrt{\zeta}} \leq r \leq R$).
\end{lemma}

\begin{proof}[Proof of Lemma \ref{lem:one_iteration}]
We first consider the case of frozen percolation. We have $c_{\TT} r^2 \theta(\next_N(r)) = N = c_{\TT} R^2 \theta(\next_N(R))$ (from \eqref{eq:def_t_hat_FP}), so
\begin{equation}
\frac{r^2}{R^2} = \frac{\theta(\next_N(R))}{\theta(\next_N(r))} \asymp \frac{\pi_1(L(\next_N(R)))}{\pi_1(L(\next_N(r)))} \asymp \pi_1(L(\next_N(r)), L(\next_N(R)))
\end{equation}
(using \eqref{eq:equiv_theta} and \eqref{eq:quasi_mult}). The desired result now follows from \eqref{eq:uniform_arm_exp} with $\sigma = (o)$.

In the case of forest fires, we have $c_{\TT} r^2 \theta(\next_{\zeta}(r)) \big| \next_{\zeta}(r) - t_c \big| = \zeta^{-1} = c_{\TT} R^2 \theta(\next_{\zeta}(R)) \big| \next_{\zeta}(R) - t_c \big|$ (from \eqref{eq:def_t_hat_FF}). Hence, it follows from \eqref{eq:equiv_theta} and \eqref{eq:equiv_L} that
\begin{equation}
\frac{r^2}{R^2} = \frac{\theta(\next_{\zeta}(R))}{\theta(\next_{\zeta}(r))} \cdot \frac{\big| \next_{\zeta}(R) - t_c \big|}{\big| \next_{\zeta}(r) - t_c \big|} \asymp  \frac{\pi_1(L(\next_{\zeta}(R)))}{\pi_1(L(\next_{\zeta}(r)))} \cdot \frac{L(\next_{\zeta}(r))^2 \pi_4(L(\next_{\zeta}(r)))}{L(\next_{\zeta}(R))^2 \pi_4(L(\next_{\zeta}(R)))}.
\end{equation}
Hence, using \eqref{eq:quasi_mult},
\begin{equation}
\frac{r^2}{R^2} \asymp \pi_1(L(\next_{\zeta}(r)), L(\next_{\zeta}(R))) \cdot \frac{L(\next_{\zeta}(r))^2}{L(\next_{\zeta}(R))^2} \cdot \pi_4(L(\next_{\zeta}(r)), L(\next_{\zeta}(R)))^{-1}.
\end{equation}
This allows us to conclude, using twice \eqref{eq:uniform_arm_exp}, with $\sigma = (o)$ and $\sigma = (ovov)$.
\end{proof}

In both cases, we can get from $L(\next_{\cdot}(m_{\infty})) = m_{\infty}$ that: for all $r \leq m_{\infty}$,
\begin{equation} \label{eq:compar_minf}
C_1 \bigg( \frac{r}{m_{\infty}} \bigg)^{\aval + \ve} \leq \frac{L(\next_{\cdot}(r))}{m_{\infty}} \leq C_2 \bigg( \frac{r}{m_{\infty}} \bigg)^{\aval - \ve}.
\end{equation}
This yields in particular
\begin{equation} \label{eq:diam_apart}
\frac{L(\next_{\cdot}(r))}{r} = \frac{L(\next_{\cdot}(r))}{m_{\infty}} \cdot \bigg( \frac{r}{m_{\infty}} \bigg)^{-1} \leq C_2 \bigg( \frac{r}{m_{\infty}} \bigg)^{\aval - 1 - \ve},
\end{equation}
so the fact that $\aval > 1$ ensures that for the successive frozen / burnt clusters, their typical diameters get further and further apart.

\begin{remark} \label{rem:other_lattices2}
We can also consider other two-dimensional lattices as in Remark~\ref{rem:other_lattices}, e.g. the square lattice. If the a-priori bounds \eqref{eq:apriori_arm} are available, as a substitute for \eqref{eq:uniform_arm_exp}, the reader can check that the proof of Lemma~\ref{lem:one_iteration} yields that \eqref{eq:ratioL} can be replaced by
$$C_1 \bigg( \frac{r}{R} \bigg)^{\frac{2}{\alpha}} \leq \frac{L(\next_N(r))}{L(\next_N(R))} \leq C_2 \bigg( \frac{r}{R} \bigg)^4 \quad \text{and} \quad C'_1 \bigg( \frac{r}{R} \bigg)^{\frac{1}{\alpha}} \leq \frac{L(\next_{\zeta}(r))}{L(\next_{\zeta}(R))} \leq C'_2 \bigg( \frac{r}{R} \bigg)^{\frac{4}{3 - 2 \alpha'}},$$
in the case of frozen percolation and forest fires, respectively. In particular, the exponents in the r.h.s. are $> 1$ for both processes, so that the previous observation about successive frozen or burnt clusters (\eqref{eq:diam_apart} and below) applies.

Furthermore, \eqref{eq:exp_t_infty_FP} and \eqref{eq:exp_t_infty_FF} become (resp.)
$$C_3 N^{\frac{1}{2 - \alpha}} \leq m_{\infty}(N) \leq C_4 N^{\frac{2}{3}} \quad \text{and} \quad C'_3 \zeta^{- \frac{1}{2 - 2 \alpha}} \leq m_{\infty}(\zeta) \leq C'_4 \zeta^{- \frac{2}{2 \alpha' + 1}}$$
for some $C_3, C_4, C'_3, C'_4 > 0$. Observe that the exponents $\frac{1}{2 - \alpha}$ and $\frac{1}{2 - 2 \alpha}$ in the l.h.s. are both $> \frac{1}{2}$.
\end{remark}

\subsection{Stochastic domination by percolation with impurities} \label{sec:stoch_domin_FF}

We now explain how the percolation process with impurities can be used as a stochastic lower bound for forest fire processes. More precisely, we compare it to forest fires where ignitions are stopped at a time $T \in (0,t_c)$ (i.e. we ignore ignitions occurring at later times $s > T$), without or with recovery. We denote these processes by $\sigma^{[T]} = (\sigma^{[T]}(t))_{t \geq 0}$ and $\ol{\sigma}^{[T]} = (\ol{\sigma}^{[T]}(t))_{t \geq 0}$, respectively. In particular, $\sigma^{[T]}(t) = \sigma(t)$ for all $t \in [0,T]$, and similarly for $\ol{\sigma}$. In the result below, we let
$$\rad(C) := \inf \{ n \geq 0 \: : \: C \subseteq \Ball_n\}$$
be the radius, seen from $0$, of a subset $C \subseteq V_{\TT}$.

\begin{lemma} \label{lem:stoch_domin}
Assume that the graph $G=(V,E)$ is finite. Let $\zeta \in (0,\frac{1}{2})$, $\bar{\ve} \in (0, \frac{t_c}{2})$, and $m = L(t_c - \bar{\ve})$. Consider percolation with impurities obtained from $\pi_v^{(m)} = t_c \, \zeta \cdot e^{t_c \, \zeta}$ for all $v \in V$, and the distribution $\rho^{(m)}$ of $\rad(\cluster(0))$ in Bernoulli percolation with parameter $\tau$, where $\tau$ is uniform in $[0,t_c - \bar{\ve}]$ (denote this process by $\tilde{\omega}$).
\begin{enumerate}[(i)]
\item For all $t \geq t_c - \bar{\ve}$, $(\ind_{\sigma^{[t_c - \bar{\ve}]}_v(t) = 1})_{v \in V}$ stochastically dominates $(\tilde{\omega}_v(t))_{v \in V}$.

\item Moreover, there exists $c > 0$ (universal) such that the following holds. For any $\ve \in (0,\frac{7}{48})$, if $m \leq m_{\infty}(\zeta)$, then Assumption~\ref{ass:impurities} is satisfied with $c$,
$$\gamma = \frac{55}{48} - \ve \:\: ( \in (1,2) ), \quad \text{and} \quad \upsilon(m) = c' \bigg( \frac{m}{m_{\infty}} \bigg)^{\frac{55}{48} - \ve},$$
for some $c' = c'(\ve) > 0$.
\end{enumerate}
\end{lemma}

Note that this result also holds for $\ol{\sigma}$ instead of $\sigma$ (from the same proof), but we will not use this fact later.

\begin{proof}[Proof of Lemma \ref{lem:stoch_domin}]
(i) The stochastic domination follows directly by combining two ingredients from \cite{BN2018}: Lemma~6.2, and the last paragraph in the proof of Lemma~6.8.

\bigskip

(ii) We can then use the computation in the proof of Lemma~6.8 to check that Assumption~\ref{ass:impurities} is satisfied, as we explain now. This computation shows the existence of universal constants $C, C' > 0$ such that: for all $r \geq 1$,
\begin{equation} \label{eq:upper_bd_rho}
\rho^{(m)}([r, +\infty)) \leq C \frac{\pi_1(m)}{m^2 \pi_4(m)} e^{- C' \frac{r}{m}} \: \: \: (r \geq m), \quad \text{and} \: \: \: \rho^{(m)}([r, +\infty)) \leq C \frac{\pi_1(r)}{r^2 \pi_4(r)} \: \: \: (r \leq m).
\end{equation}
Furthermore, it follows from \eqref{eq:def_t_hat_FF} and \eqref{eq:def_t_infty_FF} that
$$\zeta^{-1} \asymp L(t_{\infty})^2 \theta(t_{\infty}) \big( t_{\infty} - t_c \big).$$
Combining it with \eqref{eq:equiv_L} and \eqref{eq:equiv_theta} (and using $m_{\infty} = L(t_{\infty})$), we obtain
\begin{equation} \label{eq:estim_zeta}
\zeta \asymp \frac{\pi_4(m_{\infty})}{\pi_1(m_{\infty})}.
\end{equation}
For all $v \in V$, observe that
\begin{equation} \label{eq:upper_bd_pi}
\pi_v^{(m)} \leq 2 t_c \cdot \zeta
\end{equation}
from the assumption $\zeta < \frac{1}{2}$. We distinguish the two cases $r \geq m$ and $r \leq m$.
\begin{itemize}
\item \ul{Case 1}: $r \geq m$. We get from \eqref{eq:upper_bd_rho}, \eqref{eq:estim_zeta} and \eqref{eq:upper_bd_pi} that
\begin{equation}
\pi_v^{(m)} \cdot \rho^{(m)}([r, +\infty)) \leq C'' \frac{\pi_4(m_{\infty})}{\pi_1(m_{\infty})} \cdot \frac{\pi_1(m)}{m^2 \pi_4(m)} e^{- C' \frac{r}{m}} \leq \frac{\upsilon'(m)}{m^2} e^{- C' \frac{r}{m}},
\end{equation}
where $C''$ is universal, and
$$\upsilon'(m) = C'' \frac{\pi_4(m_{\infty})}{\pi_4(m)} \cdot \frac{\pi_1(m)}{\pi_1(m_{\infty})} \leq \tilde{C} \bigg( \frac{m}{m_{\infty}} \bigg)^{\frac{5}{4} - \frac{\ve}{2}} \cdot \bigg( \frac{m}{m_{\infty}} \bigg)^{- \frac{5}{48} - \frac{\ve}{2}} = \tilde{C} \bigg( \frac{m}{m_{\infty}} \bigg)^{\frac{55}{48} - \ve},$$
for some $\tilde{C} = \tilde{C}(\ve)$ (using twice \eqref{eq:quasi_mult} and \eqref{eq:uniform_arm_exp} for the inequality). We can then use
$$e^{- C' \frac{r}{m}} \leq \tilde{C}' \bigg( \frac{r}{m} \bigg)^{\frac{55}{48} - 2 - \ve} e^{- \frac{C'}{2} \cdot \frac{r}{m}}$$
to deduce \eqref{eq:impurities_assumption} (with $c = \frac{C'}{2}$, $\gamma = \frac{55}{48} - \ve$, and $\upsilon(m) = \tilde{C}' \cdot \tilde{C} ( \frac{m}{m_{\infty}} )^{\frac{55}{48} - \ve}$).

\item \ul{Case 2}: $0 \leq r \leq m$. If $r \geq 1$,
\begin{equation}
\pi_v^{(m)} \cdot \rho^{(m)}([r, +\infty)) \leq C'' \frac{\pi_4(m_{\infty})}{\pi_1(m_{\infty})} \cdot \frac{\pi_1(r)}{r^2 \pi_4(r)} \leq \frac{\upsilon''(r,m)}{m^2},
\end{equation}
(using again \eqref{eq:upper_bd_rho}, \eqref{eq:estim_zeta} and \eqref{eq:upper_bd_pi}), with
\begin{align*}
\upsilon''(r,m) & = C'' \frac{\pi_4(m_{\infty})}{\pi_4(r)} \cdot \frac{\pi_1(r)}{\pi_1(m_{\infty})} \cdot \bigg( \frac{m}{r} \bigg)^2\\[1mm]
& \leq \tilde{C} \bigg( \frac{r}{m_{\infty}} \bigg)^{\frac{55}{48} - \ve} \cdot \bigg( \frac{m}{r} \bigg)^2 = \tilde{C} \bigg( \frac{m}{m_{\infty}} \bigg)^{\frac{55}{48} - \ve} \cdot \bigg( \frac{r}{m} \bigg)^{\frac{55}{48} - 2 - \ve}
\end{align*}
(where the inequality comes from \eqref{eq:quasi_mult} and \eqref{eq:uniform_arm_exp}). We have thus checked \eqref{eq:impurities_assumption} in this case as well. Moreover, \eqref{eq:impurities_assumption} also holds for $0 \leq r \leq 1$, starting from
$$\pi_v^{(m)} \cdot \rho^{(m)}([r, +\infty)) \leq \pi_v^{(m)} \leq C''' \frac{\pi_4(m_{\infty})}{\pi_1(m_{\infty})},$$
and following a similar computation.
\end{itemize}
Since \eqref{eq:impurities_assumption} is verified in both cases, the proof is complete.
\end{proof}

Note that formally, Lemma~6.2 of \cite{BN2018} is written for finite graphs only, and this is why we assumed $G$ to be finite. However, in our situation all connected components which burn are finite, since we stop ignitions at the subcritical time $t_c - \bar{\ve}$ (we even have an exponentially decaying upper bound on their diameter). Hence, it would not be difficult to extend the proof of Lemma~6.2 to the full lattice $\TT$ (but this is not needed for our applications).

\bigskip

\textbf{In the remainder of the paper, we always take $\ve = \frac{1}{48}$, so that $\gamma = \frac{9}{8}$, $c' = c'(\ve)$ can be considered as an absolute constant, and $\upsilon(m) = c' \big( \frac{m}{m_{\infty}} \big)^{\frac{9}{8}}$.}

\section{Frozen percolation and forest fires at scale $m_{\infty}$} \label{sec:scale_minf}

In this section, we explain how to handle the processes around their respective characteristic scales ($m_{\infty}(N)$ or $m_{\infty}(\zeta)$): frozen percolation in Section~\ref{sec:scale_minf_FP}, and the FFWoR process in Section~\ref{sec:scale_minf_FF}. In each case, we show that the number of frozen / burnt clusters surrounding $0$ and with a diameter of order $m_{\infty}$ is, roughly speaking, ``tight''. This is achieved through the introduction of a near-critical parameter scale associated with $m_{\infty}$, which allows one to compare the process studied to near-critical percolation, and use an argument based on the Russo-Seymour-Welsh bounds \eqref{eq:RSW} (for frozen percolation) or \eqref{eq:RSW_impurities} (in the case of forest fires).

Recall that we introduced the notation $\calF_t$ for the set of frozen / burnt clusters (depending on the process) surrounding the origin $0$, at every time $t \geq 0$ (this includes possibly the cluster containing $0$), with $\calF := \calF_{\infty}$ (two adjacent such clusters being considered as distinct if they freeze / burn at different times).

\subsection{Frozen percolation} \label{sec:scale_minf_FP}

\subsubsection{Near-critical parameter scale} \label{sec:def_nc_parameter}

In order to analyze the first frozen clusters in a box with side length of order $m_{\infty}$, we introduce a near-critical parameter scale as follows. For this purpose, we denote $t_{\infty}(N) = t_c + \ve_{\infty}(N)$.
\begin{definition}
For $\lambda \in \RR$ and $N \geq 1$, let
\begin{equation} \label{eq:def_scale_m_inf}
t_{\infty}^{\lambda}(N) := t_c + \lambda \, \ve_{\infty}(N).
\end{equation}
\end{definition}

We know from \eqref{eq:exp_t_infty_FP} that $t_{\infty}^{\lambda}(N) = t_c + \lambda \cdot N^{- \frac{36}{91} + o(1)}$. In particular, for any fixed $\lambda$, $t_{\infty}^{\lambda}(N) \to t_c$ as $N \to \infty$, so for all $N$ large enough, $t_{\infty}^{\lambda}(N) \in (0,\infty)$ (and we always assume it to be the case, implicitly).

\begin{remark}
Note that $( t_{\infty} - t_c ) (L(t_{\infty}))^2 \pi_4(L(t_{\infty})) \asymp 1$ (from \eqref{eq:equiv_L}), so (using $m_{\infty} = L(t_{\infty})$) $\ve_{\infty} (m_{\infty})^2 \pi_4(m_{\infty}) \asymp 1$. Hence, such a near-critical parameter scale could be defined equivalently as
$$\tilde{t}_{\infty}^{\lambda}(N) := t_c + \frac{\lambda}{(m_{\infty}(N))^2 \pi_4(m_{\infty}(N))}.$$
This expression is just the usual near-critical parameter scale, after replacement of $N$ by $m_{\infty}(N)$, which turned out to be very convenient to study near-critical percolation and related processes, for example in \cite{NW2009, Ki2015, GPS2018a, BN2017}.
\end{remark}

We will make use of the following elementary properties.
\begin{enumerate}[(i)]
\item For each fixed $\lambda \in \RR \setminus \{ 0 \}$,
\begin{equation} \label{eq:equiv_scale_m_inf}
L \big( t_{\infty}^{\lambda}(N) \big) \asymp m_{\infty}(N) \quad \text{as $N \to \infty$.}
\end{equation}

\item For all $\lambda \geq 0$ and $K > 0$, there exists $\bar{\delta}_4 = \bar{\delta}_4(\lambda, K) > 0$ such that: for all $N \geq 1$, $n \leq K m_\infty(N)$, and $t \in \big[ t_{\infty}^{-\lambda}(N) \vee \frac{t_c}{2}, t_{\infty}^{\lambda}(N) \big]$,
\begin{equation} \label{eq:RSW_scale_m_inf}
\PP_t \big( \Ch( [0,4n] \times [0,n] ) \big) \geq \bar{\delta}_4 \quad \text{and} \quad \PP_t \big( \Ch^*( [0,4n] \times [0,n] ) \big) \geq \bar{\delta}_4.
\end{equation}

\item For all $\ve > 0$, there exists $\lambda > 0$ large enough so that: for all $N$ sufficiently large,
\begin{equation} \label{eq:small_scale_m_inf}
\frac{L \big( t_{\infty}^{-\lambda}(N) \big)}{m_{\infty}(N)} \leq \ve.
\end{equation}
\end{enumerate}
Properties (i) and (iii) both follow, using standard arguments, from \eqref{eq:def_scale_m_inf}, \eqref{eq:equiv_L}, \eqref{eq:quasi_mult} and \eqref{eq:uniform_arm_exp}. Property (ii) can then be obtained from \eqref{eq:equiv_scale_m_inf} and \eqref{eq:RSW}.

\subsubsection{Frozen clusters around $0$}

We use this near-critical parameter scale to analyze the frozen clusters surrounding $0$ which have a diameter of order $m_{\infty}(N)$.

\begin{lemma} \label{lem:disjoint_circuits_FP}
Let $K > 0$. For all $\ve > 0$, there exists $C_1 = C_1(K, \ve)$ such that for all $N$ large enough, we have: for all $n \leq \frac{K}{2} m_{\infty}(N)$,
$$\PP_N^{(\Ball_{K m_\infty(N)})} \Big( \big| \calF \setminus \calF^{(\Ball_n)} \big| \geq C_1 \log \Big( \frac{K m_{\infty}(N)}{n} \Big) \Big) \leq \ve.$$
Moreover, the same conclusion holds (with a possibly larger $C_1$) for frozen percolation with modified boundary rules.
\end{lemma}
Note that for the process in $\Ball_{K m_\infty(N)}$, $\calF \setminus \calF^{(\Ball_n)}$ denotes the set of frozen clusters (at time $t = \infty$) surrounding $0$ and intersecting $\Ann_{n, K m_{\infty}(N)}$. In particular, $\calF \setminus \calF^{(\Ball_n)} \supseteq \calF^{(\Ann_{n, K m_{\infty}(N)})}$, the set of frozen clusters surrounding $0$ and entirely contained in $\Ann_{n, K m_{\infty}(N)}$.

\begin{proof}[Proof of Lemma \ref{lem:disjoint_circuits_FP}]
We first consider frozen percolation with ``original'' boundary rules. We start by claiming the following. In $\Ball_{K m_{\infty}(N)}$ (recall that $K$ is fixed, but it can be arbitrarily large), it is possible to find $\lambda = \lambda(K, \ve) > 0$ large enough so that: with probability at least $1 - \frac{\ve}{2}$, there is no occupied cluster with volume at least $N$ (so no vertex has frozen yet) at time $t_{\infty}^{-\lambda}(N)$.

In order to prove this claim, let $\kappa > 0$ (we explain how to choose it later), consider all the (horizontal and vertical) rectangles of the form
$$\kappa m_{\infty}(N) \cdot \big( (i, j) + [0, 2] \times [0, 1] \big) \quad \text{and} \quad \kappa m_{\infty}(N) \cdot \big( (i, j) + [0, 1] \times [0, 2] \big) \quad (i,j \in \ZZ)$$
intersecting the box $\Ball_{K m_{\infty}(N)}$, and introduce the event $\calN^*(K m_{\infty}(N),\kappa m_{\infty}(N))$ that there exists a vacant crossing in the ``difficult direction'' in each of these rectangles.

We have: for all $N \geq 1$ and $t \in [0,t_c)$,
\begin{equation} \label{eq:net_upper_bd}
\PP_t \big( \calN^*(K m_{\infty}(N),\kappa m_{\infty}(N)) \big) \geq 1 - c \Big( \frac{K m_{\infty}(N)}{\kappa m_{\infty}(N)} \Big)^2 e^{-c' \frac{\kappa m_{\infty}(N)}{L(t)}}
\end{equation}
for some universal constants $c, c' > 0$. Indeed, the event $\calN^*(K m_{\infty}(N),\kappa m_{\infty}(N))$ involves crossings in of order $\big( \frac{K m_{\infty}(N)}{\kappa m_{\infty}(N)} \big)^2$ rectangles, each with side lengths $\kappa m_{\infty}(N)$ and $2\kappa m_{\infty}(N)$, so \eqref{eq:net_upper_bd} follows directly from \eqref{eq:exp_decay}.

\begin{figure}[t]
\begin{center}

\includegraphics[width=.72\textwidth]{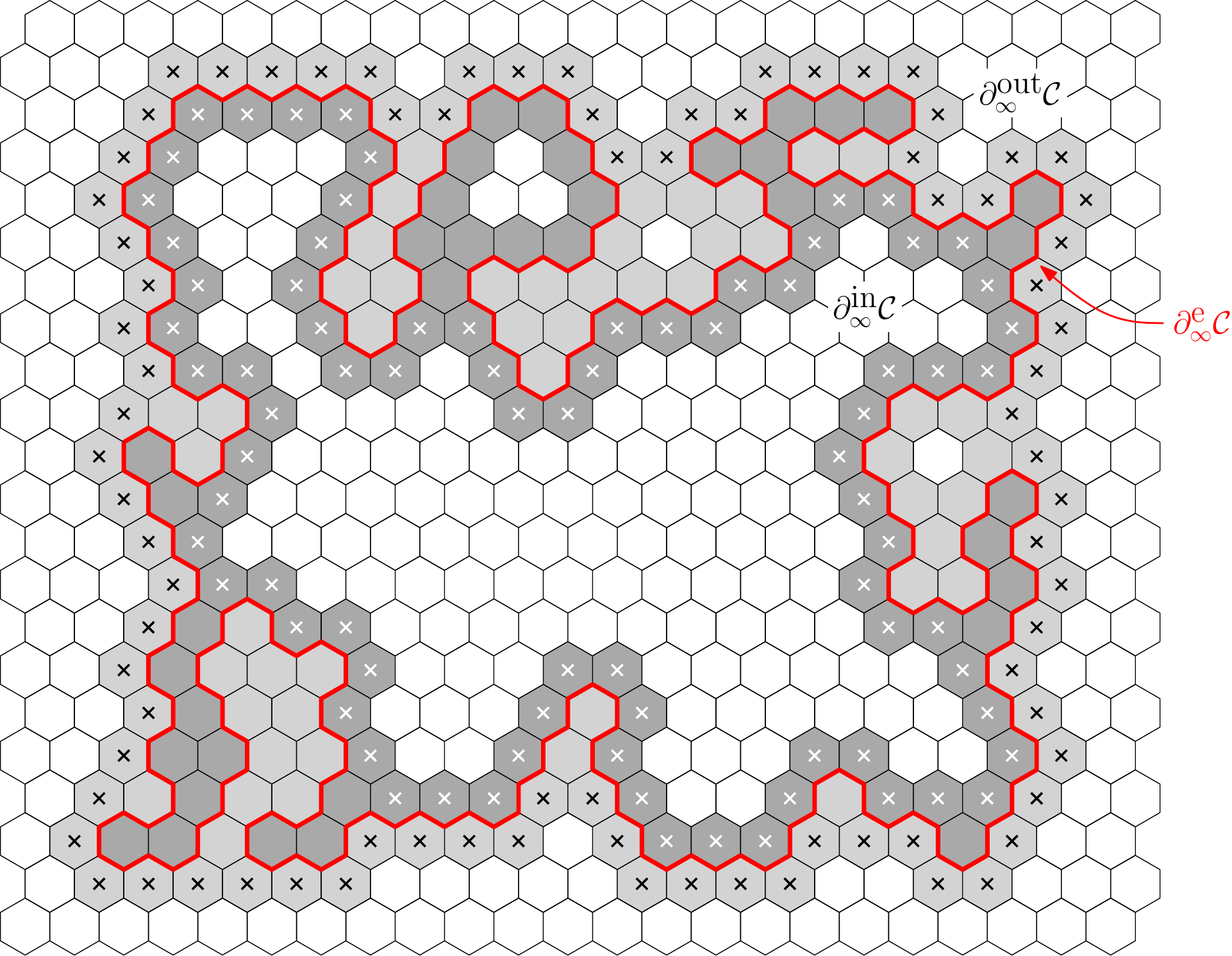}
\caption{\label{fig:boundaries} This figure uses the usual representation of site percolation on $\TT$ as face percolation on the dual hexagonal lattice (that is, we draw a hexagon around each vertex of $V_{\TT}$). It depicts the external inner and outer vertex boundaries of an occupied cluster $\cluster$ surrounding $0$, at time $\tau$: $\din_{\infty} \cluster$ (in dark gray) and $\dout_{\infty} \cluster$ (in light gray), respectively. Note that they consist of $\tau$-occupied and $\tau$-vacant sites (resp.), and they each contain a circuit surrounding $0$ (marked in white and black, resp.). The external edge boundary $\de_{\infty} \cluster$ can be drawn as a circuit on the hexagonal lattice (in red).}

\end{center}
\end{figure}

For the percolation configuration inside $\Ball_{K m_{\infty}(N)}$, the event $\calN^*(K m_{\infty}(N), \kappa m_{\infty}(N))$ implies the existence of a vacant connected set $\calN^*$ such that all the connected components of its complement have a diameter at most $2 \kappa m_{\infty}(N)$. At time $t_{\infty}^{-\lambda}(N)$ ($\leq t_c$), the probability that one of these ``cells'' contains an occupied cluster with volume at least $N$ is thus at most:
\begin{equation} \label{eq:exists_large_cluster}
C_1 \Big( \frac{K}{\kappa} \Big)^2 \PP_{t_c}\Big( \big|\lclus_{\Ball_{2 \kappa m_{\infty}(N)}}\big| \geq x \cdot (2 \kappa m_{\infty}(N))^2 \pi_1(2 \kappa m_{\infty}(N)) \Big) \leq \frac{C_2}{\kappa^2} e^{-C_3 x}
\end{equation}
for some constants $C_2 = C_2(K)$ and $C_3 > 0$ (using \eqref{eq:largest_cluster_exp_tail}), with
$$x = \frac{N}{(2 \kappa m_{\infty}(N))^2 \pi_1(2 \kappa m_{\infty}(N))}.$$
We have
$$x \geq \frac{C_4}{\kappa^2} \cdot \frac{\pi_1(m_{\infty}(N))}{\pi_1(2 \kappa m_{\infty}(N))} \geq C'_4 \kappa^{-\frac{91}{48} + \frac{1}{10}}$$
for some universal constants $C_4, C'_4 > 0$. Indeed, the first inequality follows by noting that $(m_{\infty}(N))^2 \pi_1(m_{\infty}(N)) \asymp N$, from \eqref{eq:def_t_hat_FP}, \eqref{eq:def_t_infty_FP}, \eqref{eq:equiv_theta}, and $m_{\infty} = L(t_{\infty})$, and the second inequality comes from \eqref{eq:uniform_arm_exp} (with $\sigma = (o)$). We can thus choose $\kappa = \kappa(K, \ve) > 0$ sufficiently small so that the right-hand side of \eqref{eq:exists_large_cluster} is at most $\frac{\ve}{4}$. For this particular choice of $\kappa$, we can then find $\lambda = \lambda(K, \ve) > 0$ large enough so that at time $t = t_{\infty}^{-\lambda}(N)$, the right-hand side of \eqref{eq:net_upper_bd} is at least $1 - \frac{\ve}{4}$ (using \eqref{eq:small_scale_m_inf}). This completes the proof of the claim.

Consider then an occupied cluster $\cluster$ surrounding $0$ and intersecting $\Ann_{n,K m_{\infty}(N)}$, which freezes at some time $\tau > t_{\infty}^{-\lambda}(N)$. We may assume that $\cluster$ does not intersect $\din \Ball_{K m_{\infty}(N)}$ (since at most one cluster surrounding $0$ does). Hence, at time $\tau$, its external outer boundary $\dout_{\infty} \cluster$ (see Figure \ref{fig:boundaries}) is contained in $\Ball_{K m_{\infty}(N)}$, and it is made of vertices which are vacant in the frozen percolation process (in $\Ball_{K m_{\infty}(N)}$). Each such vertex $v$ is either vacant in the pure birth process, or it lies along the outer boundary of a cluster $\cluster'$ which froze at an earlier time $\tau' < \tau$, which means that $v$ was vacant at time $\tau'$ in the pure birth process (see Figure \ref{fig:frozen_circuits}), and necessarily $\tau' > t_{\infty}^{-\lambda}(N)$ (since at this time, nothing has frozen yet). In both cases, $v$ has to be vacant in the pure birth process at time $t_{\infty}^{-\lambda}(N)$ (we want to stress that here, we use our particular choice of ``boundary rules'' for frozen percolation). From $\dout_{\infty} \cluster$, we can extract a circuit which is $t_{\infty}^{-\lambda}(N)$-vacant, is contained in $\Ball_{K m_{\infty}(N)}$ (from $\dout_{\infty} \cluster \subseteq \Ball_{K m_{\infty}(N)}$), and intersects $\Ann_{n, K m_{\infty}(N)}$ (since it surrounds $\cluster$, which itself intersects $\Ann_{n,K m_{\infty}(N)}$).

From the reasoning in the previous paragraph, we deduce that
$$\big| \calF \setminus \calF^{(\Ball_n)} \big| \leq 1 + \big| \calF^*_{t_{\infty}^{-\lambda}(N)} \big|,$$
where for each $t \geq 0$, $|\calF^*_t|$ denotes the maximal number of disjoint $t$-vacant circuits (in the pure birth process) surrounding $0$, contained in $\Ball_{K m_{\infty}(N)}$ and intersecting $\Ann_{n,K m_{\infty}(N)}$.

\begin{figure}[t]
\begin{center}

\includegraphics[width=.5\textwidth]{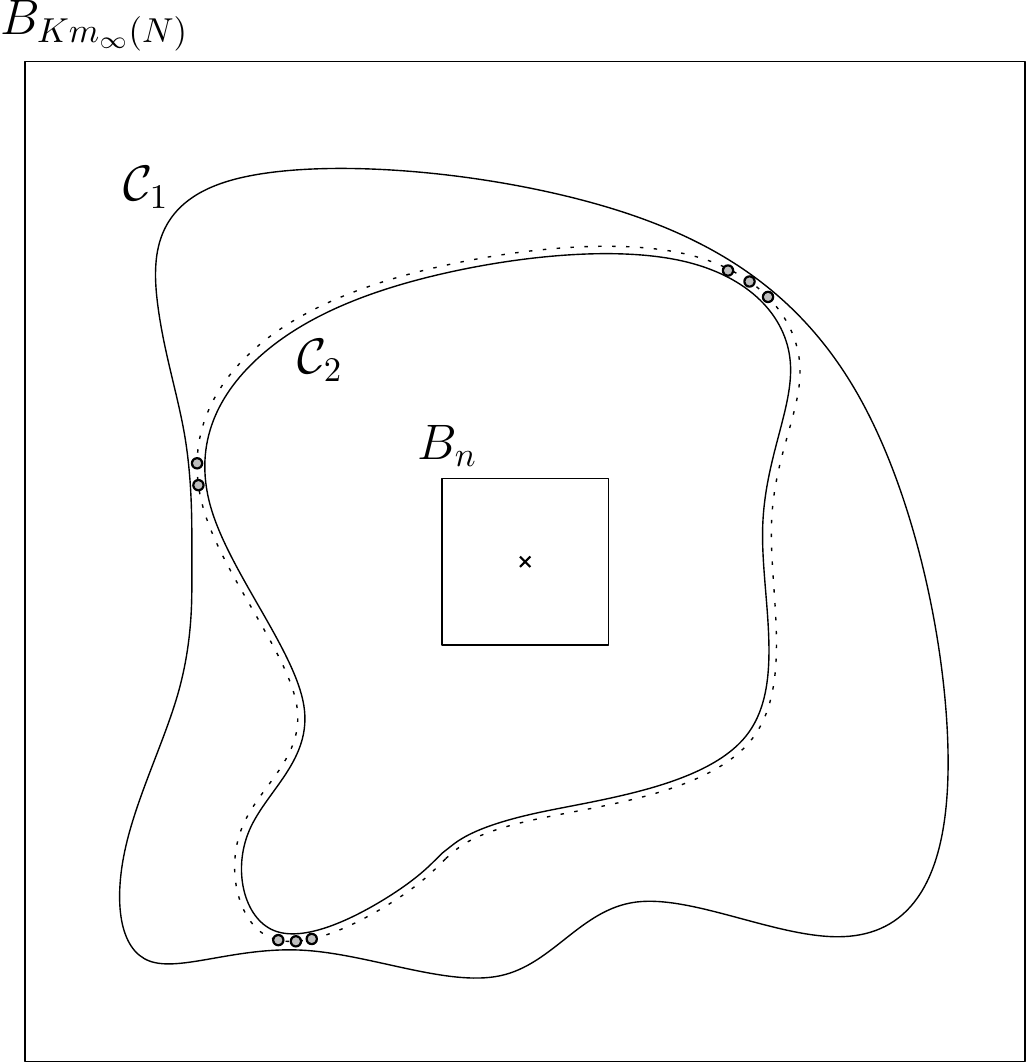}
\caption{\label{fig:frozen_circuits} The two occupied circuits $\circuit_1$ and $\circuit_2$ freeze at times $\tau_1$ and $\tau_2$, respectively. If for instance $\tau_1 < \tau_2$, the vacant vertices along the outer boudary of $\circuit_2$ may still be vacant at time $\tau_2$ in the pure birth process, but not necessarily: they can also be adjacent to a frozen cluster. For example, the grey vertices lie on the outer boundary of $\circuit_1$, so in the pure birth process, they were vacant at time $\tau_1$, but they may be occupied at time $\tau_2$. In both cases, in the pure birth process, these vertices were vacant at the earlier time $\tau_1$. Moreover, the claim ensures that with probability at least $1 - \frac{\ve}{2}$, all freezing times are larger than $t_{\infty}^{-\lambda}(N)$, for some well-chosen $\lambda = \lambda(K, \ve) > 0$.}

\end{center}
\end{figure}

From the RSW-type estimate provided by \eqref{eq:RSW_scale_m_inf} (combined with \eqref{eq:equiv_scale_m_inf}), we can deduce, using standard arguments, that for some $C_5 = C_5(K, \ve) > 0$,
\begin{equation}
\PP \Big( \big| \calF^*_{t_{\infty}^{-\lambda}(N)} \big| \geq C_5 \log \Big( \frac{K m_{\infty}(N)}{n} \Big) \Big) \leq \frac{\ve}{2}.
\end{equation}
This establishes Lemma \ref{lem:disjoint_circuits_FP} for the process with original boundary rules.

In the case of modified boundary rules, the proof proceeds essentially in the same way but it requires more care. One additional difficulty in this case is that when a cluster $\cluster$ freezes, say at some time $\tau$, one cannot necessarily produce a $t_{\infty}^{-\lambda}(N)$-vacant circuit from its external outer boundary $\dout_{\infty} \cluster$. Indeed, some of the vertices along $\dout_{\infty} \cluster$ may now be frozen at time $\tau$, so it is possible that they were already occupied at time $t_{\infty}^{-\lambda}(N)$.

However, we claim that it is possible to construct a $t_{\infty}^{-\lambda}(N)$-vacant path $\gamma$ surrounding $0$ by adding the vertices in $\din_{\infty} \cluster$, i.e. from $(\dout_{\infty} \cluster) \cup (\din_{\infty} \cluster)$. For this purpose, consider any frozen vertex $v \in \dout_{\infty} \cluster$, which froze at an earlier time $\tau' \in (t_{\infty}^{-\lambda}(N), \tau)$: then necessarily, all its neighbors which are not frozen at time $\tau$ (which means that in particular, they did not freeze together with $v$ at time $\tau'$) were vacant at time $\tau'$, so at time $t_{\infty}^{-\lambda}(N)$ as well. Hence, all the ($\tau$-occupied) neighbors of $v$ on $\din_{\infty} \cluster$ were vacant at time $t_{\infty}^{-\lambda}(N)$.

Using this observation, it is then easy to construct the path $\gamma$ by following the edge boundary $\de \cluster$, as illustrated on Figure~\ref{fig:boundaries_modified}: for each edge $e = \{v,v'\} \in E$ with $v \in \dout_{\infty} \cluster$ and $v' \in \din_{\infty} \cluster$, use $v$ or $v'$ depending on whether $v$ is vacant or frozen at time $\tau$ (resp.). It is easy to convince oneself that this procedure produces a $t_{\infty}^{-\lambda}(N)$-vacant path $\gamma$ surrounding $0$, from which one can extract a circuit around $0$ (by removing the ``loops'').

\begin{figure}[t]
\begin{center}

\includegraphics[width=.75\textwidth]{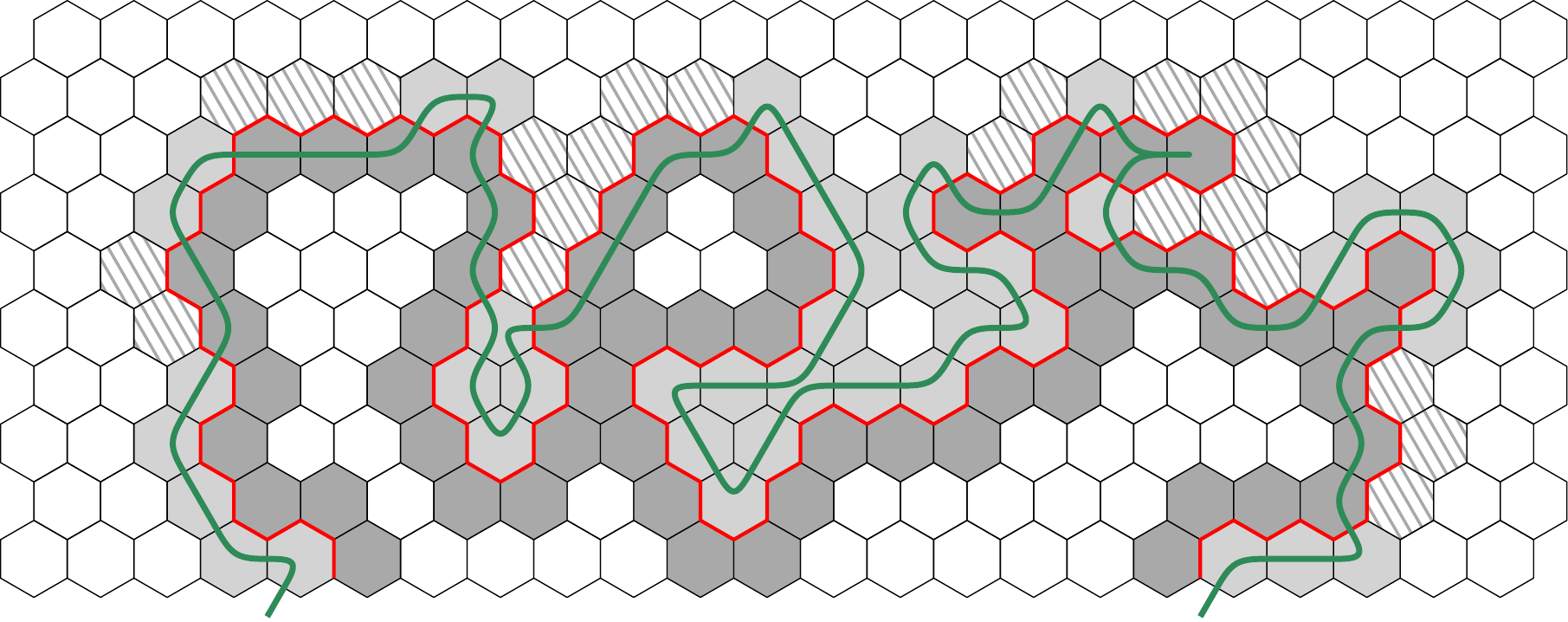}
\caption{\label{fig:boundaries_modified} This figure illustrates the procedure to construct a $t_{\infty}^{-\lambda}(N)$-vacant path $\gamma$ (in green) from $(\dout_{\infty} \cluster) \cup (\din_{\infty} \cluster)$. The dark gray and light gray vertices are respectively $\tau$-occupied and $\tau$-vacant, as before, while the dashed vertices are frozen at time $\tau$. The path $\gamma$ is obtained by following $\de \cluster$, drawn in red.}

\end{center}
\end{figure}

At this point, one has to be a bit careful, since the $t_{\infty}^{-\lambda}(N)$-vacant circuits constructed from distinct frozen clusters may intersect (or even completely coincide). However, one can use that the frozen clusters are ``nested'': if we list them as $\cluster_1, \ldots, \cluster_k$, starting from the outside, then the circuits corresponding to $\cluster_i$ for odd $i$ are disjoint. Hence, the number of frozen clusters in $\Ball_{K m_{\infty}(N)}$ which surround $0$ and intersect $\Ann_{n,K m_{\infty}(N)}$ (and do not intersect $\din \Ball_{K m_{\infty}(N)}$) is at most $2 \, \big| \calF^*_{t_{\infty}^{-\lambda}(N)} \big|$, which allows us to conclude the proof in the same way as for the original process.
\end{proof}

\subsection{Forest fires} \label{sec:scale_minf_FF}

\subsubsection{Near-critical parameter scale}

We now study forest fire processes in boxes with a side length of order $m_{\infty}$. To this end, as in Section~\ref{sec:def_nc_parameter}, we introduce a near-critical parameter scale, writing $t_{\infty}(\zeta) = t_c + \ve_{\infty}(\zeta)$.
\begin{definition}
For $\lambda \in \RR$ and $\zeta > 0$, let
\begin{equation} \label{eq:def_scale_m_inf_FF}
t_{\infty}^{\lambda}(\zeta) := t_c + \lambda \, \ve_{\infty}(\zeta).
\end{equation}
\end{definition}
This parameter scale satisfies analogous properties to \eqref{eq:equiv_scale_m_inf} and \eqref{eq:small_scale_m_inf}, for exactly the same reasons.

Obtaining an analog of \eqref{eq:RSW_scale_m_inf} requires the use of Proposition~\ref{prop:crossing_impurities}. In order to study the FFWoR process in a box $\Ball_{K m_{\infty}(\zeta)}$, $K > 0$, we will first stop ignitions at time $t_{\infty}^{-\lambda} = t_c - \lambda \ve_{\infty}$, and consider percolation with impurities where the parameter is
\begin{equation} \label{eq:param_m_nc_FF}
m = m_{\lambda}(\zeta) = L(t_{\infty}^{-\lambda}(\zeta)).
\end{equation}
Using the analogs of \eqref{eq:equiv_scale_m_inf} and \eqref{eq:small_scale_m_inf}, $m \asymp m_{\infty}$ for any fixed $\lambda \in \RR \setminus \{ 0 \}$, and it can be made arbitrarily small compared to $m_{\infty}$, uniformly in $\zeta > 0$, by considering $\lambda$ large enough. In particular, we can choose $\lambda$ so that
\begin{equation} \label{eq:upper_bd_ups}
\upsilon(m) = c' \bigg( \frac{m}{m_{\infty}} \bigg)^{\frac{9}{8}} \leq \frac{1}{2} C^{-1},
\end{equation}
where $C = C(c, \gamma, 1) > 0$ is from Proposition~\ref{prop:crossing_impurities}. From now on, we fix such a value, and we denote it by $\lambda_0$ to stress that it is a universal constant. Indeed, recall that the constants $c$, $c'$ and $\gamma = \frac{9}{8}$ are either absolute, or considered to be so, hence $\lambda_0$ as well (see the last sentence of Section~\ref{sec:stoch_domin_FF}).

From Proposition~\ref{prop:crossing_impurities} (with $p = p(t_{\infty}^{-\lambda_0})$, so that $L(p) = m$), we have: for all $n \leq m$,
$$\PPh_{p(t_{\infty}^{-\lambda_0})}^{(m)} \big( \Ch([0,2n] \times [0,n]) \big) \geq \big( 1 - C \upsilon(m) \big) \, \PP_{p(t_{\infty}^{-\lambda_0})} \big( \Ch([0,2n] \times [0,n]) \big) \geq \frac{1}{2} \cdot \delta_4$$
uniformly in $\zeta > 0$ (for the second inequality, we used the upper bound \eqref{eq:upper_bd_ups} on $\upsilon(m)$, as well as \eqref{eq:RSW}). We deduce, for any $K > 0$, the existence of $\tilde{\delta} = \tilde{\delta}(K) > 0$ so that
\begin{equation} \label{eq:RSW_FF}
\PPh_{p(t_{\infty}^{-\lambda_0})}^{(m)} \big( \Cv^*( [0,4n] \times [0,n] ) \big) \leq 1 - \tilde{\delta},
\end{equation}
uniformly in $\zeta > 0$ and $n \leq K m_{\infty}(\zeta)$. Indeed, this follows by combining a bounded number of occupied crossings in rectangles, thanks to the FKG inequality ($K$ is fixed and $\lambda_0$ is universal, so $m = L(t_{\infty}^{-\lambda_0}) \asymp K m_{\infty}$).

\subsubsection{Burnt clusters around $0$}

We are now in a position to derive an analog of Lemma~\ref{lem:disjoint_circuits_FP} for the FFWoR process.

\begin{lemma} \label{lem:disjoint_circuits_FF}
Let $K > 0$. For all $\ve > 0$, there exists $C_1 = C_1(K, \ve)$ such that for all $\zeta$ small enough, we have: for all $n \leq \frac{K}{2} m_{\infty}(\zeta)$,
$$\PP_\zeta^{(\Ball_{K m_\infty(\zeta)})} \Big( \big| \calF \setminus \calF^{(\Ball_n)} \big| \geq C_1 \log \Big( \frac{K m_{\infty}(\zeta)}{n} \Big) \Big) \leq \ve.$$
\end{lemma}

\begin{proof}[Proof of Lemma \ref{lem:disjoint_circuits_FF}]
Observe that by standard arguments, for any $\ve > 0$, \eqref{eq:RSW_FF} implies the following. With a probability at least $1 - \ve$, the maximal number of disjoint vacant circuits (and so of disjoint vacant clusters) in $\Ann_{n, K m_{\infty}(\zeta)}$, at time $t_{\infty}^{-\lambda_0}$, is at most $C_1 \log(\frac{K m_{\infty}(\zeta)}{n})$, for some $C_1 = C_1(K,\ve)$ (recall that $\lambda_0$ is considered as a universal constant). More precisely, this holds true for the process with impurities (with parameter $m = m_{\lambda_0}(\zeta)$, see \eqref{eq:param_m_nc_FF}), and also, using in addition Lemma~\ref{lem:stoch_domin}, for the FFWoR process $\sigma^{[t_{\infty}^{-\lambda_0}]} = \big( \sigma^{[t_{\infty}^{-\lambda_0}]}(t) \big)_{t \geq 0}$, with ignitions stopped at time $t_{\infty}^{-\lambda_0}$, if we consider instead circuits made of vacant \emph{and} burnt vertices.

The proof then relies on a similar reasoning as for frozen percolation with \emph{modified} boundary rules. More specifically, if some cluster $\cluster$ burns at a time $\tau > t_{\infty}^{-\lambda_0}$, then its external outer boundary $\dout_{\infty} \cluster$ is composed both of vertices which are vacant, so were already vacant at time $t_{\infty}^{-\lambda_0}$, and of burnt vertices. Any such burnt vertex $v$ was either already burnt at time $t_{\infty}^{-\lambda_0}$, and thus in $\sigma^{[t_{\infty}^{-\lambda_0}]}$, or if it burnt at some time $\tau' \in (t_{\infty}^{-\lambda_0}, \tau)$, then any neighboring vertex $v' \sim v$ belonging to $\din_{\infty} \cluster$ had to be vacant at that time $\tau'$, so at the earlier time $t_{\infty}^{-\lambda_0}$. Indeed, $v'$ cannot burn at time $\tau'$, nor be already burnt at time $\tau'^-$, since it is occupied just before time $\tau$ ($> \tau'$): here we use the absence of recoveries. We can thus proceed in the same way as for frozen percolation with modified boundary rules, and extract from $(\dout_{\infty} \cluster) \cup (\din_{\infty} \cluster)$ a circuit (as illustrated earlier on Figure~\ref{fig:boundaries_modified}, for frozen percolation) which, at time $t_{\infty}^{-\lambda_0}$ in $\sigma^{[t_{\infty}^{-\lambda_0}]}$, contains only vacant and burnt vertices. This shows that the number of disjoint burnt clusters is at most twice the number of disjoint vacant / burnt circuits in $\sigma^{[t_{\infty}^{-\lambda_0}]}(t_{\infty}^{-\lambda_0})$, which allows us to use the observation above, based on \eqref{eq:RSW_FF}, and completes the proof.
\end{proof}

\begin{remark} \label{rem:FFWR_minf}
If we considered instead the FFWR process, potential recoveries between time $t_c - \lambda_0 \ve_{\infty}$ and time $\tau$ would be problematic, causing the argument above to break down. Indeed, it is possible that two neighboring vertices $v$ and $v'$ as above are both occupied at time $t_c - \lambda_0 \ve_{\infty}$ and burn together at some time $\tau' \in (t_c - \lambda_0 \ve_{\infty}, \tau)$, and that the vertex $v$ then becomes occupied again during $(\tau',\tau)$.
\end{remark}

\section{Avalanches for frozen percolation} \label{sec:proof_FP}

We now prove our results for frozen percolation: Theorem~\ref{thm:mainFP} in Section~\ref{sec:proof_FP1}, and Proposition~\ref{prop:circuits} in Section~\ref{sec:proof_FP2}. In order to simplify notation, we denote
\begin{equation} \label{eq:nu_FP}
\aval = \aval^{\textrm{FP}} = \frac{96}{5} \quad \text{and} \quad \avalnb = \avalnb^{\textrm{FP}} = \frac{1}{\log \aval^{\textrm{FP}}}
\end{equation}
in this section (only).

\subsection{Proof of Theorem \ref{thm:mainFP}} \label{sec:proof_FP1}

The main ideas of the proof of Theorem~\ref{thm:mainFP} are as follows.
\begin{itemize}
\item[(1)] First, in Step~1, we decompose the box $\Ball_{Km_\infty(N)}$ into two regions: $\Lambda^\#$ and $\Ball_{Km_\infty(N)}\setminus \Lambda^\#$, where $\Lambda^\#$ is ``nice'' and its radius is of order $\frac{m_\infty(N)}{(\log{N})^{\beta}}$, for some well-chosen $\beta > 0$ (we have to take it sufficiently small). This is achieved by using an intermediate result from Section~7 of \cite{BKN2015}, which allows one to compare the process in $\Ball_{Km_\infty(N)}$ to the process in domains with a radius which is both $\ll m_\infty(N)$, but also sufficiently large (as a function of $N$). If $\beta$ is chosen small enough, Lemma~\ref{lem:disjoint_circuits_FP} ensures that there are at most $\frac{\ve}{2} \log \log N$ frozen clusters in $\Ball_{Km_\infty(N)}\setminus \Lambda^\#$.

\item[(2)] Hence, there remains to analyze the frozen clusters surrounding $0$ in $\Lambda^\#$, and we show that there are approximately $\avalnb \log\log{N}$ of them. For this purpose, we follow closely the ``dynamics'' of the successively freezing clusters, using an iterative construction inspired by \cite{BN2015}. This construction, developed in Steps~2 and 3, produces a (finite) sequence of nested domains $\Lambda^{(0)} = \Lambda^\# \supseteq \Lambda^{(1)} \supseteq \Lambda^{(2)} \supseteq \ldots$ (and a corresponding sequence of times $t_c < \tau^{(0)} < \tau^{(1)} < \tau^{(2)} < \ldots$). These domains are such that for every $i$, $\Ball_{r^{(i)}} \subseteq \Lambda^{(i)} \subseteq \Ball_{R^{(i)}}$, with $1 \leq r^{(i)} \leq R^{(i)}$, and $\Lambda^{(i)} \setminus \Lambda^{(i+1)}$ contains exactly one frozen cluster surrounding $0$ (in the final configuration). The ``uncertainty'' $\frac{R^{(i)}}{r^{(i)}}$ on the precise location of their boundaries increases as we move down scales. The main difficulty in the proof is to control this uncertainty, and show that it does not grow too quickly. In particular, the construction uses crucially a property of separation of scales, i.e. that $r^{(i)} \gg R^{(i+1)}$ for each $i$. We have to make sure that this property remains valid along the way, even after of order $\log \log N$ steps. This is the reason why we start from the scale $\frac{m_\infty(N)}{(\log{N})^{\beta}}$, instead of a scale of order $m_\infty(N)$ directly (see Remark~\ref{rem:separation_scales}).

\item[(3)] Finally, we study the end of the iteration in Step~4, and show that at most one additional frozen cluster can form around $0$ after the last step of the scheme. We then explain quickly how to combine Steps~1--4, and conclude the proof, in Step~5.
\end{itemize}
We now present these stages in detail.

\begin{proof}[Proof of Theorem \ref{thm:mainFP}]

We let $\ve, \eta > 0$, and $\delta = 0.001$. Without loss of generality, we assume that $N \geq e^{e^e}$, so that $\log \log \log N$ is well-defined, and it is $\geq 1$.

\bigskip

\textbf{Step 1}: We first need to introduce the notion of ``stopping sets'', as in \cite{BKN2015}. It refers to the natural analog for percolation of stopping times: we say that a set of vertices $\Lambda^\#$ is a stopping set if for any finite $\Lambda$, the event $\{ \Lambda^\# = \Lambda \}$ is measurable with respect to the percolation configuration in $\Lambda^c$. This property allows us to condition on the value of $\Lambda^\#$ while leaving the percolation configuration inside it unaffected, in other words to consider $\Lambda^\#$ as being fixed.

By Proposition~7.2 in \cite{BKN2015}, there exist $c_2>c_1>0$ and $M>0$ (depending only on $\eta$) such that: for all $t>t_c$ with $L(t)\leq \frac{m_\infty(N)}{M}$, we can construct a simply connected stopping set $\Lambda^\#$ so that with probability $\geq 1- \frac{\eta}{4}$, the following two properties hold true.
\begin{enumerate}[(i)]
\item We have $\Ball_{c_1L(t^\#)}\subseteq \Lambda^\# \subseteq \Ball_{c_2L(t^\#)}$, for some $t^\# > t$ satisfying $t^\# \leq \hat{\hat{\hat{t}}}$.

\item For frozen percolation on the whole lattice $\TT$, its restriction to $\Lambda^\#$ coincides with frozen percolation in $\Lambda^\#$ directly.
\end{enumerate}
We denote by $\calE_{(i)}$ and $\calE_{(ii)}$ the corresponding events, so that $\PP_N^{(\TT)}(\calE_{(i)} \cap \calE_{(ii)}) \geq 1 - \frac{\eta}{4}$. In addition, we also note that the proof in \cite{BKN2015} yields the same conclusions for frozen percolation in $\Ball_{Km_\infty(N)}$, instead of $\TT$.

Here in particular, if $L(t) = \frac{m_\infty(N)}{(\log{N})^{\alpha}}$, we get from Lemma~\ref{lem:one_iteration} (applied three times) that:
\begin{equation} \label{eq:t_sharp}
C_1 \frac{m_\infty(N)}{(\log{N})^{\alpha (\aval + \delta)^3}} \leq L(t^\#) \leq \frac{m_\infty(N)}{(\log{N})^{\alpha}}.
\end{equation}
Let
$$\ul{n}^{(0)}(N) := c_1 C_1 \frac{m_\infty(N)}{(\log{N})^{\alpha (\aval + \delta)^3}} \: \: \big( \leq c_1 L(t^\#) \big) \quad \text{and} \quad \ol{n}^{(0)}(N) := c_2 L(t) = c_2 \frac{m_\infty(N)}{(\log{N})^{\alpha}}.$$
It follows from Lemma~\ref{lem:disjoint_circuits_FP} that the number of frozen clusters surrounding $0$, contained in $\Ball_{K m_\infty(N)}$, and intersecting $\Ann_{\ul{n}^{(0)}(N), K m_{\infty}(N)}$, is at most
\begin{equation} \label{eq:step1_FP}
C_2 \log \bigg( \frac{K m_{\infty}(N)}{\ul{n}^{(0)}(N)} \bigg) \leq C_2 \alpha (\aval + \delta)^3 \log \log N + O(1)
\end{equation}
with probability $\geq 1 - \frac{\eta}{4}$, for some $C_2$ depending only on $\eta$ and $K$. We can thus make sure that this number is at most $\frac{\ve}{2} \log \log N$ for all sufficiently large $N$, by choosing $\alpha$ small enough so that $C_2 \alpha (\aval + \delta)^3 \leq \frac{\ve}{4}$. From now on, we fix such an $\alpha = \alpha(\ve,\eta,K)$, and we assume that $N$ is large enough so that $(\log{N})^{\alpha} \geq M$.

\bigskip

\textbf{Step 2}: We now claim that there exists $N_0 = N_0(\ve, \eta, K)$ such that for all $\Lambda^\#$ as in Step~1, the following holds. Consider frozen percolation with parameter $N$ in $\Lambda^\#$, and recall that $\calF$ denotes the set of frozen clusters surrounding the origin (in the final configuration). We have
\begin{equation} \label{eq:claim_Lambda}
\text{for all $N \geq N_0$,} \quad \PP_N^{(\Lambda^\#)} \bigg( \frac{|\calF|}{\log \log N} \in \Big( \avalnb - \frac{\ve}{2}, \avalnb + \frac{\ve}{2} \Big) \bigg) \geq 1 - \frac{\eta}{4}.
\end{equation}
This claim is proved in Steps~2--4. Once it is established, Theorem~\ref{thm:mainFP} will follow, as we explain briefly in Step~5.

For future use, note the following fact about near-critical percolation, which follows easily from \eqref{eq:exp_decay} and \eqref{eq:connection_infinity}. There exist universal constants $\kappa_3, \kappa_4 > 0$ such that: for all $p > p_c$ and $n \geq 1$,
\begin{equation} \label{eq:a-priori}
\PP_p \big( \circuitevent ( \Ann_{n, 2n} ) \cap \{ \dout \Ball_n \lra \infty \} \big) \geq 1 - \kappa_3 e^{- \kappa_4 \frac{n}{L(p)}}.
\end{equation}

We show the claim by iterating a percolation construction, in a similar way as for the proof of Theorem~2 in \cite{BN2015}. For that, we define by induction two (deterministic) sequences $(r^{(i)})_{i \geq 0}$ and $(R^{(i)})_{i \geq 0}$, with $r^{(i)} \leq R^{(i)}$ for all $i \geq 0$.

For some $\Lambda^\#$, $t^\#$, and $0 < c_1 < c_2$ as above, we start from $\Lambda^{(0)} := \Lambda^\#$, and $r^{(0)} < R^{(0)}$ defined by
$$r^{(0)} := c_1 L(t^\#) \quad \text{and} \quad R^{(0)} := c_2 L(t^\#)$$
(so that $\Ball_{r^{(0)}} \subseteq \Lambda^{(0)} \subseteq \Ball_{R^{(0)}}$). It follows from \eqref{eq:t_sharp} that
\begin{equation} \label{eq:R0_minf}
\frac{c}{(\log{N})^{\beta}} \leq \frac{r^{(0)}}{m_\infty(N)} < \frac{R^{(0)}}{m_\infty(N)} \leq \frac{c'}{(\log{N})^{\alpha}}
\end{equation}
for some constants $c, c' > 0$ (depending only on $\eta$), and $\alpha, \beta = \alpha (\aval + \delta)^3 > 0$ (which depend on $\ve, \eta, K$).

If $r^{(i)} \leq R^{(i)}$ are determined for some $i \geq 0$, we define the times
\begin{equation} \label{eq:t_def}
\ul{t}^{(i)} := \next_N \bigg( \frac{1}{1 + \delta} \, \frac{9}{10} r^{(i)} \bigg) \quad \text{and} \quad \ol{t}^{(i)} := \next_N \bigg( \frac{1}{1 - \delta} R^{(i)} \bigg).
\end{equation}
We have clearly $t_c < \ol{t}^{(i)} \leq \ul{t}^{(i)} \leq \infty$, since $\next_N$ is nonincreasing, so $L(\ul{t}^{(i)}) \leq L(\ol{t}^{(i)})$, and they satisfy
\begin{equation} \label{eq:theta_t}
c_{\TT} \bigg( \frac{9}{10} r^{(i)} \bigg)^2 \theta(\ul{t}^{(i)}) = N (1 + \delta) \quad \text{and} \quad c_{\TT} \big( R^{(i)} \big)^2 \theta(\ol{t}^{(i)}) = N (1 - \delta)
\end{equation}
(from \eqref{eq:def_next_FP}), unless, of course, $\ul{t}^{(i)} = \infty$ (for the first equality) or $\ol{t}^{(i)} = \infty$ (for the second one). Note that $\ul{t}^{(i)}$ and $\ol{t}^{(i)}$ may (and, in fact, will) be equal to $\infty$ after some point. We introduce
$$j := \min \big\{ i \geq 1 \: : \: r^{(i)} < 3 c_{\TT}^{-\frac{1}{2}} \sqrt{N} \big\} - 1 \quad \text{and} \quad J := \min \big\{ i \geq 1 \: : \: R^{(i)} < 3 c_{\TT}^{-\frac{1}{2}} \sqrt{N} \big\} - 1.$$
Obviously $j \leq J$. We show later that they are finite, and differ by at most $1$.

We then let
\begin{equation} \label{eq:r_def}
r^{(i+1)} := \frac{1}{(\log \log N)^{24}} \, L \big( \ul{t}^{(i)} \big) \quad \text{and} \quad R^{(i+1)} = \frac{4}{\kappa_4} \big( \log \log \log N \big) \, L \big( \ol{t}^{(i)} \big),
\end{equation}
where $\kappa_4$ is as in \eqref{eq:a-priori} (we assume that $\frac{4}{\kappa_4} (\log \log \log N) \geq 1$). Note that $0 \leq r^{(i+1)} \leq R^{(i+1)} < \infty$.

Let $\ve' > 0$, that we explain how to choose later (as a function of $\ve$ only). In the remainder of the proof, all the constants appearing are allowed to depend on $\ve$ (or, equivalently, $\ve'$), $\eta$, and $K$, but not on $i$.

First, it follows immediately from \eqref{eq:compar_minf}, together with \eqref{eq:t_def} and \eqref{eq:r_def}, that if $R^{(i)} \geq 3 c_{\TT}^{-\frac{1}{2}} \sqrt{N}$,
\begin{equation} \label{eq:comp_R_minf}
C_1 \bigg( \frac{R^{(i)}}{m_{\infty}} \bigg)^{\aval + \ve'} \leq \frac{R^{(i+1)}}{m_{\infty}} \leq C_2 \bigg( \frac{R^{(i)}}{m_{\infty}} \bigg)^{\aval - \ve'} \big( \log \log \log N \big),
\end{equation}
for some $C_1, C_2 > 0$. By induction, starting from \eqref{eq:R0_minf}, we deduce that for all $i = 0, \ldots, J+1$,
\begin{equation} \label{eq:r_estimate}
\bigg( \frac{c_1}{\log N} \bigg)^{\beta (\aval + \ve')^i} \leq \frac{R^{(i)}}{m_{\infty}} \leq \bigg( \frac{c_2 \log \log \log N}{(\log N)^{\alpha}} \bigg)^{(\aval - \ve')^i} \leq \bigg( \frac{c'_2}{\log N} \bigg)^{\frac{\alpha}{2} (\aval - \ve')^i},
\end{equation}
where $c_1, c_2, c'_2 > 0$.

Second, using \eqref{eq:compar_minf} and \eqref{eq:R0_minf} again (and $r^{(0)} = \frac{c_1}{c_2} R^{(0)}$), we have
$$\frac{R^{(1)}}{r^{(0)}} = c_3 \, \frac{R^{(1)}}{m_{\infty}} \cdot \bigg( \frac{R^{(0)}}{m_{\infty}} \bigg)^{-1} \leq c'_3 \big( \log \log \log N \big) \bigg( \frac{R^{(0)}}{m_{\infty}} \bigg)^{\aval - \ve' - 1} \leq \bigg( \frac{c''_3}{\log N} \bigg)^{\xi}$$
for some $c_3, c'_3, c''_3, \xi > 0$ (it is important, here, that $\aval > 1$). By applying repeatedly Lemma \ref{lem:one_iteration}, we get that: for all $i = 0, \ldots, J$,
\begin{equation} \label{eq:r_estimate2}
\frac{R^{(i+1)}}{r^{(i)}} \leq \bigg( \frac{c_4 (\log \log N)^{25}}{(\log N)^{\xi}} \bigg)^{(\aval - \ve')^i} \leq \bigg( \frac{c'_4}{\log N} \bigg)^{\frac{\xi}{2} (\aval - \ve')^i}
\end{equation}
(so in particular $j \geq J-1$). This implies that for all $N$ large enough: for all $i = 0, \ldots, J$, $R^{(i+1)} < \frac{1}{10} r^{(i)}$.

Finally, $J$ can easily be estimated from \eqref{eq:r_estimate} (and so $j$, which is either equal to $J$ or $J-1$). On the one hand, we deduce from $R^{(J)} \geq 3 c_{\TT}^{-\frac{1}{2}} \sqrt{N}$ that
$$\frac{R^{(J)}}{m_{\infty}} \geq c N^{- \alpha}$$
where $c > 0$ and $\alpha = \frac{48}{91} - \frac{1}{2} + \delta > 0$ are universal (using \eqref{eq:exp_t_infty_FP}). Since
$$\frac{R^{(J)}}{m_{\infty}} \leq \bigg( \frac{c'_2}{\log N} \bigg)^{\frac{\beta'}{2} (\aval - \ve')^J}$$
(from \eqref{eq:r_estimate}), we obtain (for all $N$ large enough)
\begin{align}
J & \leq \frac{1}{\log(\aval - \ve')} \bigg( \log \bigg( \frac{2}{\beta'} \bigg) + \log \bigg( \frac{\alpha \log N - \log c}{\log \log N - \log c'_2} \bigg) \bigg) \nonumber \\[2mm]
& = \frac{1}{\log(\aval - \ve')} \big( \log \log N + O(\log \log \log N) \big). \label{eq:J_est1}
\end{align}
On the other hand, $R^{(J+1)} < 3 c_{\TT}^{-\frac{1}{2}} \sqrt{N}$ so
$$\frac{R^{(J+1)}}{m_{\infty}} \leq c' N^{- \alpha'},$$
with $c' > 0$ and $\alpha' = \frac{48}{91} - \frac{1}{2} - \delta > 0$, from which we get
\begin{equation} \label{eq:J_est2}
J + 1 \geq \frac{1}{\log(\aval + \ve')} \big( \log \log N + O(\log \log \log N) \big).
\end{equation}
Recall that $\avalnb = \avalnb^{\textrm{FP}} = \frac{1}{\log \aval}$ (see \eqref{eq:nu_FP}), and choose $\ve'$ small enough so that $\frac{1}{\log(\aval - \ve')} \leq \avalnb + \frac{\ve}{3}$ and $\frac{1}{\log(\aval + \ve')} \geq \avalnb - \frac{\ve}{3}$. By combining \eqref{eq:J_est1} and \eqref{eq:J_est2}, we obtain
\begin{equation} \label{eq:bounds_J}
\avalnb - \frac{\ve}{3} + o(1) \leq \frac{J}{\log \log N} \leq \avalnb + \frac{\ve}{3} + o(1).
\end{equation}

\bigskip

\textbf{Step 3}: We show that with high probability, the number of frozen clusters surrounding $0$ in the final configuration is roughly $J$ (more precisely, between $J$ and $J+2$). For that, we consider the following events, for all $i =0, \ldots, j$, involving the percolation configuration in the pure birth process at times $\ol{t}^{(i)}$ and $\ul{t}^{(i)}$:
\begin{itemize}
\item $E_1^{(i)} := \Big\{$at $\ol{t}^{(i)}$, $\Big| \lclus_{\Ball_{R^{(i)}}} \Big| < N \Big\}$,

\item $E_2^{(i)} := \Big\{$at $\ul{t}^{(i)}$, $\Big| \lclus_{\Ball_{\frac{9}{10}r^{(i)}}} \Big| \geq N \Big\}$,

\item $E_3^{(i)} := \Big\{$at $\ul{t}^{(i)}$, $\Big| \lclus_{\Ball_{\frac{8}{10}r^{(i)}}} \Big| < N$ and $\Big| \lclus_{\Ann_{\frac{7}{10}r^{(i)}, r^{(i)}}} \Big| < N \Big\}$,

\item $E_4^{(i)} := \Big\{$at $\ol{t}^{(i)}$, $\circuitevent \big( \Ann_{\frac{7}{10}r^{(i)}, \frac{8}{10}r^{(i)}} \big)$ and $\circuitevent \big( \Ann_{\frac{9}{10}r^{(i)},r^{(i)}} \big)$ occur$\Big\}$,

\item $E_5^{(i)} := \Big\{$at $\ol{t}^{(i)}$, $\circuitevent \big( \Ann_{\frac{1}{2} R^{(i+1)}, R^{(i+1)}} \big)$ occurs and $\dout \Ball_{\frac{1}{2} R^{(i+1)}} \lra \infty \Big\}$,

\item $E_6^{(i)} := \Big\{$at $\ul{t}^{(i)}$, $\circuitevent^* \big( \Ann_{r^{(i+1)}, \frac{1}{2} R^{(i+1)}} \big)$ occurs$\Big\}$.
\end{itemize}

\begin{figure}[t]
\begin{center}

\includegraphics[width=.88\textwidth]{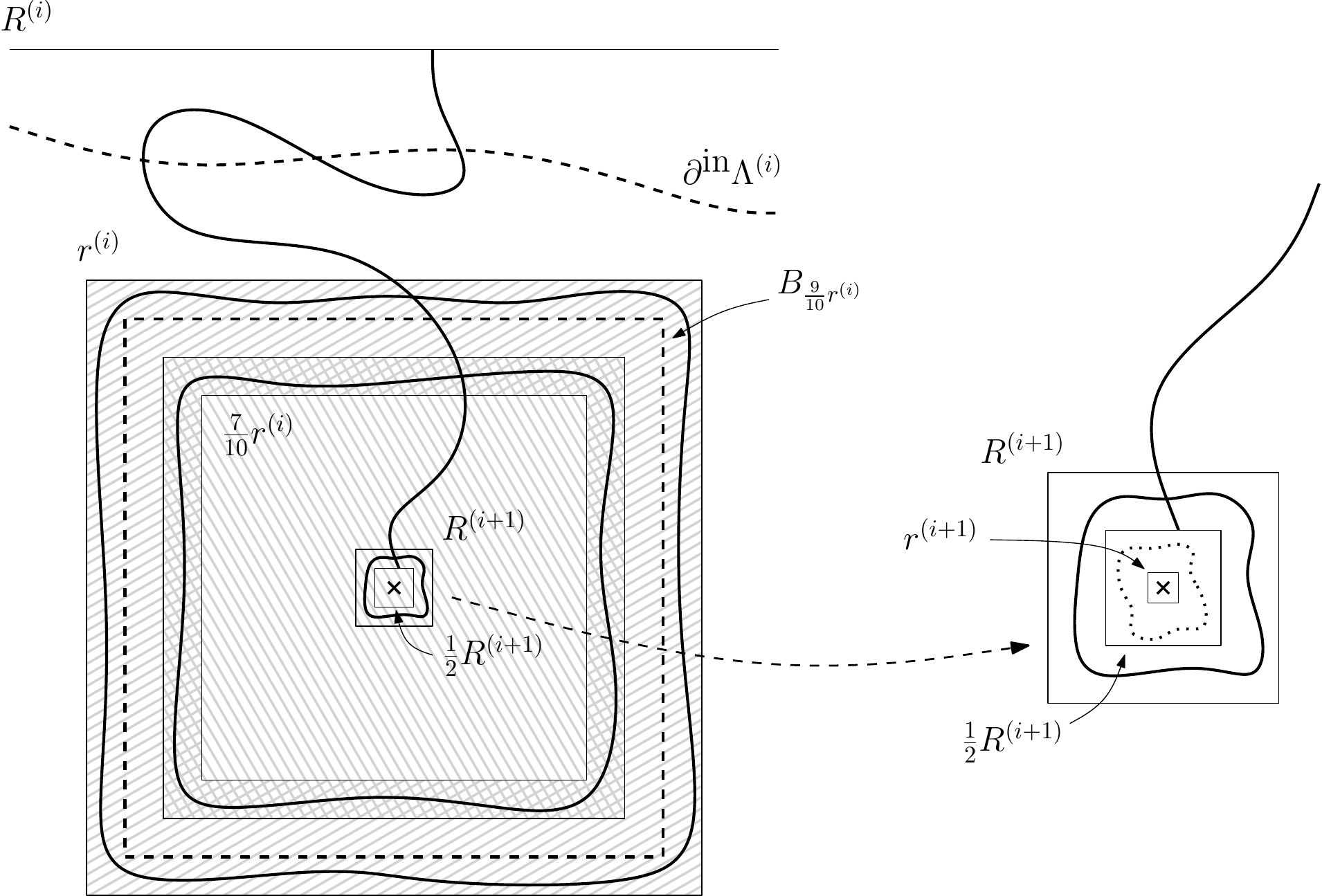}
\caption{\label{fig:iteration_FP} This figure shows the $i$th step of the construction producing the successive frozen clusters surrounding $0$. In a domain $\Lambda^{(i)}$ with $\Ball_{r^{(i)}} \subseteq \Lambda^{(i)} \subseteq \Ball_{R^{(i)}}$, we introduce times $t_c < \ol{t}^{(i)} < \ul{t}^{(i)}$ such that the next freezing occurs between these two times. At time $\ol{t}^{(i)}$, nothing has frozen yet, and the solid paths are occupied, while the dotted path is vacant at time $\ul{t}^{(i)}$. Vertices in the box $\Ball_{\frac{9}{10}r^{(i)}}$ are used to trigger the next freezing, and we make sure that the two regions $\Ball_{\frac{8}{10}r^{(i)}}$ and $\Ann_{\frac{7}{10}r^{(i)}, r^{(i)}}$ are too small to freeze separately.}

\end{center}
\end{figure}

We first show that these events have a high probability as $N \to \infty$, uniformly in $i$. More precisely, we prove that for some $\kappa > 0$ (depending on $\ve,\eta,K$, but not on $i$), we have: for all $i =0, \ldots, j$,
\begin{equation} \label{eq:unif_lower_bound}
\PP(E_k^{(i)}) \geq 1 - \frac{\kappa}{(\log \log N)^2} \quad (k = 1, \ldots, 6).
\end{equation}
This will allow us to simply use the union bound in the end, over the $j+1$ successive steps.

Let $i \in \{0, \ldots, j\}$. We start with $E_1^{(i)}$, $E_2^{(i)}$ and $E_3^{(i)}$, for which \eqref{eq:unif_lower_bound} can be obtained from \eqref{eq:largest_cluster_quant} (in combination with \eqref{eq:theta_t}). First,
\begin{equation}
\PP(E_1^{(i)}) \geq 1 - C \, \frac{L(\ol{t}^{(i)})}{R^{(i)}} = 1 - C \, \bigg( \frac{4}{\kappa_4} \big( \log \log \log N \big) \bigg)^{-1} \frac{R^{(i+1)}}{R^{(i)}}.
\end{equation}
(we also used \eqref{eq:r_def}). We can then write
\begin{equation}
\frac{R^{(i+1)}}{R^{(i)}} \leq \frac{R^{(i+1)}}{r^{(i)}} \leq \bigg( \frac{c_5}{\log N} \bigg)^{\frac{\xi}{2} (\aval - \ve')^i}
\end{equation}
(using $r^{(i)} \leq R^{(i)}$, and \eqref{eq:r_estimate2} for the second inequality), so
\begin{equation}
\PP(E_1^{(i)}) \geq 1 - \frac{C'}{(\log N)^{\frac{\xi}{2}}}
\end{equation}
for some $C' > 0$. Similarly,
\begin{equation}
\PP(E_2^{(i)}) \geq 1 - C \, \frac{L(\ul{t}^{(i)})}{\frac{9}{10}r^{(i)}} \quad \text{and} \quad \PP(E_3^{(i)}) \geq 1 - C \, \frac{L(\ul{t}^{(i)})}{\frac{8}{10}r^{(i)}} - C \, \frac{L(\ul{t}^{(i)})}{r^{(i)}}.
\end{equation}
We then use \eqref{eq:r_def}, $r^{(i+1)} \leq R^{(i+1)}$, and \eqref{eq:r_estimate2} to get
\begin{equation} \label{eq:ineq_L_r}
\frac{L(\ul{t}^{(i)})}{r^{(i)}} = (\log \log N)^{24} \, \frac{r^{(i+1)}}{r^{(i)}} \leq (\log \log N)^{24} \, \frac{R^{(i+1)}}{r^{(i)}} \leq (\log \log N)^{24} \, \bigg( \frac{c_5}{\log N} \bigg)^{\frac{\xi}{2} (\aval - \ve')^i},
\end{equation}
and thus the existence of $C'' > 0$ so that
\begin{equation}
\PP(E_2^{(i)}), \,\, \PP(E_3^{(i)}) \geq 1 - \frac{C''}{(\log N)^{\frac{\xi}{4}}}.
\end{equation}
For $E_4^{(i)}$, we deduce from several applications of \eqref{eq:exp_decay} that
\begin{equation}
\PP(E_4^{(i)}) \geq 1 - \kappa'_1 e^{-\kappa'_2 \frac{\frac{1}{10} r^{(i)}}{L(\ol{t}^{(i)})}} \geq 1 - \kappa'_1 e^{-\kappa''_2 (\log N)^{\frac{\xi}{2}}} \geq 1 - \frac{\tilde{C}}{\log N}
\end{equation}
for some $\tilde{C} > 0$. For the second inequality, we used that (again, from \eqref{eq:r_def} and \eqref{eq:r_estimate2})
\begin{equation}
\frac{r^{(i)}}{L(\ol{t}^{(i)})} = \frac{4}{\kappa_4} \big( \log \log \log N \big) \frac{r^{(i)}}{R^{(i+1)}} \geq \tilde{C}' (\log N)^{\frac{\xi}{2}}.
\end{equation}
For $E_5^{(i)}$ we use \eqref{eq:a-priori} directly:
\begin{equation}
\PP(E_5^{(i)}) \geq 1 - \kappa_3 e^{- \kappa_4 \frac{\frac{1}{2} R^{(i+1)}}{L(\ol{t}^{(i)})}} = 1 - \kappa_3 e^{- 2 \log \log \log N} = 1 - \frac{\kappa_3}{(\log \log N)^2}.
\end{equation}
Finally, using \eqref{eq:r_def} and then \eqref{eq:near_critical_arm}, we have
\begin{align}
\PP(E_6^{(i)}) & \geq \PP_{p(\ul{t}^{(i)})} \big( \circuitevent^* \big( \Ann_{(\log \log N)^{-24} L(\ul{t}^{(i)}), L(\ul{t}^{(i)})} \big) \big) \nonumber\\[1.5mm]
& = 1 - \PP_{p(\ul{t}^{(i)})} \big( \arm_1 \big( \Ann_{(\log \log N)^{-24} L(\ul{t}^{(i)}), L(\ul{t}^{(i)})} \big) \big) \nonumber\\[1.5mm]
& \geq 1 - c \, \pi_1 \big( (\log \log N)^{-24} L(\ul{t}^{(i)}), L(\ul{t}^{(i)}) \big). \label{eq:P_E6a}
\end{align}
Hence, it follows from \eqref{eq:uniform_arm_exp} that
\begin{equation} \label{eq:P_E6b}
\PP(E_6^{(i)}) \geq 1 - c' (\log \log N)^{- 24 \cdot ( \frac{5}{48} - \delta )} \geq 1 - \frac{c''}{(\log \log N)^2}.
\end{equation}
We have thus established \eqref{eq:unif_lower_bound}, from which we can deduce that
\begin{equation}
\PP \Bigg( \bigcup_{\substack{0 \leq i \leq j\\ 1 \leq k \leq 6}} E_k^{(i)} \Bigg) \geq 1 - (j+1) \cdot 6 \cdot \frac{\kappa}{(\log \log N)^2} \geq 1 - \frac{\eta}{8}
\end{equation}
for all $N$ large enough (using the union bound, combined with \eqref{eq:bounds_J} and the fact that $j \leq J$).

We observe that if the events $E_1^{(i)}$--$E_6^{(i)}$ occur simultaneously for some $i \in \{0, \ldots, j\}$, then for any simply connected domain $\Lambda^{(i)}$ with $\Ball_{r^{(i)}} \subseteq \Lambda^{(i)} \subseteq \Ball_{R^{(i)}}$, frozen percolation in $\Lambda^{(i)}$ has the following properties.
\begin{itemize}
\item Exactly one cluster surrounding $0$ freezes in the time interval $(\ol{t}^{(i)}, \ul{t}^{(i)}]$, at a time $\tau^{(i)}$. This cluster is also the unique cluster intersecting $\Ball_{\frac{9}{10}r^{(i)}}$ to freeze in this time interval.

\item Moreover, this frozen cluster surrounds $0$, which is left in a simply connected ``island'' $\Lambda^{(i+1)}$ whose boundary is contained in $\Ann_{ r^{(i+1)},  R^{(i+1)}}$.
\end{itemize}
We want to stress that our construction does not preclude other frozen clusters to emerge in $\Lambda^{(i)}$ before time $\ul{t}^{(i)}$, but it ensures that such clusters cannot surround $0$.

This implies that if all events $(E_k^{(i)})_{0 \leq i \leq j, 1 \leq k \leq 6}$ occur, then exactly $j+1$ clusters surrounding the origin freeze successively, at times $\tau^{(i)} \in (\ol{t}^{(i)}, \ul{t}^{(i)}]$, $0 \leq i \leq j$. After the ($j+1$)th freezing, the origin is left in a domain $\Lambda^{(j+1)}$ satisfying $\Ball_{r^{(j+1)}} \subseteq \Lambda^{(j+1)} \subseteq \Ball_{R^{(j+1)}}$.

Before we conclude this step, we have to mention that a small technical issue was ``swept under the rug'': in the last step $i = j$, it is well possible that $r^{(j+1)} < 1$ (and even much smaller), in which case \eqref{eq:P_E6a} and \eqref{eq:P_E6b} do not make much sense. If this happens, we simply discard $E_6^{(j+1)}$, i.e. let it be the whole sample space, so that \eqref{eq:unif_lower_bound} clearly holds. We keep the other events $E_1^{(j+1)}$--$E_5^{(j+1)}$, so the only consequence in this last step is that the vacant circuit in $\Ann_{r^{(i+1)}, \frac{1}{2} R^{(i+1)}}$ does not necessarily exist (the rest is unchanged). Hence, we can deduce that $\Lambda^{(j+1)} \subseteq \Ball_{R^{(j+1)}}$, but $\Lambda^{(j+1)}$ might be empty.

\bigskip

\textbf{Step 4}: We now explain how to end the iteration scheme, and obtain the total number of frozen clusters in $\Lambda^\#$. For that, we show that at most one extra frozen cluster surrounding $0$ can arise after time $\tau^{(j)}$ (in $\Lambda^{(j+1)}$). We need to distinguish two cases. Recall that by definition, $r^{(j)} \geq 3 c_{\TT}^{-\frac{1}{2}} \sqrt{N}$ and $r^{(j+1)} < 3 c_{\TT}^{-\frac{1}{2}} \sqrt{N}$.
\begin{itemize}
\item \ul{Case 1}: $R^{(j+1)} < \frac{1}{2} c_{\TT}^{-\frac{1}{2}} \sqrt{N}$. We have $j=J$, and $0$ is left in an island with volume $<N$, where no cluster can thus freeze. The procedure just stops here, and there are exactly $j+1$ frozen clusters surrounding $0$ in the final configuration.

\item \ul{Case 2}: $R^{(j+1)} \geq \frac{1}{2} c_{\TT}^{-\frac{1}{2}} \sqrt{N}$. In this case, $j = J-1$ or $j = J$, and the island $\Lambda^{(j+1)}$ may have a volume $< N$ or $\geq N$ (we may assume it to be non-empty, otherwise we just stop). We perform one more step, as we explain now. Let $\ol{t}^{(j+1)}$ and $R^{(j+2)}$ be associated with $\tilde{R}^{(j+1)} := \max(R^{(j+1)}, 3 c_{\TT}^{-\frac{1}{2}} \sqrt{N})$. From $r^{(j+1)} < 3 c_{\TT}^{-\frac{1}{2}} \sqrt{N}$, we can deduce that $R^{(j+2)} \ll \sqrt{N}$ (by the same reasoning that led to \eqref{eq:r_estimate2}). We can consider the corresponding events $E_1^{(j+1)}$ and $E_5^{(j+1)}$ (but not the other events $E_2^{(j+1)}$, $E_3^{(j+1)}$, and so on). They again satisfy the lower bound \eqref{eq:unif_lower_bound}, and $E_1^{(j+1)} \cap E_5^{(j+1)}$ ensures that if one frozen cluster $\cluster$ surrounding $0$ appears, then necessarily $\cluster$ freezes after time $\ol{t}^{(j+1)}$, and $\cluster$ contains the circuit in the definition of $E_5^{(j+1)}$. In this case, $0$ is left in an island with volume $\ll N$. This proves that at most one additional frozen cluster surrounding $0$ can appear.
\end{itemize}
We deduce that in both cases, the total number of frozen clusters surrounding $0$ satisfies
$$J \leq |\calF| \leq J+2$$
with a probability at least $1 - \frac{\eta}{4}$ (by the union bound). This allows us to obtain the claim \eqref{eq:claim_Lambda}, thanks to \eqref{eq:bounds_J}.

\bigskip

\textbf{Step 5}: We finally wrap up the proof. Consider the events
$$\calE := \bigg\{ \text{for frozen percolation in } \Lambda^\#, \,\, \frac{|\calF|}{\log \log N} \in \Big( \avalnb - \frac{\ve}{2}, \avalnb + \frac{\ve}{2} \Big) \bigg\}$$
and
$$\calE' := \bigg\{ \frac{|\calF^{(\Ball_{\ol{n}^{(0)}(N)})}|}{\log \log N} \geq \avalnb - \frac{\ve}{2}, \,\, \frac{|\calF^{(\Ball_{\ul{n}^{(0)}(N)})}|}{\log \log N} \leq \avalnb + \frac{\ve}{2} \bigg\}.$$
We have
$$\PP_N^{(\Ball_{K m_\infty(N)})}\big( \calE \big) \geq \sum_{\Lambda} \PP_N^{(\Ball_{K m_\infty(N)})}\big( \calE \, | \, \Lambda^\# = \Lambda \big) \cdot \PP_N^{(\Ball_{K m_\infty(N)})}\big( \Lambda^\# = \Lambda \big),$$
where the sum is over all simply connected domains $\Lambda$ as in $\calE_{(i)}$. We have
$$\PP_N^{(\Ball_{K m_\infty(N)})}\big( \calE \, | \, \Lambda^\# = \Lambda \big) = \PP_N^{(\Lambda)}\bigg( \frac{|\calF|}{\log \log N} \in \Big( \avalnb - \frac{\ve}{2}, \avalnb + \frac{\ve}{2} \Big) \bigg) \geq 1 - \frac{\eta}{4}$$
for all $N \geq N_0(\ve, \eta, K)$, where the equality comes from the property that $\Lambda^\#$ is a stopping set (so conditioning on $\{\Lambda^\# = \Lambda\}$ leaves the process inside $\Lambda$ unaffected), and the inequality from \eqref{eq:claim_Lambda}. Hence,
$$\PP_N^{(\Ball_{K m_\infty(N)})}\big( \calE \big) \geq \bigg( 1 - \frac{\eta}{4} \bigg) \, \sum_{\Lambda} \PP_N^{(\Ball_{K m_\infty(N)})}\big( \Lambda^\# = \Lambda \big) \geq \bigg( 1 - \frac{\eta}{4} \bigg) \, \PP_N^{(\Ball_{K m_\infty(N)})}\big( \calE_{(i)} \big) \geq 1 - \frac{\eta}{2},$$
from which we deduce
\begin{equation} \label{eq:end_proof_FP1}
\PP_N^{(\Ball_{K m_\infty(N)})}\big( \calE' \big) \geq \PP_N^{(\Ball_{K m_\infty(N)})}\big( \calE \cap \calE_{(i)} \cap \calE_{(ii)} \big) \geq \PP_N^{(\Ball_{K m_\infty(N)})}\big( \calE \big) - \frac{\eta}{4} \geq 1 - \frac{3 \eta}{4}.
\end{equation}
On the other hand, we know that from our choice of $\alpha$,
\begin{equation} \label{eq:end_proof_FP2}
\PP_N^{(\Ball_{K m_\infty(N)})} \bigg( \big| \calF \setminus \calF^{(\Ball_{\ul{n}^{(0)}(N)})} \big| \leq \frac{\ve}{2} \log \log N \bigg) \geq 1 - \frac{\eta}{4}
\end{equation}
(for all $N$ large enough). The desired result now follows immediately by combining \eqref{eq:end_proof_FP1} and \eqref{eq:end_proof_FP2}:
 $$\PP_N^{(\Ball_{K m_\infty(N)})} \bigg( \frac{|\calF|}{\log \log N} \in \Big( \avalnb - \frac{\ve}{2}, \avalnb + \ve \Big) \bigg) \geq 1 - \eta.$$

\end{proof}

\begin{remark} \label{rem:separation_scales}
In order to control the avalanche of successive frozen clusters, as in Step~3 above, we use the fact that the successive scales get more and more separated as $N \to \infty$, i.e. that $r^{(i)}$ (the lower bound in the $i$th step) remains $\gg R^{(i+1)}$ (the upper bound in the $(i+1)$th step) as $i$ increases. This is ensured by \eqref{eq:r_estimate2}, which is obtained inductively from $\frac{R^{(1)}}{r^{(0)}}$. This explains why we started from a scale of order $\frac{m_\infty(N)}{(\log{N})^{\beta}}$. As the reader can check, we could even start the analysis from $\frac{m_\infty(N)}{(\log \log N)^{\alpha}}$, with $\alpha$ large enough (i.e. for all $\alpha \geq \alpha_0 > 0$). In any case, we have to handle separately the scales close to $m_\infty(N)$.

We also want to stress that in order to check this separation of scales, the upper bound on $\frac{R^{(i)}}{r^{(i)}}$ obtained thanks to a repeated use of Lemma~\ref{lem:one_iteration} would not be good enough, after of order $\log \log N$ steps (because of the error $\ve'$ in the exponent). In fact, even though we expect $r^{(i)}$ to stay much closer to $R^{(i)}$ than to $R^{(i+1)}$ (as we know for $i=0$, and as we can check for the first few steps), this does not seem to follow from our reasoning.

Finally, we mention that exactly the same issues arise for forest fire processes in Section~\ref{sec:proof_FF}.
\end{remark}

\subsection{Proof of Proposition \ref{prop:circuits}} \label{sec:proof_FP2}

By modifying slightly the previous proof, it is possible to obtain Proposition \ref{prop:circuits}: with high probability, the number of disjoint frozen circuits surrounding the origin grows at least as a power law in $N$. In particular, it is much larger than the number of frozen clusters, which is of order $\log \log N$.

\begin{proof}[Proof of Proposition \ref{prop:circuits}]

Let $\delta = 10^{-10}$ and $\eta > 0$. As in Step 1 of the proof of Theorem \ref{thm:mainFP}, we use Proposition~7.2 in \cite{BKN2015}, but starting from $t > t_c$ with $L(t) = \frac{m_\infty(N)}{N^{\delta}}$ (instead of $L(t) = \frac{m_\infty(N)}{(\log{N})^{\alpha}}$, for some $\alpha > 0$). This proves the existence of a simply connected stopping set $\Lambda^\#$ such that with probability $\geq 1- \frac{\eta}{3}$, the following two properties are satisfied.
\begin{enumerate}[(i)]
\item $\Ball_{c_1L(t^\#)}\subseteq \Lambda^\# \subseteq \Ball_{c_2L(t^\#)}$, where $t \leq t^\# \leq \hat{\hat{\hat{t}}}$.

\item For frozen percolation on the whole lattice $\TT$, its restriction to $\Lambda^\#$ coincides with frozen percolation in $\Lambda^\#$ directly.
\end{enumerate}
From Lemma \ref{lem:one_iteration},
\begin{equation} \label{eq:rem_start}
C_1 \frac{m_\infty(N)}{N^{\delta (\aval + \delta)^3}} \leq L(t^\#) \leq \frac{m_\infty(N)}{N^{\delta}}.
\end{equation}
We first let
$$r^{(0)} := c_1 L(t^\#) \quad \text{and} \quad R^{(0)} := c_2 L(t^\#),$$
and perform one step as in the proof above: define $\ul{t}^{(1)}$ and $\ol{t}^{(1)}$, and then $r^{(1)}$ and $R^{(1)}$ (see \eqref{eq:t_def}, \eqref{eq:r_def}). Observe that by \eqref{eq:rem_start}, \eqref{eq:compar_minf}, and \eqref{eq:exp_t_infty_FP}, $r^{(1)} \gg \sqrt{N}$. We also have
$$\frac{r^{(0)}}{R^{(1)}} \geq c N^{\xi}$$
for some universal $c, \xi > 0$. Indeed,
$$\frac{r^{(0)}}{R^{(1)}} = \frac{c_1}{c_2} \cdot \frac{R^{(0)}}{m_{\infty}} \cdot \frac{m_{\infty}}{R^{(1)}}\geq c' \bigg( \frac{m_{\infty}}{R^{(0)}} \bigg)^{\aval - \delta - 1} \big( \log \log \log N \big)^{-1} \geq c'' (N^{\delta})^{\aval - 2 \delta - 1},$$
where we used \eqref{eq:comp_R_minf} and \eqref{eq:rem_start} for the first and second inequalities, respectively. We thus let $\xi := \delta (\aval - 2 \delta - 1) > 0$.

We then consider the same events $E_1^{(0)}$--$E_6^{(0)}$ as before, together with the additional event
$$E_7^{(0)} := \Big\{\text{at } \ol{t}^{(0)}, \,\, \circuitevent \big( \Ann_{k N^{\frac{\xi}{3}} R^{(1)}, (k+1) N^{\frac{\xi}{3}} R^{(1)}} \big) \text{ occurs for all } k=1, \ldots, N^{\frac{\xi}{3}} \Big\}.$$
From \eqref{eq:exp_decay},
$$\PP(E_7^{(0)}) \geq 1 - \sum_{k=1}^{N^{\frac{\xi}{3}}} (c k) \cdot \kappa_1 e^{-\kappa_2 N^{\frac{\xi}{3}}} \geq 1 - \kappa'_1 N^{\frac{2 \xi}{3}} e^{-\kappa_2 N^{\frac{\xi}{3}}}.$$
Indeed, a circuit in $\Ann_{k N^{\frac{\xi}{3}} R^{(1)}, (k+1) N^{\frac{\xi}{3}} R^{(1)}}$ can be obtained from occupied crossings in rectangles with side lengths $4 N^{\frac{\xi}{3}} R^{(1)}$ and $N^{\frac{\xi}{3}} R^{(1)}$: at most $c k$ of them, where $c$ is universal. Hence,
\begin{equation}
\PP \Bigg( \bigcup_{1 \leq k \leq 7} E_k^{(0)} \Bigg) \geq 1 - \frac{\eta}{3}
\end{equation}
for all $N$ large enough (for $E_k^{(0)}$, $k = 1, \ldots, 6$, we use similar computations as for \eqref{eq:unif_lower_bound}).

The extra event $E_7^{(0)}$ ensures the existence of disjoint occupied circuits $(\circuit_k)_{1 \leq k \leq N^{\frac{\xi}{3}}}$, each of them intersecting the occupied path from $\dout \Ball_{\frac{1}{2} R^{(1)}}$ to $\infty$ provided by $E_5^{(0)}$. Hence, all these circuits $\circuit_k$ freeze, since by construction, they are contained in the first cluster surrounding $0$ to freeze (which happens in the time interval $(\ol{t}^{(0)}, \ul{t}^{(0)}]$). This allows us to conclude.

\end{proof}

\begin{remark} \label{rem:other_lattices3}
The proof of Theorem \ref{thm:mainFP} uses crucially \eqref{eq:ratioL}, which relies on the exact value of the one-arm exponent $\alpha_1$. On other two-dimensional lattices (as considered in Remark~\ref{rem:other_lattices}), we can still derive a weaker, but non-trivial, statement. The reader can check that the observations from Remark~\ref{rem:other_lattices2} yield the following: for all $K, \ve > 0$,
$$\PP_N^{(\Ball_{K m_\infty(N)})} \bigg( \frac{|\calF|}{\log \log N} \in \bigg( \frac{1}{\log \frac{2}{\alpha}} - \ve, \frac{1}{\log 4} + \ve \bigg) \bigg) \stackrel[N \to \infty]{}{\longrightarrow} 1.$$
Moreover, the proof of Proposition \ref{prop:circuits} applies in this case as well, after minor adjustments, leading to the same result. For both proofs, we also need to observe that the reasoning in Section~7 of \cite{BKN2015} does not use any of the fine properties of near-critical percolation which are known only for site percolation on $\TT$ at the moment (contrary to earlier sections in that paper, where, in particular, the full scaling limit of near-critical percolation \cite{GPS2018a} is used).
\end{remark}

\section{Avalanches for forest fires} \label{sec:proof_FF}

In this final section, we establish the results for the FFWoR process, namely Theorem~\ref{thm:mainFF} (Section \ref{sec:proof_FF1}) and Proposition \ref{prop:large_burnt_clusters} (Section~\ref{sec:proof_FF2}). We let
\begin{equation}
\aval = \aval^{\textrm{FF}} = \frac{96}{41} \quad \text{and} \quad \avalnb = \avalnb^{\textrm{FF}} = \frac{1}{\log \aval^{\textrm{FF}}}.
\end{equation}

\subsection{Preliminaries}

We will make use of the following ``uniform'' versions of \eqref{eq:exp_theta_L}.

\begin{lemma} \label{lem:unif_theta_L}
For all $\ve > 0$, we have: for all $0 < \ve_1 \leq \ve_2$,
\begin{equation} \label{eq:uniform_theta}
C_1 \bigg( \frac{\ve_1}{\ve_2} \bigg)^{\frac{5}{36} + \ve} \leq \frac{\theta(t_c+\ve_1)}{\theta(t_c+\ve_2)} \leq C_2 \bigg( \frac{\ve_1}{\ve_2} \bigg)^{\frac{5}{36} - \ve},
\end{equation}
and similarly
\begin{equation} \label{eq:uniform_L}
C'_1 \bigg( \frac{\ve_1}{\ve_2} \bigg)^{-\frac{4}{3} + \ve} \leq \frac{L(t_c+\ve_1)}{L(t_c+\ve_2)} \leq C'_2 \bigg( \frac{\ve_1}{\ve_2} \bigg)^{-\frac{4}{3} - \ve},
\end{equation}
for some constants $C_1, C_2, C'_1, C'_2 > 0$ that depend only on $\ve$.
\end{lemma}

\begin{proof}[Proof of Lemma \ref{lem:unif_theta_L}]
First, \eqref{eq:uniform_L} follows from
$$\ve_1 \, L(t_c + \ve_1)^2 \pi_4(L(t_c + \ve_1)) \asymp 1 \asymp \ve_2 \, L(t_c + \ve_2)^2 \pi_4(L(t_c + \ve_2))$$
(using \eqref{eq:equiv_L}), and
$$\frac{\pi_4(L(t_c + \ve_1))}{\pi_4(L(t_c + \ve_2))} \asymp \pi_4(L(t_c + \ve_2), L(t_c + \ve_1)),$$
(from \eqref{eq:quasi_mult}), which can be estimated using \eqref{eq:uniform_arm_exp} (with $\sigma = (ovov)$). For \eqref{eq:uniform_theta}, it then suffices to write
$$\frac{\theta(t_c+\ve_1)}{\theta(t_c+\ve_2)} \asymp \frac{\pi_1(L(t_c+\ve_1))}{\pi_1(L(t_c+\ve_2))} \asymp \pi_1(L(t_c + \ve_2), L(t_c + \ve_1))$$
(using \eqref{eq:equiv_theta} and \eqref{eq:quasi_mult}), apply \eqref{eq:uniform_arm_exp} again (now with $\sigma = (o)$), and use \eqref{eq:uniform_L}.
\end{proof}

\subsection{Proof of Theorem \ref{thm:mainFF}} \label{sec:proof_FF1}

We turn to Theorem~\ref{thm:mainFF}. Again, we proceed by iterating a percolation construction, now inspired by the proof of Theorem~7.2 of \cite{BN2018}. It involves events for the configuration in the pure birth process at successive times after $t_c$, as in Section \ref{sec:proof_FP}, and it also uses the existence of connections and large clusters in the presence of impurities, created by fires occurring before $t_c$ (more precisely, we stop ignitions at a corresponding sequence of times smaller than $t_c$). However, the iteration from \cite{BN2018} turns out to be too crude, and it cannot be carried out over a number of steps tending to infinity. The scheme that we present is quite involved, requiring the introduction of nine different events at every step.

As for Theorem~\ref{thm:mainFP}, the proof can be decomposed, roughly, into three successive stages.
\begin{itemize}
\item[(1)] First, in Step~1, we explain how to explore the process in $\Ball_{K m_\infty(\zeta)}$ from the boundary, in order to produce a random subdomain $\Lambda^+$ having the following property. With high probability, the FFWoR process in the bigger box $\Ball_{K m_\infty(\zeta)}$, when restricted to $\Lambda^+$, approximately coincides with the FFWoR process in this domain. For this purpose, we perform the exploration in such a way that little information is read on the birth and ignition processes inside $\Lambda^+$ (or, at least, ``deep'' inside it). This domain $\Lambda^+$ cannot be too irregular, and we ensure that its boundary is contained in an annulus $\Ann_{r,R}$, with $\frac{r}{R}$ bounded away from $0$. In addition, $\Lambda^+$ needs to have a radius slightly below $m_\infty(\zeta)$, in order to initiate the iterative construction, but not too much so that Lemma~\ref{lem:disjoint_circuits_FF} can be applied, to take care of the burnt circuits in $\Ball_{K m_\infty(\zeta)} \setminus \Lambda^+$. In the case of frozen percolation, we were able to appeal directly to a proposition from Section~7 of \cite{BKN2015}. No such results is available for forest fires, so we need to develop a new approach here.

\item[(2)] We then perform the iteration scheme itself, which is the heart of the proof. In a similar fashion as for frozen percolation, we show that in $\Lambda^+$, $(\avalnb + o(1)) \log\log \frac{1}{\zeta}$ clusters burn around $0$. We produce again a sequence of domains $\Lambda^{(0)} = \Lambda^+ \supseteq \Lambda^{(1)} \supseteq \Lambda^{(2)} \supseteq \ldots$ such that for each $i$, exactly one cluster surrounding $0$ burns in $\Lambda^{(i)} \setminus \Lambda^{(i+1)}$. In order to explain one step of the iteration, consider the FFWoR process in a domain $\Lambda$ with a radius of order $L(t)$, for some $t > t_{\infty}$. The strategy to analyze the first macroscopic burning in such a domain can be described informally as follows, where we write $\hat{t} = t_c + \ve$ ($> t$ since $t > t_{\infty}$). We use the times $t_c + \kappa_1 \ve < t_c + \kappa_2 \ve$, for some $\kappa_1, \kappa_2 > 0$ which have to be chosen sufficiently small and sufficiently large, respectively. We want to ensure that with high probability, the largest cluster at time $t_c + \kappa_1 \ve$ is not hit by lightning in $[t_c - \kappa_1 \ve, t_c + \kappa_1 \ve]$, and hit in $[t_c + \kappa_1 \ve, t_c + \kappa_2 \ve]$. This uses the definition of $\hat{t}$, and the process with ignitions stopped at time $t_c - \kappa_1 \ve$. However, the scheme turns out to be quite subtle to implement, so as not to ``lose too much'' along the way (i.e. keep the property of ``separation of scales'', analogously to frozen percolation), and non-trivial technicalities arise.

\item[(3)] In the final stage (Step~4), we explain how to terminate the iterative procedure, by showing that at most two extra clusters burn around $0$. We then conclude the proof.
\end{itemize}

\begin{proof}[Proof of Theorem~\ref{thm:mainFF}]

Let $\ve, \eta > 0$, and $\delta = 0.001$. We can assume, without loss of generality, that $\zeta > 0$ is small enough so that $\log \log \log \frac{1}{\zeta}$ is well-defined, and it is $\geq 1$.

\bigskip

\textbf{Step 1}: In this first step, we explain how to relate the full-plane FFWoR process to the process in a suitable (random) finite domain, with a diameter sufficiently smaller than $m_\infty(\zeta)$. We achieve this through an exploration ``from outside'', producing a random island $\Lambda^+$ which contains $0$, while leaving the processes inside $\Lambda^+$ more or less untouched.

However, performing such explorations for forest fire processes is somewhat trickier than for Bernoulli or frozen percolation. In these latter processes, when e.g. exploring the vertices connected to the boundary of a box (to determine the innermost circuit around $0$, i.e. the island where $0$ lies after a freezing), one only needs to read an additional layer of vertices: the vertices which are vacant along the outer boundary. In the case of forest fires, these boundary vertices might be burnt because of some earlier fires: we thus need to ``explore more'', and potentially move rather deep inside the domain. In this situation, we use truncated ignition processes and the absence of crossing holes in well-chosen annuli (implied by Lemma~\ref{lem:crossing_hole}) as a substitute for spatial independence.

Roughly speaking, our reasoning in this step can be described as follows. It is made of two successive substeps. Given a time $t_0$ sufficiently later than $t_{\infty}$ (in a sense to be made precise below), we first show that some large cluster burns in $\Ball_{L(t_0)}$ ``not too much later'' than $\widehat{t_0}$: before a time $t^*$, which is such that $L(t^*)$ is at least of order $L(\widehat{t_0})$. This burning leaves $0$ in an island $\Lambda^*$ which has a diameter at least of order $L(t^*)$, but possibly much bigger (up to $L(t_0)$).

In a second substep, we then introduce a time $t^{**}$ slightly later than $t^*$, so that on the one hand no macroscopic cluster has already burnt at that time in $\Lambda^*$, and on the other hand we still have $L(t^{**}) \ll L(t^*)$. This can be formulated as: ``the island in which $0$ lies at time $t^{**}$ is big'', in the sense that its diameter is much larger than the characteristic length $L(t^{**})$ at this time. This property is used for an ad hoc construction, showing that we can determine the next burning event without looking at the processes inside a ball with a radius which is both $\ll L(t^*)$ and $\gg L(t^{**})$. In this way, we manage to keep the birth and ignition processes inside the ball sufficiently ``fresh'', which allows us to start the iterative procedure, carried out in the subsequent steps.

Along the way, we introduce several ``modified'' FFWoR processes. We want to emphasize that these processes are only needed here, they will not reappear later. They are used to ensure that certain random times have the right measurability property with respect to the birth and ignition processes.

Let us now describe the argument in detail, and for this purpose, let $r(\zeta) := \frac{m_\infty(\zeta)}{(\log \frac{1}{\zeta})^{\alpha}}$, for some $\alpha > 0$ (small enough) that we explain how to choose later. Let $t_0 = t_0(\zeta) > t_c$ so that $L(t_0) = r(\zeta)$ (from now on, we often drop the dependence on $\zeta$ in the notation). We write $t_0 = t_c + \ve_0$, $t_1 = \widehat{t_0} = t_c + \ve_1$, and $t_2 = \widehat{\widehat{t_0}} = t_c + \ve_2$. We introduce the following time
$$t'_0 = t_c + \ve'_0 := t_c + (\ve_0)^{\frac{1}{2}} (\ve_1)^{\frac{1}{2}},$$
intermediate between $t_0$ and $t_1$. In addition, we consider the times
$$\ol{t}_1 = t_c + \ol{\ve}_1 := t_c + \ol{\kappa} \, \ve_1 \quad \text{and} \quad \ul{t}_1 = t_c + \ul{\ve}_1 := t_c + \ul{\kappa} \, \ve_1,$$
where $\ol{\kappa} \in (0,1)$ and $\ul{\kappa} > 1$ depend only on $\eta$. These constants need to be taken sufficiently small and large, respectively, as we explain just below. Observe that, for all $\zeta$ small enough,
$$t_0 < t'_0 < \ol{t}_1 < t_1 < \ul{t}_1 < t_2.$$

We use later the largest cluster $\lclus$ in $\Ball_{L(t_0)}$ in the configuration $\sigma^{[t_c - \ol{\ve}_1]}$:
$$\PP \bigg( \text{at $\ol{t}_1$, in $\sigma^{[t_c - \ol{\ve}_1]}$, } \frac{|\lclus_{\Ball_{L(t_0)}}|}{\theta(\ol{t}_1) |\Ball_{L(t_0)}|} \in \Big( \frac{1}{2}, \frac{3}{2} \Big) \text{ and } \circuitarm \Big( \Ann_{\frac{1}{2} L(t_0), L(t_0)} \, \big| \, \lclus_{\Ball_{L(t_0)}} \Big) \text{ occurs}\bigg) \geq 1 - \frac{\eta}{100}$$
for all $\zeta$ sufficiently small (from Proposition~\ref{prop:largest_cluster_impurities}, using $L(t_c - \ol{\ve}_1) \asymp L(\ol{t}_1) \ll L(t_0)$). Since
$$c_{\TT} L(t_0)^2 \theta ( t_1 ) \ve_1 = \zeta^{-1},$$
and
$$C_1 \ol{\kappa}^{\frac{5}{36} + \delta} \leq \frac{\theta(t_c+ \ol{\kappa} \ve_1)}{\theta(t_c+\ve_1)} \leq C_2 \ol{\kappa}^{\frac{5}{36} - \delta}$$
for some constants $C_1, C_2 > 0$ (from \eqref{eq:uniform_theta}), we deduce that $\ol{\kappa}(\eta)$ and $\ul{\kappa}(\eta)$ can be chosen so that the following holds. With a probability at least $1 - \frac{\eta}{100}$, no vertex of $\lclus$ gets ignited during the time interval $[t_c - \ol{\ve}_1, \ol{t}_1]$, and at least one of its vertices gets ignited in $(\ol{t}_1, \ul{t}_1]$.

Finally, we introduce the event
$$E_0 := \Big\{ \text{at $t_0$, } \circuitevent^* \big( \Ann_{\delta' L(t_0), L(t_0)} \big) \text{ occurs} \Big\} \cap \Big\{ \text{at $t_0$, in $\sigma^{[t_c - \ve_0]}$, } \circuitevent \big( \Ann_{\delta'^2 L(t_0), \delta' L(t_0)} \big) \text{ occurs} \Big\}.$$
Here, we choose $\delta' = \delta'(\eta) > 0$ small enough so that
\begin{equation} \label{eq:2circuits}
\PP(E_0) \geq 1 - \frac{\eta}{100}
\end{equation}
(this is possible, thanks to \eqref{eq:RSW}, as well as Proposition~\ref{prop:crossing_impurities} combined with Lemma~\ref{lem:stoch_domin}, using that $L(t_c - \ve_0) \asymp L(t_0)$). Observe that by monotonicity, any occupied circuit as in $E_0$, in the configuration $\sigma^{[t_c - \ve_0]}$, is also occupied in $\sigma^{[t_c - \ol{\ve}_1]}$, where we stop ignitions at an earlier time.

We denote by $\circuit$ the outermost occupied circuit in $\Ann_{\delta'^2 L(t_0), \delta' L(t_0)}$ at time $t_0$ in $\sigma^{[t_c - \ol{\ve}_1]}$, when such a circuit exists (and we let $\circuit = \dout \Ball_{\delta'^2 L(t_0)}$ otherwise). Note that with probability at least $1 - \frac{\eta}{100}$, $\circuit$ is contained in $\lclus$ at time $\ol{t}_1$. Indeed, this follows from the occurrence of $\circuitarm \big( \Ann_{\frac{1}{2} L(t_0), L(t_0)} \, | \, \lclus_{\Ball_{L(t_0)}} \big)$, combined with an occupied arm in $\Ann_{\delta'^2 L(t_0), L(t_0)}$ (provided by Proposition~\ref{prop:exp_decay_impurities}).

Let us now consider for a moment a modified FFWoR process $\sigma'$, where ignitions in $\Ball_{L(t_0)}$ after time $t_c - \ol{\ve}_1$ are discarded. From the observation above, as far as $\lclus$ is concerned, $\sigma'$ coincides with the original process $\sigma$ up to time $\ol{t}_1$ (with probability at least $1 - \frac{\eta}{100}$). Furthermore, the circuit $\circuit$ is ``protected'' by the vacant circuit in $\Ann_{\delta' L(t_0), L(t_0)}$ up to time $t_0$ (i.e. over the interval $(t_c - \ol{\ve}_1, t_0]$), so necessarily none of its vertices is burnt at time $t_0$. Hence, all vertices of $\circuit$ burn simultaneously at a later time, that we denote by $t'$. We first consider the situation at time $t'_0$ (for the process $\sigma'$), and we distinguish the following two cases.
\begin{itemize}
\item If $\circuit$ has burnt already, i.e. $t' \leq t'_0$, this is necessarily because of a fire coming from an ignition outside $\Ball_{L(t_0)}$. We define $t^* = t'_0$, and we introduce the event
$$E'_0 := \Big\{ \text{at $t'_0$, } \circuitevent^* \big( \Ann_{\delta' L(t'_0), L(t'_0)} \big) \text{ occurs} \Big\}.$$
Note that from our choice of $\delta'$,
$$\PP(E'_0) \geq 1 - \frac{\eta}{100},$$
and the vacant circuit that it provides (in $\Ann_{\delta' L(t'_0), L(t'_0)}$) is also vacant at time $t_0$. Hence, at time $t^*$, $0$ is surrounded by a burnt circuit lying in $\Ann_{\ul{r}', \ol{r}'}$, where
\begin{equation} \label{eq:burnt_island1}
\ul{r}' = \delta' L(t'_0) \quad \text{and} \quad \ol{r}' = \delta' L(t_0).
\end{equation}
We define the simply connected domain $\Lambda^*$ as the connected component of vertices containing $0$ when one removes the burnt cluster of $\circuit$.

\item If $\circuit$ is not yet burnt at time $t'_0$, then analogously to the first case, we consider the event
$$E'_1 := \Big\{ \text{at $\ul{t}_1$, } \circuitevent^* \big( \Ann_{\delta' L(\ul{t}_1), L(\ul{t}_1)} \big) \text{ occurs} \Big\},$$
which has a probability at least $1 - \frac{\eta}{100}$. We also introduce the additional event
$$\ol{E}_1 := \Big\{ \text{at $t'_0$, in $\sigma^{[t_c - \ve'_0]}$, } \circuitevent \big( \Ann_{C L(t'_0), 2 C L(t'_0)} \big) \text{ occurs, and } \dout \Ball_{C L(t'_0)} \lra \circuit \Big\},$$
where $C = C(\eta) \geq 1$ is chosen large enough so that
$$\PP(\ol{E}_1) \geq 1 - \frac{\eta}{100}$$
(this uses Proposition~\ref{prop:exp_decay_impurities}). By monotonicity, such occupied paths are present in $\sigma^{[t_c - \ol{\ve}_1]}$ as well.

Let us consider the modified FFWoR process $\sigma''$, where ignitions in $\Ball_{2 C L(t'_0)}$ after time $t_c - \ol{\ve}_1$ are discarded. Again, if we restrict ourselves to $\lclus$ (or even the largest cluster at time $\ul{t}_1$), $\sigma''$ coincides with $\sigma$ up to time $\ul{t}_1$. Indeed, $\lclus \cap \Ball_{2 C L(t'_0)}$ contains too few vertices for it to be ignited during $[t_c - \ol{\ve}_1, \ul{t}_1]$.

Moreover, we observed earlier that at least one of the vertices of $\lclus$ gets ignited during $(\ol{t}_1, \ul{t}_1]$, and $\circuit \subseteq \lclus$. Hence, either $\circuit$ burns in this time interval, or it burns earlier, i.e. in the time interval $(t'_0, \ol{t}_1]$. Note that the circuit in $\Ann_{C L(t'_0), 2 C L(t'_0)}$, which at time $t'_0$ is occupied and connected to $\circuit$, has to burn together with $\circuit$, so that at time $\ul{t}_1$, $0$ is surrounded by a burnt circuit lying in $\Ann_{\ul{r}', \ol{r}'}$, with
\begin{equation} \label{eq:burnt_island2}
\ul{r}' = \delta' L(\ul{t}_1) \quad \text{and} \quad \ol{r}' = 2 C L(t'_0)
\end{equation}
in this case. We let $t^* = \ul{t}_1$, and we define $\Lambda^*$ in a similar way as before.
\end{itemize}

In both cases, we can condition on the birth and ignition processes outside a suitable ball: hence, $t'$ and $t^*$ can be seen as fixed. Note that
\begin{equation} \label{eq:hole_Lambda*1}
\Ball_{\delta' L(t^*)} \subseteq \Lambda^*
\end{equation}
(from the definition of $t^*$ in each case, \eqref{eq:burnt_island1} and \eqref{eq:burnt_island2}). Moreover, we have
\begin{equation} \label{eq:hole_Lambda*2}
\Lambda^* \subseteq \Ball_{\ol{r}'},
\end{equation}
where the radius $\ol{r}'$ satisfies the following property: writing $t^* = t_c + \ve^*$,
\begin{equation} \label{eq:small_vol_Lambda*1}
( \ol{r}' )^2 \theta ( t_c + \ve^* ) \ve^* \ll \zeta^{-1} \quad \text{as $\zeta \searrow 0$.}
\end{equation}
In other words, $\next_\zeta(\ol{r}')$ is much later than $t^*$, which ensures that ``nothing substantial'' burns in $\Lambda^*$ before time $t^*$. In addition, using the ``no crossing hole'' event $\calH(\Ann_{\frac{1}{2} \delta' L(t^*), \delta' L(t^*)})$, which has a probability $\geq 1 - \frac{\eta}{100}$ for all $\zeta$ sufficiently small, from Lemma~\ref{lem:crossing_hole}, we can consider the situation inside $\Ball_{\frac{1}{2} \delta' L(t^*)}$ as ``fresh''. More precisely, on this event, the stochastic lower bound provided by percolation with impurities remains valid in $\Ball_{\frac{1}{2} \delta' L(t^*)}$. Indeed, the conditioning only impacts the impurities centered outside $\Ball_{\delta' L(t^*)}$. For the FFWoR process (with ignitions stopped at suitable times), this allows us in the following to obtain that certain occupied connections occur with high probability.

We now have to study the FFWoR process inside $\Lambda^*$. First, we introduce
\begin{equation} \label{eq:def_t**}
t^{**} = t_c + \ve^{**} := t_c + \bigg( \log \log \frac{1}{\zeta} \bigg)^2 \ve^*.
\end{equation}
It follows immediately from \eqref{eq:uniform_L} that for some universal $\kappa' > 0$,
\begin{equation} \label{eq:L_t**}
L(t^{**}) \leq \kappa' \bigg( \log \log \frac{1}{\zeta} \bigg)^{-2} L(t^*).
\end{equation}
On the other hand, using \eqref{eq:uniform_theta}, a similar reasoning as for \eqref{eq:small_vol_Lambda*1} yields
\begin{equation} \label{eq:small_vol_Lambda*2}
( \ol{r}' )^2 \theta ( t_c + \ve^{**} ) \ve^{**} \ll \zeta^{-1} \quad \text{as $\zeta \searrow 0$.}
\end{equation}
We then introduce the radii
\begin{equation} \label{eq:radii_big_hole}
\ul{r}'' = \bigg( \log \log \frac{1}{\zeta} \bigg) L(t^{**}) \quad \text{and} \quad \ol{r}'' = \frac{1}{2} \delta' L(t^*).
\end{equation}
Note that $\ul{r}'' < \ol{r}''$ for all $\zeta$ sufficiently small, and $\Ball_{\ol{r}''} \subseteq \Lambda^*$ (from \eqref{eq:hole_Lambda*1}).

We know that for all $\zeta$ small enough,
$$\PP \Big( \text{at $t^{**}$, in $\sigma^{[t_c - \ve^{**}]}$, } \circuitevent \big( \Ann_{\frac{1}{2} \ol{r}'', \ol{r}''} \big) \text{ occurs} \Big) \geq 1 - \frac{\eta}{100}.$$
Indeed, this follows from Proposition~\ref{prop:exp_decay_impurities}, since $L(t_c - \ve^{**}) \asymp L(t^{**}) \ll \ol{r}''$ (using \eqref{eq:L_t**}). We denote by $\circuit^*$ the outermost such occupied circuit in $\Ann_{\frac{1}{2} \ol{r}'', \ol{r}''}$, and by $\cluster_{\circuit^*}$ its connected component in $\Lambda^*$ (at $t^{**}$, in $\sigma^{[t_c - \ve^{**}]}$). It follows from \eqref{eq:small_vol_Lambda*2} and \eqref{eq:hole_Lambda*2} that among the vertices in $\cluster_{\circuit^*}$, none of them gets ignited in the time interval $(t_c - \ve^{**}, t^{**}]$. We deduce that $\circuit^*$ is not yet burnt at time $t^{**}$.

Let us introduce
\begin{equation} \label{eq:next_zeta''}
t'' = t_c + \ve'' := \next_\zeta ( \ol{r}'' ).
\end{equation}
Note that
$$\ve'' \gg \ve^{**}.$$
Indeed, $\ve^{**} = \big( \log \log \frac{1}{\zeta} \big)^2 \ve^*$ (from \eqref{eq:def_t**}), while it follows from our initial choice $L(t_0) = r(\zeta) = \frac{m_\infty(\zeta)}{(\log \frac{1}{\zeta})^{\alpha}}$ that for some $\xi > 0$ (depending on $\alpha$) and $c'' > 0$,
$$\ve'' \geq c'' \bigg( \log \frac{1}{\zeta} \bigg)^{\xi} \ve^*.$$
More precisely, this last inequality can be obtained by applying \eqref{eq:diam_apart} with $\ol{r}'' \asymp L(t^*)$ and $L(\next_\zeta ( \ol{r}'' )) = L(t'')$ (from \eqref{eq:radii_big_hole} and \eqref{eq:next_zeta''}, respectively), using $L(t^*) \leq L(t_0)$ (since $t^* \geq t_0 > t_c$), and then combining it with \eqref{eq:uniform_L}.

We consider the associated times
\begin{equation} \label{eq:times''}
\ol{t}'' = t_c + \ol{\ve}'' := t_c + \ol{\kappa}'' \, \ve'' \quad \text{and} \quad \ul{t}'' = t_c + \ul{\ve}'' := t_c + \ul{\kappa}'' \, \ve'',
\end{equation}
where, in a similar way as before, $\ol{\kappa}''(\eta) \in (0,1)$ and $\ul{\kappa}''(\eta) > 1$ are chosen (small and large enough, respectively) so that for the largest cluster $\lclusbar$ in $\Ball_{\ol{r}''}$ at time $\ol{t}''$, in the configuration $\sigma^{[t_c - \ol{\ve}'']}$: with a probability at least $1 - \frac{\eta}{100}$, no vertex of $\lclusbar$ gets ignited during the time interval $[t_c - \ol{\ve}'', \ol{t}'']$, and at least one of its vertices gets ignited in $(\ol{t}'', \ul{t}'']$. In addition, we can require the event $\circuitarm \big( \Ann_{\frac{1}{2} \ol{r}'', \ol{r}''} \, | \, \lclus_{\Ball_{\ol{r}''}} \big)$ to occur, which ensures that $\circuit^* \subseteq \lclusbar$, since $\circuit^*$ is also occupied in $\sigma^{[t_c - \ol{\ve}'']}$ (using that $\ve'' \gg \ve^{**}$). The above observations show, together, that $\circuit^*$ must burn as a whole, during the time interval $(t^{**},\ul{t}'']$.

Finally, we have
$$\ul{r}'' = \bigg( \log \log \frac{1}{\zeta} \bigg) L(t^{**}) \leq \kappa' \bigg( \log \log \frac{1}{\zeta} \bigg)^{-1} L(t^*)$$
(from \eqref{eq:radii_big_hole} and \eqref{eq:L_t**}), which, combined with \eqref{eq:next_zeta''} and \eqref{eq:times''}, yields
$$( \ul{r}'' )^2 \theta ( t_c + \ul{\ve}'' ) \ul{\ve}'' \ll \zeta^{-1} \quad \text{as $\zeta \searrow 0$.}$$
Hence, with probability at least $1 - \frac{\eta}{100}$, no vertex of $\lclusbar \cap \Ball_{\ul{r}''}$ gets ignited before time $\ul{t}''$. We thus consider the modified FFWoR process $\sigma^*$ in $\Lambda^*$, where ignitions in $\Ball_{\ul{r}''}$ are discarded after time $t_c - \ol{\ve}''$, and in this process, we denote by $t^+$ ($> t^{**}$) the time at which $\circuit^*$ burns. We then define $\Lambda^+$ as the connected component containing $0$ in $\Ball_{\ul{r}''}$ after removing the burnt cluster of $\circuit^*$.

Observe that $t^+$ can be seen as fixed if we condition on the processes outside $\Ball_{\ul{r}''}$, which we do from now on. From \eqref{eq:radii_big_hole},
$$\ul{r}'' = \bigg( \log \log \frac{1}{\zeta} \bigg) L(t^{**}) \gg L(t^{**}) > L(t^+)$$
so we can deduce that for some universal constants $0 < \ul{c}^+ < \ol{c}^+$,
\begin{equation}
\PP \Big( \Ball_{\ul{c}^+ L(t^+)} \subseteq \Lambda^+ \subseteq \Ball_{\ol{c}^+ L(t^+)} \Big) \geq 1 - \frac{\eta}{100}.
\end{equation}
This can be seen by considering the events, for some $0 < \ul{c}^+ < \ol{c}^+$ (depending on $\eta$),
$$\ul{E}^+ := \Big\{ \text{at $t^+$, } \circuitevent^* \big( \Ann_{\ul{c}^+ L(t^+), 2 \ul{c}^+ L(t^+)} \big) \text{ occurs} \Big\}$$
and
$$\ol{E}^+ := \Big\{ \text{at $t^+$, in $\sigma^{[t_c - \ve^{**}]}$, } \circuitevent \big( \Ann_{\frac{1}{2} \ol{c}^+ L(t^+), \ol{c}^+ L(t^+)} \big) \text{ occurs, and } \dout \Ball_{\frac{1}{2} \ol{c}^+ L(t^+)} \lra \circuit^* \Big\}.$$
Indeed, we can ensure that $\PP(\ul{E}^+ \cap \ol{E}^+) \geq 1 - \frac{\eta}{100}$ by choosing $\ul{c}^+(\eta)$ and $\ol{c}^+(\eta)$ small and large enough, respectively, and the occupied paths from $\ol{E}^+$ are also present in $\sigma^{[t_c - \ol{\ve}'']}$ (using again $\ve'' \gg \ve^{**}$). Moreover, $\Lambda^+$ coincides with the connected component of $0$ when one removes all vertices connected to $\din \Ball_{\frac{1}{2} \ul{r}''}$ at time $t^+$, and the situation in $\Ball_{\frac{1}{2} \ul{r}''}$ can be considered as fresh (in the same sense as before, i.e. for the percolation processes with impurities used as stochastic lower bounds). For this, it suffices to introduce the event $\calH(\Ann_{\frac{1}{2} \ul{r}'', \ul{r}''})$, which has a probability $\geq 1 - \frac{\eta}{100}$ for all $\zeta$ sufficiently small (using Lemma~\ref{lem:crossing_hole}).

Finally, let us estimate $L(t^+)$. On the one hand, we clearly have $t^+ \geq t_0$, which gives the upper bound
$$L(t^+) \leq L(t_0) = r(\zeta) = \frac{m_\infty(\zeta)}{(\log \frac{1}{\zeta})^{\alpha}}.$$
On the other hand, a lower bound can be derived by noting that $t^+ \leq \ul{t}''$, which is itself essentially $t_2$. More precisely,
\begin{equation} \label{eq:L_t+1}
L(t^+) \geq L(\ul{t}'') = L(t_c + \ul{\kappa}'' \, \ve'') \asymp L(t_c + \ve'') = L \big( \next_\zeta ( \ol{r}'' ) \big) = L \bigg( \next_\zeta \bigg( \frac{1}{2} \delta' L(t^*) \bigg) \bigg) \asymp L \big( \widehat{t^*} \big),
\end{equation}
and since $t^* \leq \ul{t}_1$,
\begin{equation} \label{eq:L_t+2}
L(t^*) \geq L(\ul{t}_1) = L(t_c + \ul{\kappa} \, \ve_1) \asymp L(t_c + \ve_1) = L(\widehat{t_0}).
\end{equation}
By combining \eqref{eq:L_t+1} and \eqref{eq:L_t+2} with $L(t_0) = r(\zeta) = \frac{m_\infty(\zeta)}{(\log \frac{1}{\zeta})^{\alpha}}$, we obtain
\begin{equation} \label{eq:L_t+3}
L(t^+) \geq c_1^+ L \big( \widehat{\widehat{t_0}} \big) \geq c_2^+ \frac{m_\infty(\zeta)}{(\log \frac{1}{\zeta})^{\alpha (\aval + \delta)^2}}
\end{equation}
(by applying Lemma~\ref{lem:one_iteration} twice). This can be used to estimate, in a similar way as for frozen percolation (see \eqref{eq:step1_FP}), the number of burnt clusters surrounding $0$, contained in $\Ball_{K m_\infty(\zeta)}$, and intersecting $\Ann_{\ul{c}^+ L(t^+), K m_{\infty}}$. From Lemma~\ref{lem:disjoint_circuits_FF}, we obtain that with probability $\geq 1 - \frac{\eta}{100}$, there are at most
\begin{equation} \label{eq:apriori_circuits_FF}
C_2 \log \bigg( \frac{K m_{\infty}}{\ul{c}^+ L(t^+)} \bigg) \leq C_2 \alpha (\aval + \delta)^2 \log \log \frac{1}{\zeta} + O(1)
\end{equation}
such clusters, for some $C_2$ which depends only on $\eta$ and $K$. In particular, this allows us to choose $\alpha = \alpha(\ve,\eta,K)$ sufficiently small so that the r.h.s. of \eqref{eq:apriori_circuits_FF} is at most $\frac{\ve}{2} \log \log \frac{1}{\zeta}$ for all $\zeta$ small enough.

\bigskip

\textbf{Step 2}: We prove the result by following an iterative construction, which, in some sense, refines the one used for Theorem 7.2 in \cite{BN2018}. Note that all sequences that we introduce now can be considered as deterministic, owing to the conditionings that we made in order to ensure that the ``starting time'' $t^+$ is fixed.

We consider two sequences $(r^{(i)})_{i \geq 0}$ and $(R^{(i)})_{i \geq 0}$, with $r^{(i)} \leq R^{(i)}$ ($i \geq 0$), defined by induction as follows. We start from $\Lambda^{(0)} := \Lambda^+$,
\begin{equation} \label{eq:start_FF}
r^{(0)} := \ul{c}^+ L(t^+) \quad \text{and} \quad R^{(0)} := \ol{c}^+ L(t^+),
\end{equation}
so that $\Ball_{r^{(0)}} \subseteq \Lambda^{(0)} \subseteq \Ball_{R^{(0)}}$. Given $r^{(i)} \leq R^{(i)}$ for some $i \geq 0$, we introduce the times
\begin{equation} \label{eq:t_FFWoR}
\ul{t}^{(i)} = t_c + \ul{\ve}^{(i)} := \next_\zeta \big( r^{(i)} \big) \quad \text{and} \quad \ol{t}^{(i)} = t_c + \ol{\ve}^{(i)} := \next_\zeta \big( R^{(i)} \big).
\end{equation}
Note that $t_c < \ol{t}^{(i)} \leq \ul{t}^{(i)} < \infty$ (since $\next_\zeta$ is nonincreasing), and thus $0 < L(\ul{t}^{(i)}) \leq L(\ol{t}^{(i)})$. We then let
\begin{equation} \label{eq:tt_FFWoR}
\dul{t}^{(i)} = t_c + \dul{\ve}^{(i)} := t_c + \bigg( \log \log \frac{1}{\zeta} \bigg) \ul{\ve}^{(i)} \quad \text{and} \quad \dol{t}^{(i)} = t_c + \dol{\ve}^{(i)} := t_c + \frac{1}{\big( \log \log \frac{1}{\zeta} \big)^2} \ol{\ve}^{(i)},
\end{equation}
and
\begin{equation} \label{eq:r_FFWoR}
r^{(i+1)} := \frac{1}{\big( \log \log \frac{1}{\zeta} \big)^{24}} \, L \big( \dul{t}^{(i)} \big) \quad \text{and} \quad R^{(i+1)} = \bigg( \log \log \log \frac{1}{\zeta} \bigg)^4 \, L \big( \dol{t}^{(i)} \big).
\end{equation}
We have clearly $0 < r^{(i+1)} \leq R^{(i+1)} < \infty$. We then define, as for frozen percolation,
$$j := \min \bigg\{ i \geq 1 \: : \: r^{(i)} < \frac{1}{\sqrt{\zeta}} \bigg\} - 1 \quad \text{and} \quad J := \min \bigg\{ i \geq 1 \: : \: R^{(i)} < \frac{1}{\sqrt{\zeta}} \bigg\} - 1$$
(observe that $j \leq J$). Finally, we also make use of the intermediate time
\begin{equation} \label{eq:ttt_FFWoR}
\doul{t}^{(i)} = t_c + \doul{\ve}^{(i)} := t_c + \frac{1}{\big( \log \log \frac{1}{\zeta} \big)^2} \ul{\ve}^{(i)}.
\end{equation}

\begin{figure}[t]
\begin{center}

\includegraphics[width=.8\textwidth]{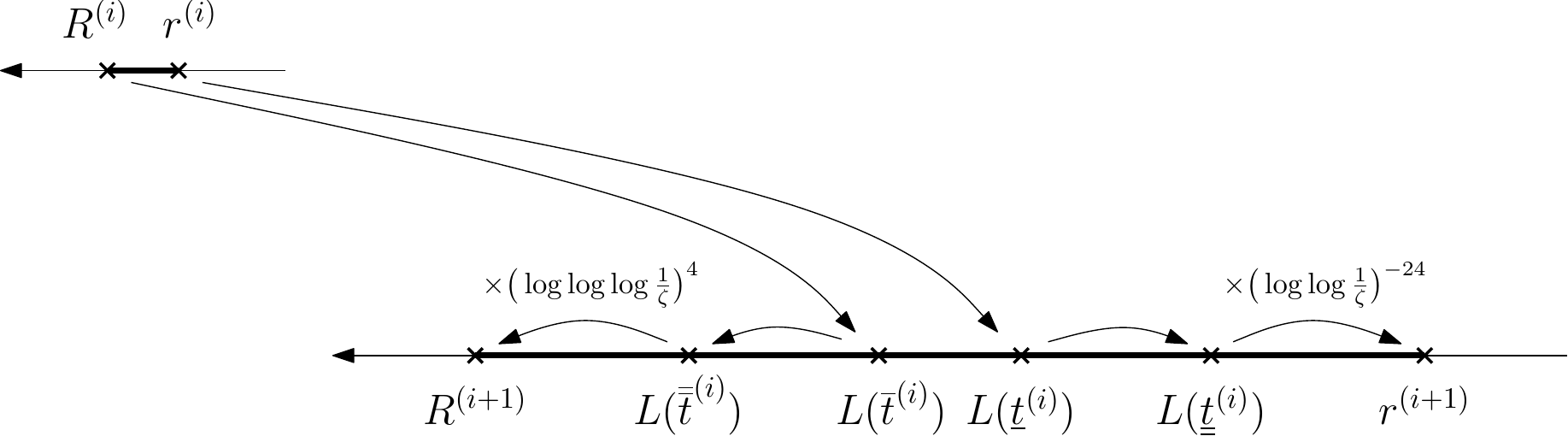}
\caption{\label{fig:scales_FF} This figure presents schematically the successive scales involved in the iterative construction. If we have a domain with boundary contained in $\Ann_{r^{(i)},R^{(i)}}$, the construction ensures that one cluster surrounding $0$ burns, leaving $0$ in an island with boundary contained in $\Ann_{r^{(i+1)},R^{(i+1)}}$.}

\end{center}
\end{figure}

For future reference, note that
$$t_c < \dol{t}^{(i)} \leq \ol{t}^{(i)}, \, \doul{t}^{(i)} \leq \ul{t}^{(i)} < \dul{t}^{(i)}.$$
We also need later that, from \eqref{eq:uniform_theta} and \eqref{eq:ttt_FFWoR},
\begin{equation} \label{eq:lower_vol}
\frac{\theta(\doul{t}^{(i)})}{\theta(\ul{t}^{(i)})} \geq \kappa'_1 \bigg( \frac{\doul{\ve}^{(i)}}{\ul{\ve}^{(i)}} \bigg)^{\frac{5}{36}+\delta} = \kappa'_1 \bigg( \log \log \frac{1}{\zeta} \bigg)^{- 2 (\frac{5}{36}+\delta)} \geq \kappa'_1 \bigg( \log \log \frac{1}{\zeta} \bigg)^{- \frac{1}{3}}
\end{equation}
for some universal constant $\kappa'_1 > 0$. Similarly, from \eqref{eq:uniform_L} and \eqref{eq:tt_FFWoR},
\begin{equation} \label{eq:L_t_estimate}
L(\dul{t}^{(i)}) \geq \kappa'_2 \bigg( \log \log \frac{1}{\zeta} \bigg)^{-2} L(\ul{t}^{(i)}) \quad \text{and} \quad L(\dol{t}^{(i)}) \leq \kappa'_3 \bigg( \log \log \frac{1}{\zeta} \bigg)^3 L(\ol{t}^{(i)}),
\end{equation}
where $\kappa'_2, \kappa'_3 > 0$ are universal as well. By combining \eqref{eq:r_FFWoR} and \eqref{eq:L_t_estimate}, we obtain
\begin{equation} \label{eq:r_bounds_FFWoR}
\frac{\kappa'_2}{\big( \log \log \frac{1}{\zeta} \big)^{26}} \, L \big( \ul{t}^{(i)} \big) \leq r^{(i+1)} \leq L \big( \ul{t}^{(i)} \big) \quad \text{and} \quad L \big( \ol{t}^{(i)} \big) \leq R^{(i+1)} \leq \kappa'_3 \bigg( \log \log \frac{1}{\zeta} \bigg)^7 \, L \big( \ol{t}^{(i)} \big).
\end{equation}

Let $\ve' > 0$, that we specify later as a function of $\ve$. In the following, all constants may depend only on $\ve$ (or $\ve'$), but not on $i$. The relations \eqref{eq:r_bounds_FFWoR} are similar to \eqref{eq:r_def} in the proof of Theorem \ref{thm:mainFP}. After making the necessary adjustments, we can deduce by induction, starting from
$\frac{R^{(0)}}{m_{\infty}} = \frac{1}{(\log \frac{1}{\zeta})^{\alpha}}$ (see \eqref{eq:start_FF}) and using repeatedly \eqref{eq:r_bounds_FFWoR} and \eqref{eq:compar_minf} (together with \eqref{eq:t_FFWoR}), that: for all $i = 0, \ldots, J+1$,
\begin{equation} \label{eq:r_estimate_FFWoR}
\bigg( \frac{c_1}{\log \frac{1}{\zeta}} \bigg)^{\alpha (\aval + \ve')^i} \leq \frac{R^{(i)}}{m_{\infty}} \leq \bigg( \frac{c'_1}{\log \frac{1}{\zeta}} \bigg)^{\frac{\alpha}{2} (\aval - \ve')^i},
\end{equation}
for some $c_1, c'_1 > 0$ (analogously to \eqref{eq:r_estimate}).

Also, using \eqref{eq:start_FF}, \eqref{eq:t_FFWoR}, \eqref{eq:r_bounds_FFWoR} and \eqref{eq:compar_minf}, we have
$$\frac{R^{(1)}}{r^{(0)}} = \frac{R^{(1)}}{m_{\infty}} \cdot \bigg( \frac{R^{(0)}}{m_{\infty}} \bigg)^{-1} \leq c_2 \bigg( \log \log \frac{1}{\zeta} \bigg)^7 \bigg( \frac{R^{(0)}}{m_{\infty}} \bigg)^{\aval - \ve' - 1} \leq \bigg( \frac{c'_2}{\log \frac{1}{\zeta}} \bigg)^{\xi},$$
for some $c_2, c'_2, \xi > 0$ (again, we use that $\aval > 1$). By induction, using Lemma \ref{lem:one_iteration}, we obtain the following analog of \eqref{eq:r_estimate2}: for all $i = 0, \ldots, J$,
\begin{equation} \label{eq:r_estimate2_FFWoR}
\frac{R^{(i+1)}}{r^{(i)}} \leq \bigg( \frac{c_3 \big( \log \log \frac{1}{\zeta} \big)^{33}}{\big( \log \frac{1}{\zeta} \big)^{\xi}} \bigg)^{(\aval - \ve')^i} \leq \bigg( \frac{c'_4}{\log \frac{1}{\zeta}} \bigg)^{\frac{\xi}{2} (\aval - \ve')^i}.
\end{equation}
In particular, $j \geq J-1$, and for all $\zeta > 0$ small enough, we have: for all $i = 0, \ldots, J$, $R^{(i+1)} < \frac{1}{10} r^{(i)}$.

Similarly to frozen percolation, we can estimate $J$ thanks to \eqref{eq:r_estimate_FFWoR}, starting from $R^{(J)} \geq \frac{1}{\sqrt{\zeta}} > R^{(J+1)}$ and using \eqref{eq:exp_t_infty_FF}. We have
$$\frac{R^{(J)}}{m_{\infty}} \geq c \, \zeta^{\alpha + \delta} \quad \text{and} \quad \frac{R^{(J+1)}}{m_{\infty}} \leq c' \, \zeta^{\alpha - \delta},$$
where $c, c' > 0$ and $\alpha = \frac{48}{55} - \frac{1}{2} > 0$ are universal. This yields
\begin{equation} \label{eq:bound_J_FFWoR}
\frac{1}{\log(\aval + \ve')} \bigg( \log \log \frac{1}{\zeta} + O \bigg( \log \log \log \frac{1}{\zeta} \bigg) \bigg) \leq J \leq \frac{1}{\log(\aval - \ve')} \bigg( \log \log \frac{1}{\zeta} + O \bigg( \log \log \log \frac{1}{\zeta} \bigg) \bigg).
\end{equation}
Since $\avalnb = \avalnb^{\textrm{FF}} = \frac{1}{\log \aval}$, we can get from \eqref{eq:bound_J_FFWoR}, by choosing $\ve'$ small enough,
\begin{equation} \label{eq:bound_J_FFWoR2}
\avalnb - \frac{\ve}{3} + o(1) \leq \frac{J}{\log \log \frac{1}{\zeta}} \leq \avalnb + \frac{\ve}{3} + o(1).
\end{equation}

\bigskip

\textbf{Step 3}: We now introduce events $E_1^{(i)}$--$E_9^{(i)}$, for all $i =0, \ldots, j$, involving
\begin{itemize}
\item the percolation configuration in the pure birth process, at times $\dol{t}^{(i)}$, $\doul{t}^{(i)}$ and $\dul{t}^{(i)}$,

\item the configuration of the FFWoR process where ignitions are stopped at time $t_c - \dol{\ve}^{(i)}$ (recall that this process is denoted by $\sigma^{[t_c - \dol{\ve}^{(i)}]}$), at times $\dol{t}^{(i)}$ and $\doul{t}^{(i)}$,

\item as well as the ignition process over the time interval $\big[ t_c - \dol{\ve}^{(i)}, \dul{t}^{(i)} \big]$.
\end{itemize}

First, let
\begin{itemize}
\item $E_1^{(i)} := \Big\{$at $\dol{t}^{(i)}$, $\Big| \lclus_{\Ball_{R^{(i)}}} \Big| \leq 2 \big( \ol{\ve}^{(i)} \zeta \big)^{-1}$ and $\circuitarm \big( \Ann_{\frac{1}{2} R^{(i)}, R^{(i)}} \, | \, \lclus_{\Ball_{R^{(i)}}} \big)$ occurs$\Big\}$,

\item $E_2^{(i)} := \Big\{$no vertex of $\ol{\ol{\cluster}}^{(i)}$ gets ignited in $\big[ t_c - \dol{\ve}^{(i)}, \dol{t}^{(i)} \big] \Big\}$, where $\ol{\ol{\cluster}}^{(i)} := \lclus_{\Ball_{R^{(i)}}}$ at time $\dol{t}^{(i)}$,

\item $E_3^{(i)} := \Big\{$at $\dol{t}^{(i)}$, in $\sigma^{[t_c - \dol{\ve}^{(i)}]}$, $\circuitevent \big( \Ann_{\frac{1}{2} R^{(i+1)}, R^{(i+1)}} \big)$ occurs and $\dout \Ball_{\frac{1}{2} R^{(i+1)}} \lra \infty \Big\}$,

\item $E_4^{(i)} := \Big\{$at $\dol{t}^{(i)}$, in $\sigma^{[t_c - \dol{\ve}^{(i)}]}$, $\circuitevent \big( \Ann_{\frac{1}{2} r^{(i)}, r^{(i)}} \big)$ occurs$\Big\}$,

\item $E_5^{(i)} := \Big\{$at $\dul{t}^{(i)}$, $\circuitevent^* \big( \Ann_{r^{(i+1)}, \frac{1}{2} R^{(i+1)}} \big)$ occurs$\Big\}$.
\end{itemize}
We then define four more events:
\begin{itemize}
\item $E_6^{(i)} := \Big\{$at $\doul{t}^{(i)}$, $\Big| \lclus_{\Ball_{r^{(i)}}} \Big| \leq 2 \big( \ul{\ve}^{(i)} \zeta \big)^{-1}$ and $\circuitarm \big( \Ann_{\frac{1}{2} r^{(i)}, r^{(i)}} \, | \, \lclus_{\Ball_{r^{(i)}}} \big)$ occurs$\Big\}$,

\item $E_7^{(i)} := \Big\{$no vertex of $\doul{\cluster}^{(i)}$ gets ignited in $\big[ t_c - \dol{\ve}^{(i)}, \doul{t}^{(i)} \big] \Big\}$, where $\doul{\cluster}^{(i)} := \lclus_{\Ball_{r^{(i)}}}$ at time $\doul{t}^{(i)}$,

\item $E_8^{(i)} := \Big\{$at $\doul{t}^{(i)}$, in $\sigma^{[t_c - \dol{\ve}^{(i)}]}$, $\Big| \lclus_{\Ball_{r^{(i)}}} \cap \Ball_{\frac{1}{2} r^{(i)}} \Big| \geq \frac{\kappa'_1}{8} \big( \log \log \frac{1}{\zeta} \big)^{-\frac{1}{3}} \big( \ul{\ve}^{(i)} \zeta \big)^{-1}$ and $\circuitarm \big( \Ann_{\frac{1}{2} r^{(i)}, r^{(i)}} \, | \, \lclus_{\Ball_{r^{(i)}}} \big)$ occurs$\Big\}$, where $\kappa'_1 > 0$ is the universal constant from \eqref{eq:lower_vol},

\item $E_9^{(i)} := \Big\{$at least one vertex of $\ul{\cluster}^{(i)} \cap \Ball_{\frac{1}{2} r^{(i)}}$ is ignited in $\big[ \doul{t}^{(i)}, \dul{t}^{(i)} \big] \Big\}$, where $\ul{\cluster}^{(i)} := \lclus_{\Ball_{r^{(i)}}}$ at time $\doul{t}^{(i)}$, in $\sigma^{[t_c - \dol{\ve}^{(i)}]}$.
\end{itemize}

\begin{figure}[t]
\begin{center}

\includegraphics[width=.88\textwidth]{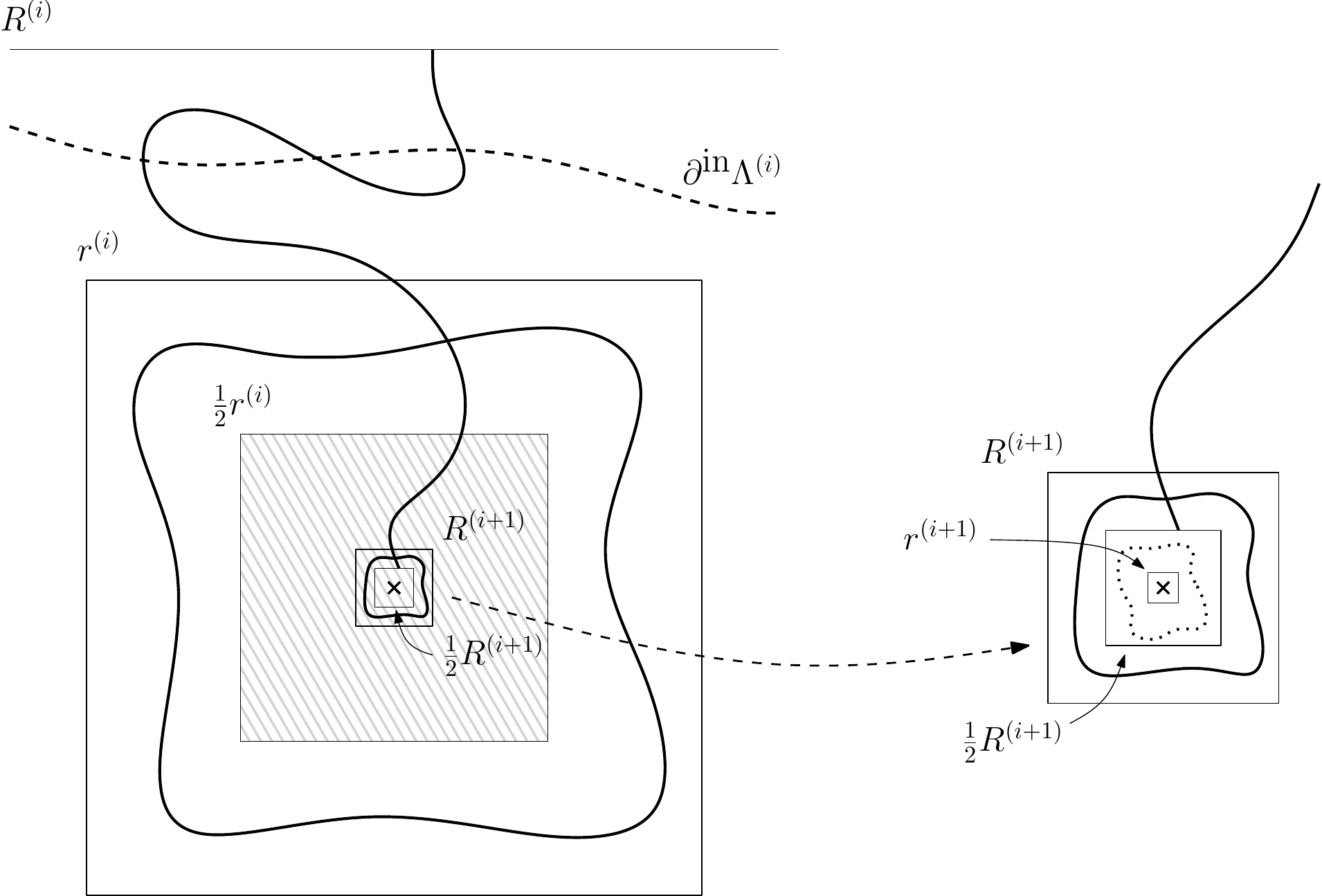}
\caption{\label{fig:iteration_FF} This figure presents the iterative construction used to produce the clusters burning successively around $0$. In the $i$th step, in a domain $\Lambda^{(i)}$ with $\Ball_{r^{(i)}} \subseteq \Lambda^{(i)} \subseteq \Ball_{R^{(i)}}$, we consider times $t_c < \dol{t}^{(i)} < \dul{t}^{(i)}$. The path to $\din \Ball_{R^{(i)}}$ and the two circuits in solid line are occupied (not burnt) at time $\dol{t}^{(i)}$, while the dotted circuit is vacant at time $\dul{t}^{(i)}$. The circuit in $\Ann_{\frac{1}{2} r^{(i)}, r^{(i)}}$ separates from outside the vertices in the gray region, which are used to ``trigger'' the burning of the solid path (more precisely, its connected component in $\Lambda^{(i)}$ containing the two circuits) before time $\dul{t}^{(i)}$, if this path has not already burnt on its own. This leaves $0$ in an island $\Lambda^{(i+1)}$ which satisfies $\Ball_{r^{(i+1)}} \subseteq \Lambda^{(i+1)} \subseteq \Ball_{R^{(i+1)}}$.}

\end{center}
\end{figure}

Observe that $\ul{\cluster}^{(i)} \subseteq \doul{\cluster}^{(i)}$. Indeed, all occupied clusters in $\sigma^{[t_c - \dol{\ve}^{(i)}]}$ are contained in occupied clusters in the pure birth process, and the two clusters $\ul{\cluster}^{(i)}$ and $\doul{\cluster}^{(i)}$ intersect, since they both contain an occupied circuit and an occupied crossing in $\Ann_{\frac{1}{2} r^{(i)}, r^{(i)}}$.

Let us examine the consequence of these nine events occurring simultaneously for the FFWoR process, in any simply connected domain $\Lambda^{(i)}$ with $\Ball_{r^{(i)}} \subseteq \Lambda^{(i)} \subseteq \Ball_{R^{(i)}}$. First, at time $\dol{t}^{(i)}$, in $\sigma^{[t_c - \dol{\ve}^{(i)}]}$, $E_3^{(i)}$ provides the existence of a ``superstructure'', containing an occupied circuit in $\Ann_{\frac{1}{2} R^{(i+1)}, R^{(i+1)}}$ connected to $\din \Lambda^{(i)}$. We denote by $\calS$ the cluster of this superstructure inside $\Lambda^{(i)}$. Note that any $\dol{t}^{(i)}$-occupied circuit in $\Ann_{\frac{1}{2} r^{(i)}, r^{(i)}}$ whose existence is provided by $E_4^{(i)}$ must be contained in $\calS$. The event $E_1^{(i)}$ implies that $\calS \subseteq \lclus_{\Ball_{R^{(i)}}}$ (the largest cluster in $\Ball_{R^{(i)}}$ in the pure birth process), and so it follows from $E_2^{(i)}$ that no vertex of $\calS$ has been ignited yet at time $\dol{t}^{(i)}$. Hence, all vertices in $\calS$ will burn together, at a later time than $\dol{t}^{(i)}$.

We claim that $\calS$ burns in the time interval $(\dol{t}^{(i)}, \dul{t}^{(i)}]$. We are fine if it burns before time $\doul{t}^{(i)}$, so we can assume that $\calS$ is not yet burnt at time $\doul{t}^{(i)}$. Note that in $E_8^{(i)}$, all vertices in $\ul{\cluster}^{(i)} \cap \Ball_{\frac{1}{2} r^{(i)}}$ ($\subseteq \doul{\cluster}^{(i)} \cap \Ball_{\frac{1}{2} r^{(i)}}$) are connected to $\din \Ball_{r^{(i)}}$ at time $\doul{t}^{(i)}$ (in the configuration $\sigma^{[t_c - \dol{\ve}^{(i)}]}$, still). Hence, at time $\doul{t}^{(i)}$ they are connected to any $\dol{t}^{(i)}$-occupied circuit in $\Ann_{\frac{1}{2} r^{(i)}, r^{(i)}}$, which ``protects'' them from outside ignitions. More precisely, since no vertex of $\doul{\cluster}^{(i)}$ gets ignited in $[ t_c - \dol{\ve}^{(i)}, \doul{t}^{(i)} ]$ (from $E_6^{(i)}$ and $E_7^{(i)}$), the only way for a vertex in $\ul{\cluster}^{(i)} \cap \Ball_{\frac{1}{2} r^{(i)}}$ to burn before time $\doul{t}^{(i)}$ would be for the fire to come from outside the circuit in $\Ann_{\frac{1}{2} r^{(i)}, r^{(i)}}$. But this is not possible, since otherwise, the aforementioned $\dol{t}^{(i)}$-occupied circuit would already be burnt at time $\doul{t}^{(i)}$, and thus $\calS$ as well.

Finally, $E_9^{(i)}$ ensures that one vertex of $\ul{\cluster}^{(i)} \cap \Ball_{\frac{1}{2} r^{(i)}}$ gets ignited before time $\dul{t}^{(i)}$. Such a vertex belongs to $\calS$ at time $\doul{t}^{(i)}$, so $\calS$ also burns before time $\dul{t}^{(i)}$ (either because of this vertex, or at an earlier time due to another ignition). We deduce that in all cases, the cluster of $\calS$, which surrounds $0$, burns at a time $\tau^{(i)} \in (\dol{t}^{(i)}, \dul{t}^{(i)}]$, and leaves $0$ in an island $\Lambda^{(i+1)} \subseteq \Ball_{R^{(i+1)}}$. Moreover, any $\dul{t}^{(i)}$-vacant circuit provided by $E_5^{(i)}$ is still vacant at time $\tau^{(i)}$, which implies that $\Lambda^{(i+1)}$ contains $\Ball_{r^{(i+1)}}$.

We now show that for some $\kappa = \kappa(\ve) > 0$ (which does not depend on $i$), we have: for all $i =0, \ldots, j$,
\begin{equation} \label{eq:unif_lower_bound_FF}
\PP(E_k^{(i)}) \geq 1 - \frac{\kappa}{\big( \log \log \frac{1}{\zeta} \big)^2} \quad (k = 1, \ldots, 9).
\end{equation}
As in the proof of Theorem~\ref{thm:mainFP}, this will then allow us to use the union bound, over the $j+1$ steps.

Let $i \in \{0, \ldots, j\}$, and start with $E_1^{(i)}$. First, it follows from $\ol{t}^{(i)} = t_c + \ol{\ve}^{(i)} = \next_\zeta(R^{(i)})$ (see \eqref{eq:t_FFWoR}) and the definition of $\next_\zeta$ (see \eqref{eq:def_next_FF}) that
\begin{equation}
c_{\TT} \big( R^{(i)} \big)^2 \theta \big( \ol{t}^{(i)} \big) \ol{\ve}^{(i)} = \zeta^{-1}.
\end{equation}
Hence,
\begin{equation} \label{eq:E1_FFWoR1}
\big| \Ball_{R^{(i)}} \big| \theta \big( \dol{t}^{(i)} \big) \leq \frac{5}{4} \big( \ol{\ve}^{(i)} \zeta \big)^{-1}
\end{equation}
(we used that $\dol{t}^{(i)} \leq \ol{t}^{(i)}$, so $\theta(\dol{t}^{(i)}) \leq \theta(\ol{t}^{(i)})$). On the other hand, it follows from \eqref{eq:r_FFWoR}, $r^{(i)} \leq R^{(i)}$, and \eqref{eq:r_estimate2_FFWoR} that
\begin{equation} \label{eq:E1_FFWoR2}
\frac{L(\dol{t}^{(i)})}{R^{(i)}} = \bigg( \log \log \log \frac{1}{\zeta} \bigg)^{-4} \frac{R^{(i+1)}}{R^{(i)}} \leq \frac{R^{(i+1)}}{r^{(i)}} \leq \bigg( \frac{c'_4}{\log \frac{1}{\zeta}} \bigg)^{\frac{\xi}{2} (\aval - \ve')^i}.
\end{equation}
We can now deduce \eqref{eq:unif_lower_bound_FF} for $E_1^{(i)}$ from \eqref{eq:largest_cluster_quant}, combined with \eqref{eq:E1_FFWoR1} and \eqref{eq:E1_FFWoR2}.

We can proceed in a similar way for $E_6^{(i)}$. Using $\ul{t}^{(i)} = t_c + \ul{\ve}^{(i)} := \next_\zeta(r^{(i)})$ (from \eqref{eq:t_FFWoR}), we have
\begin{equation} \label{eq:lower_vol2}
c_{\TT} \big( r^{(i)} \big)^2 \theta \big( \ul{t}^{(i)} \big) \ul{\ve}^{(i)} = \zeta^{-1}
\end{equation}
(from \eqref{eq:def_next_FF}), which implies
\begin{equation} \label{eq:lower_vol3}
\big| \Ball_{r^{(i)}} \big| \theta \big( \doul{t}^{(i)} \big) \leq \frac{5}{4} \big( \ul{\ve}^{(i)} \zeta \big)^{-1}
\end{equation}
(using also $\doul{t}^{(i)} \leq \ul{t}^{(i)}$). In addition, $t_c < \dol{t}^{(i)} \leq \doul{t}^{(i)}$, so
\begin{equation} \label{eq:E68_FFWoR}
\frac{L(\doul{t}^{(i)})}{r^{(i)}} \leq \frac{L(\dol{t}^{(i)})}{r^{(i)}} \leq \frac{R^{(i+1)}}{r^{(i)}} \leq \bigg( \frac{c'_4}{\log \frac{1}{\zeta}} \bigg)^{\frac{\xi}{2} (\aval - \ve')^i}
\end{equation}
(as in \eqref{eq:E1_FFWoR2}). We thus obtain \eqref{eq:unif_lower_bound_FF} for $E_6^{(i)}$ from \eqref{eq:largest_cluster_quant}, \eqref{eq:lower_vol3} and \eqref{eq:E68_FFWoR}.

The estimate \eqref{eq:unif_lower_bound_FF} for $E_8^{(i)}$ comes from Proposition~\ref{prop:largest_cluster_impurities}, about the largest cluster in the presence of impurities, together with the stochastic domination provided by Lemma~\ref{lem:stoch_domin}. Indeed, we can use again \eqref{eq:E68_FFWoR}, and also the fact that \eqref{eq:lower_vol2} and \eqref{eq:lower_vol} imply together
\begin{equation}
\big| \Ball_{\frac{1}{2} r^{(i)}} \big| \theta \big( \doul{t}^{(i)} \big) \geq \frac{1}{4} \cdot \frac{\theta \big( \doul{t}^{(i)} \big)}{\theta \big( \ul{t}^{(i)} \big)} \cdot \frac{3}{4} \big( \ul{\ve}^{(i)} \zeta \big)^{-1} \geq \frac{3 \kappa'_1}{16} \bigg( \log \log \frac{1}{\zeta} \bigg)^{-\frac{1}{3}} \big( \ul{\ve}^{(i)} \zeta \big)^{-1}.
\end{equation}

Now, given clusters $\ol{\ol{\cluster}}^{(i)}$, $\doul{\cluster}^{(i)}$ and $\ul{\cluster}^{(i)}$, as in $E_1^{(i)}$, $E_6^{(i)}$ and $E_8^{(i)}$ respectively, the events $E_2^{(i)}$, $E_7^{(i)}$ and $E_9^{(i)}$ are just straightforward statements about exponentially distributed random variables. More precisely, if $\big| \ol{\ol{\cluster}}^{(i)} \big| \leq 2 (\ol{\ve}^{(i)} \zeta)^{-1}$ as in $E_1^{(i)}$, then the probability that at least one of its vertices gets ignited in the time interval $[t_c - \dol{\ve}^{(i)}, t_c + \dol{\ve}^{(i)}]$ is at most
\begin{equation}
2 \dol{\ve}^{(i)} \zeta \cdot 2 (\ol{\ve}^{(i)} \zeta)^{-1} = 4 \, \frac{\dol{\ve}^{(i)}}{\ol{\ve}^{(i)}} = \frac{4}{\big( \log \log \frac{1}{\zeta} \big)^2}.
\end{equation}
This proves \eqref{eq:unif_lower_bound_FF} for $E_2^{(i)}$. Estimating the probability of $E_7^{(i)}$ leads to a very similar computation, and so \eqref{eq:unif_lower_bound_FF} in this case, using that $[ t_c - \dol{\ve}^{(i)}, \doul{t}^{(i)} ]$ has length $\dol{\ve}^{(i)} + \doul{\ve}^{(i)} \leq 2 \doul{\ve}^{(i)}$.

In the other direction, if $\big| \ul{\cluster}^{(i)} \cap \Ball_{\frac{1}{2} r^{(i)}} \big| \geq \frac{\kappa'_1}{8} ( \log \log \frac{1}{\zeta} )^{-\frac{1}{3}} ( \ul{\ve}^{(i)} \zeta )^{-1}$ as in $E_8^{(i)}$, then the probability that at least one of its vertices is ignited in the time interval $[t_c + \doul{\ve}^{(i)}, t_c + \dul{\ve}^{(i)}]$, which has length $\dul{\ve}^{(i)} - \doul{\ve}^{(i)} \geq \big( \log \log \frac{1}{\zeta} - 1 \big) \ul{\ve}^{(i)}$ (using \eqref{eq:ttt_FFWoR}), is at least
\begin{equation}
1 - e^{- ( \log \log \frac{1}{\zeta} - 1) \ul{\ve}^{(i)} \zeta \cdot \frac{\kappa'_1}{8} ( \log \log \frac{1}{\zeta} )^{-\frac{1}{3}} (\ul{\ve}^{(i)} \zeta)^{-1}} \geq 1 - e^{ - \frac{\kappa_1}{16} (\log \log \frac{1}{\zeta})^{\frac{2}{3}}} \geq 1 - \frac{C}{(\log \log \frac{1}{\zeta} )^2},
\end{equation}
for some universal constant $C > 0$. We deduce \eqref{eq:unif_lower_bound_FF} for $E_9^{(i)}$.

The event $E_5^{(i)}$ only involves the pure birth process, and we can write
\begin{align*}
\PP(E_5^{(i)}) & \geq \PP_{p(\dul{t}^{(i)})} \Big( \circuitevent^* \Big( \Ann_{(\log \log \frac{1}{\zeta})^{-24} L(\dul{t}^{(i)}), L(\dul{t}^{(i)})} \Big) \Big)\\[1.5mm]
& = 1 - \PP_{p(\dul{t}^{(i)})} \Big( \arm_1 \Big( \Ann_{(\log \log \frac{1}{\zeta})^{-24} L(\dul{t}^{(i)}), L(\dul{t}^{(i)})} \Big) \Big)\\[1.5mm]
& \geq 1 - C \, \pi_1 \bigg( \bigg( \log \log \frac{1}{\zeta} \bigg)^{-24} L \big( \dul{t}^{(i)} \big), L \big( \dul{t}^{(i)} \big) \bigg),
\end{align*}
from which \eqref{eq:unif_lower_bound_FF} follows, using \eqref{eq:uniform_arm_exp}.

Finally, we can obtain \eqref{eq:unif_lower_bound_FF} for $E_3^{(i)}$ and $E_4^{(i)}$ from the stretched exponential decay property for percolation with impurities, Proposition~\ref{prop:exp_decay_impurities}, combined with Lemma~\ref{lem:stoch_domin}. For this purpose, observe that $m = L(t_c - \dol{\ve}^{(i)}) \asymp L(\dol{t}^{(i)})$, and, from the earlier observation that $r^{(i)} \geq R^{(i+1)}$,
\begin{equation} \label{eq:ineq_ri_Lti}
\frac{r^{(i)}}{L(\dol{t}^{(i)})} \geq \frac{R^{(i+1)}}{L(\dol{t}^{(i)})} = \bigg( \log \log \log \frac{1}{\zeta} \bigg)^4
\end{equation}
(the equality comes from \eqref{eq:r_FFWoR}).

We have thus checked \eqref{eq:unif_lower_bound_FF} for all $k = 1, \ldots, 9$. It then follows immediately from the union bound that
\begin{equation}
\PP \Bigg( \bigcup_{\substack{0 \leq i \leq j\\ 1 \leq k \leq 9}} E_k^{(i)} \Bigg) \geq 1 - (j+1) \cdot 9 \cdot \frac{\kappa}{(\log \log \frac{1}{\zeta})^2} \geq 1 - \frac{\eta}{10}
\end{equation}
for all $\zeta$ sufficiently small (using \eqref{eq:bound_J_FFWoR2} and $j \leq J$).

Actually, we have to be a bit careful in the final step $i = j$, as in the proof of Theorem~\ref{thm:mainFP} (see the end of Step~3): we discard $E_5^{(j+1)}$ if $r^{(j+1)} < 1$ (and the other events can be left unaffected). In any case, we reach a situation where $0$ lies in an island $\Lambda^{(j+1)}$, with $\Lambda^{(j+1)} \subseteq \Ball_{R^{(j+1)}}$. Additionally, $\Ball_{r^{(j+1)}} \subseteq \Lambda^{(j+1)}$, but only if $r^{(j+1)} \geq 1$ ($\Lambda^{(j+1)}$ may even be empty otherwise).

\bigskip

\textbf{Step 4}: We now conclude the iterative procedure. Recall that by definition,
$$r^{(j)} \geq \frac{1}{\sqrt{\zeta}} \quad \text{and} \quad r^{(j+1)} < \frac{1}{\sqrt{\zeta}},$$
and similar inequalities hold true for $R^{(J)}$ and $R^{(J+1)}$. Moreover, $J \in \{j, j+1\}$.

At this stage, $j+1$ clusters have already burnt around $0$ in the domain $\Lambda^{(0)} = \Lambda^+$. Analogously to frozen percolation, we show that only a bounded number of additional burnt clusters can arise around $0$ inside $\Lambda^{(j+1)}$ (at most two, as we are going to explain). We distinguish the following cases.
\begin{itemize}
\item \ul{Case 1}: $J = j$, i.e. $R^{(j+1)} < \frac{1}{\sqrt{\zeta}}$. We have $\Lambda^{(j+1)} \subseteq \Ball_{R^{(j+1)}} \subseteq \Ball_{\frac{1}{\sqrt{\zeta}}}$. Clearly, we can stop if $\Lambda^{(j+1)} = \emptyset$, and in this case, $0$ is surrounded by exactly $j+1$ burnt clusters. We thus assume $\Lambda^{(j+1)} \neq \emptyset$, and consider the time $t^{(j+1)} = t_c + \ve^{(j+1)}$, where $\ve^{(j+1)} = \ve^{(j+1)}(\eta) > 0$ satisfies
\begin{equation} \label{eq:last_step_FF1}
\theta \big( t_c + \ve^{(j+1)} \big) = \frac{1}{2 t_c c_{\TT}} \cdot \frac{\eta}{100}.
\end{equation}
We use the event
$$E^{(j+1)} := \Big\{ \text{at $t^{(j+1)}$, } \circuitevent \big( \Ann_{\frac{1}{2} R^{(j+2)}, R^{(j+2)}} \big) \text{ occurs, and } \dout \Ball_{\frac{1}{2} R^{(j+2)}} \lra \infty \Big\}$$
with $R^{(j+2)} := c^{(j+2)} L(t^{(j+1)})$, where $c^{(j+2)} = c^{(j+2)}(\eta) > 0$ is chosen sufficiently large so that $\PP(E^{(j+1)}) \geq 1 - \frac{\eta}{100}$ (this is possible, thanks to \eqref{eq:a-priori}). We denote by $\circuit^{(j+1)}$ the corresponding circuit. The number of vertices connected to $\circuit^{(j+1)}$ in $\Lambda^{(j+1)}$ is at most the volume of the largest cluster $\lclus$ in $\Ball_{\frac{1}{\sqrt{\zeta}}}$ at time $t^{(j+1)}$, which satisfies
$$\big| \lclus \big| \leq \frac{3}{2} \big| \Ball_{\frac{1}{\sqrt{\zeta}}} \big| \theta \big( t^{(j+1)} \big) \sim \frac{3}{2} c_{\TT} \bigg( \frac{1}{\sqrt{\zeta}} \bigg)^2 \theta \big( t^{(j+1)} \big)$$
with probability $\geq 1 - \frac{\eta}{100}$ (from \eqref{eq:largest_cluster}). Hence, \eqref{eq:last_step_FF1} ensures that with probability $\geq 1 - \frac{\eta}{100}$, no vertex connected to $\circuit^{(j+1)}$ gets ignited before time $t^{(j+1)}$. This circuit thus burns at a later time, leaving $0$ in a microscopic island $\Lambda^{(j+2)} \subseteq \Ball_{R^{(j+2)}}$. Either $\Lambda^{(j+2)} = \emptyset$, in which case $0$ is only surrounded by this additional burnt cluster, or $\Lambda^{(j+2)} \neq \emptyset$, in which case all vertices inside $\Lambda^{(j+2)}$ burn together at an even later time (observe that $R^{(j+2)}$ does not depend on $\zeta$, only on $\eta$), so that $0$ is surrounded by two additional burnt clusters. We deduce that the total number of burnt clusters around $0$ in $\Lambda^+$ is either $j+1$, $j+2$, or $j+3$.

\item \ul{Case 2}: $J = j+1$, i.e. $R^{(j+1)} \geq \frac{1}{\sqrt{\zeta}}$. In this case, \eqref{eq:r_estimate2_FFWoR} implies
$$10 R^{(j+2)} \leq r^{(j+1)} < \frac{1}{\sqrt{\zeta}}$$
for all $\zeta$ small enough, so in particular $\Lambda^{(j+1)} \neq \emptyset$. First, we use the same events $E_1^{(j+1)}$, $E_2^{(j+1)}$, and $E_3^{(j+1)}$ as in the previous steps, involving $R^{(j+1)}$ and $R^{(j+2)}$ (over the time interval $[ t_c - \dol{\ve}^{(j+1)}, t_c + \dol{\ve}^{(j+1)} ]$). For similar reasons as in Step~3, one more burning takes place in $\Lambda^{(j+1)}$, leaving $0$ in an island $\Lambda^{(j+2)} \subseteq \Ball_{R^{(j+2)}}$. Obviously the process ends if $\Lambda^{(j+2)}$ is empty. Otherwise, let $t^{(j+2)} = t^{(j+2)}(\eta) > t_c$ be defined by $\theta(t^{(j+2)}) = 1 - \frac{\eta}{100}$, i.e. such that the event
$$E^{(j+2)} := \Big\{ \text{at $t^{(j+2)}$, } 0 \lra \infty \Big\}$$
has probability $\geq 1 - \frac{\eta}{100}$. Note that if $\zeta$ is sufficiently small, no vertex inside $\Lambda^{(j+2)}$ gets ignited before time $t^{(j+2)}$, with probability $\geq 1 - \frac{\eta}{100}$. Indeed, it follows from $|\Lambda^{(j+2)}| \leq |\Ball_{R^{(j+2)}}|$, $r^{(j+1)} < \frac{1}{\sqrt{\zeta}}$, and \eqref{eq:r_estimate2_FFWoR}, that
$$\zeta \cdot |\Lambda^{(j+2)}| \leq \frac{1}{t^{(j+2)}} \cdot \frac{\eta}{100}$$
for all $\zeta$ small enough. We deduce that the occupied path from $0$ to $\din \Lambda^{(j+2)}$, provided by $E^{(j+2)}$, burns as a whole, after time $t^{(j+2)}$. Hence, exactly one additional cluster surrounding $0$ burns inside $\Lambda^{(j+2)}$. In any case, the total number of clusters surrounding $0$ which burn in $\Lambda^+$ is either $j+2$ or $j+3$.

\end{itemize}

By combining these two cases together, we obtain that the number of burnt clusters around $0$ in $\Lambda^+$ is between $J+1$ and $J+3$. This allows us to complete the proof, thanks to \eqref{eq:apriori_circuits_FF} and \eqref{eq:bound_J_FFWoR2}.

\end{proof}

\begin{remark} \label{rem:hypothesis_box}
As we mentioned in Section~\ref{sec:perc_impurities}, Proposition~\ref{prop:largest_cluster_impurities} improves the analogous result from \cite{BN2018} (Proposition~5.5) in two ways, in addition to providing a quantitative lower bound on the probability. Firstly, the boxes in which it can be applied are not required to have a side length $n \gg m (\log m)^2$. This is used for the events $E_8^{(i)}$, for small values of $i$. Indeed, for $i=0$ in particular, we have, roughly speaking,
$$n = r^{(0)} \simeq \frac{m_\infty(\zeta)}{(\log \frac{1}{\zeta})^{\bar{\alpha}}}$$
(where $\bar{\alpha}$ is ``comparable'' to $\alpha$), and the corresponding value of $m$ is
$$m = L(t_c - \dol{\ve}^{(0)}) \asymp L(t_c + \dol{\ve}^{(0)}) = \bigg( \log \log \frac{1}{\zeta} \bigg)^{\frac{8}{3} + o(1)} L(t_c + \ol{\ve}^{(0)}).$$
Since $L(t_c + \ol{\ve}^{(0)}) = \next_\zeta (R^{(0)})$, we have $m = \frac{m_\infty(\zeta)}{(\log \frac{1}{\zeta})^{\frac{96}{41} \bar{\alpha} + o(1)}}$ (from \eqref{eq:compar_minf}). Hence,
$$n = m \cdot \bigg( \log \frac{1}{\zeta} \bigg)^{(\frac{96}{41} - 1) \bar{\alpha} + o(1)} = m (\log m)^{\beta + o(1)},$$
where $\beta = \frac{55}{41} \, \bar{\alpha}$ (using that $\log m = \log m_\infty(\zeta) + O(\log \log \frac{1}{\zeta}) = (\frac{48}{55} + o(1)) \log \frac{1}{\zeta}$, from \eqref{eq:exp_t_infty_FF}).

Secondly, we apply Proposition~\ref{prop:largest_cluster_impurities} at a time when the characteristic length is
$$L(\doul{t}^{(i)}) = \bigg( \log \log \frac{1}{\zeta} \bigg)^{\frac{8}{3} + o(1)} L(t_c + \ul{\ve}^{(i)}) = \bigg( \log \log \frac{1}{\zeta} \bigg)^{\frac{8}{3} + o(1)} \next_\zeta (r^{(i)}),$$
while
$$m = L(t_c - \dol{\ve}^{(i)}) \asymp L(t_c + \dol{\ve}^{(i)}) = \bigg( \log \log \frac{1}{\zeta} \bigg)^{\frac{8}{3} + o(1)} L(t_c + \ol{\ve}^{(i)}) = \bigg( \log \log \frac{1}{\zeta} \bigg)^{\frac{8}{3} + o(1)} \next_\zeta (R^{(i)}).$$
These two quantities may get far apart as $i$ increases, our control on $\frac{r^{(i)}}{R^{(i)}}$ becoming less and less accurate (in a similar way as for frozen percolation, see Remark~\ref{rem:separation_scales}).
\end{remark}

\subsection{Proof of Proposition \ref{prop:large_burnt_clusters}} \label{sec:proof_FF2}

The proof above can be adapted in order to derive Proposition \ref{prop:large_burnt_clusters}.

\begin{proof}[Proof of Proposition \ref{prop:large_burnt_clusters}]
Let $\ve, \eta > 0$. We follow the iterative procedure from the proof of Theorem~\ref{thm:mainFF}, i.e. we use the same Steps~1--3. Since we are aiming at a lower bound, we do not need Step~4, and we consider only those burnt clusters produced by steps $i = i_{\textrm{min}}, \ldots, i_{\textrm{max}} - 2$, where
$$i_{\textrm{min}} := \bigg \lceil \frac{4}{9} \ve \log \log \frac{1}{\zeta} \bigg \rceil \quad \text{and} \quad i_{\textrm{max}} := \bigg \lfloor \bigg( \avalnb - \frac{4}{9} \ve \bigg) \log \log \frac{1}{\zeta} \bigg \rfloor.$$
Note that for all $\zeta$ sufficiently small, $i_{\textrm{max}} \leq j$ (from \eqref{eq:bound_J_FFWoR2}, and $j \geq J-1$), and the number of such steps is at least $(\avalnb - \ve) \log \log \frac{1}{\zeta}$.

For all $i \in \{i_{\textrm{min}}, \ldots, i_{\textrm{max}}\}$, \eqref{eq:r_estimate_FFWoR} ensures that
\begin{equation} \label{eq:intermediate_scales}
\bigg( \frac{c_1}{\log \frac{1}{\zeta}} \bigg)^{\alpha (\aval + \ve')^{( \avalnb - \frac{4}{9} \ve ) \log \log \frac{1}{\zeta}}} \leq \frac{R^{(i)}}{m_{\infty}} \leq \bigg( \frac{c'_1}{\log \frac{1}{\zeta}} \bigg)^{\frac{\alpha}{2} (\aval - \ve')^{\frac{4}{9} \ve \log \log \frac{1}{\zeta}}},
\end{equation}
for some $c_1, c'_1 >0$ (which do not depend on $\zeta$). Recall that the constant $\ve' > 0$ appearing in \eqref{eq:r_estimate_FFWoR} was later chosen small enough so that \eqref{eq:bound_J_FFWoR2} holds. We can assume that
$$\avalnb - \frac{\ve}{3} < \frac{1}{\log(\aval + \ve')} \quad \text{and} \quad \frac{1}{\log(\aval - \ve')} < \avalnb + \frac{\ve}{3}.$$

A short computation shows that for all $\zeta$ small enough, the r.h.s. of \eqref{eq:intermediate_scales} is at most $e^{- ( \log \frac{1}{\zeta} )^{\xi'}}$, while the l.h.s. is at least $e^{- ( \log \frac{1}{\zeta} )^{1 - \xi''}}$, with
$$\xi' = \frac{\ve}{3} \log (\aval - \ve') > 0 \quad \text{and} \quad \xi'' = 1 - \bigg( \avalnb - \frac{\ve}{3} \bigg) \log (\aval + \ve') > 0.$$
Using that $R^{(i+1)} \leq r^{(i)} \leq R^{(i)}$, we obtain the following: for all $i \in \{i_{\textrm{min}}, \ldots, i_{\textrm{max}} - 1\}$,
\begin{equation} \label{eq:intermediate_scales2}
e^{- ( \log \frac{1}{\zeta} )^{1 - \xi''}} \leq \frac{r^{(i)}}{m_{\infty}} \leq \frac{R^{(i)}}{m_{\infty}} \leq e^{- ( \log \frac{1}{\zeta} )^{\xi'}}.
\end{equation}
We then estimate the volume of the clusters surrounding $0$ which burn during the steps $i \in I := \{i_{\textrm{min}}, \ldots, i_{\textrm{max}} - 2\}$. Let
\begin{equation} \label{eq:def_gamma}
\xi = \frac{1}{2} (\xi' \wedge \xi''),
\end{equation}
and recall that $V_{\infty}(\zeta) = m_{\infty}(\zeta)^2 \pi_1(m_{\infty}(\zeta))$ (see \eqref{eq:def_Vinfty}). We use the same events $E_1^{(i)}$--$E_9^{(i)}$ ($0 \leq i \leq j$) as in the iterative construction, as well as two additional events when $i \in I$:
\begin{itemize}
\item $E_{10}^{(i)} := \Big\{$at $\dul{t}^{(i)}$, $\Big| \lclus_{\Ball_{R^{(i)}}} \Big| \leq e^{- ( \log \frac{1}{\zeta} )^{\xi}} V_{\infty}(\zeta)\Big\}$,

\item $E_{11}^{(i)} := \Big\{$at $\dol{t}^{(i)}$, in $\sigma^{[t_c - \dol{\ve}^{(i)}]}$, $\Big| \lclus_{\Ball_{r^{(i)}}} \cap \Ball_{\frac{1}{2} r^{(i)}} \Big| \geq e^{- ( \log \frac{1}{\zeta} )^{1 - \xi}} V_{\infty}(\zeta)$ and $\circuitarm \big( \Ann_{\frac{1}{2} r^{(i)}, r^{(i)}} \, | \, \lclus_{\Ball_{r^{(i)}}} \big)$ occurs$\Big\}$.
\end{itemize}
These two events are enough to ensure that the cluster $\cluster^{(i)}$ around $0$ burning in the $i$th step has the right volume. Indeed, $E_{11}^{(i)}$ ensures that $\cluster^{(i)}$ contains the intersection of $\lclus_{\Ball_{r^{(i)}}}$ at time $\dol{t}^{(i)}$ with $\Ball_{\frac{1}{2} r^{(i)}}$, and $E_{10}^{(i)}$ implies that $\cluster^{(i)}$ has a volume at most $e^{- ( \log \frac{1}{\zeta} )^{\xi}} V_{\infty}(\zeta)$: this is because $|\cluster^{(i)}|$ cannot be larger than the volume of the largest cluster in $\Ball_{R^{(i)}}$ at time $\dul{t}^{(i)}$ (in the pure birth process), since $\cluster^{(i)}$ burns before that time.

First, observe that
\begin{equation} \label{eq:ratio_cmax}
\frac{L(\dul{t}^{(i)})}{R^{(i)}} \leq \frac{L(\dol{t}^{(i)})}{r^{(i)}} \leq \bigg( \log \log \log \frac{1}{\zeta} \bigg)^{-4},
\end{equation}
using that $t_c < \dol{t}^{(i)} \leq \dul{t}^{(i)}$ and $r^{(i)} \leq R^{(i)}$ for the first inequality, and \eqref{eq:ineq_ri_Lti} for the second one. We have
$$\theta(\dul{t}^{(i)}) \big| \Ball_{R^{(i)}} \big| \asymp (R^{(i)})^2 \pi_1(L(\dul{t}^{(i)})) \leq (R^{(i)})^2 \pi_1(r^{(i+1)})$$
(since $L(\dul{t}^{(i)}) \geq r^{(i+1)}$, from \eqref{eq:r_FFWoR}). We can write
$$(R^{(i)})^2 \pi_1(r^{(i+1)}) = \bigg( \frac{R^{(i)}}{m_{\infty}} \bigg)^2 \frac{\pi_1(r^{(i+1)})}{\pi_1(m_{\infty})} \cdot (m_{\infty})^2 \pi_1(m_{\infty}).$$
We deduce from \eqref{eq:intermediate_scales2} and \eqref{eq:def_gamma}, combined with \eqref{eq:quasi_mult} and \eqref{eq:uniform_arm_exp}, that
$$\theta(\dul{t}^{(i)}) \big| \Ball_{R^{(i)}} \big| \leq \frac{1}{2} \cdot \frac{V_{\infty}(\zeta)}{e^{( \log \frac{1}{\zeta} )^{\xi}}}$$
for all $\zeta$ small enough (uniformly in $i \in I$). Hence, \eqref{eq:largest_cluster_quant} implies that
\begin{equation} \label{eq:proba_E10}
\PP(E_{10}^{(i)}) \geq 1 - C \, \frac{L(\dul{t}^{(i)})}{R^{(i)}} \geq 1 - \frac{\kappa}{\big( \log \log \frac{1}{\zeta} \big)^2}
\end{equation}
for some $\kappa$ which does not depend on $i$, where we also used \eqref{eq:ratio_cmax}.

We can get the same lower bound for $\PP(E_{11}^{(i)})$ by proceeding in a similar way. In this case we have, for all $i \in I$,
$$\theta(\dol{t}^{(i)}) \big| \Ball_{\frac{1}{2} r^{(i)}} \big| \asymp (r^{(i)})^2 \pi_1(L(\dol{t}^{(i)})) \geq (r^{(i)})^2 \pi_1(R^{(i+1)})$$
(using $L(\dol{t}^{(i)}) \leq R^{(i+1)}$, from \eqref{eq:r_FFWoR} again), so
$$\theta(\dol{t}^{(i)}) \big| \Ball_{\frac{1}{2} r^{(i)}} \big| \geq 2 \cdot \frac{V_{\infty}(\zeta)}{e^{( \log \frac{1}{\zeta} )^{1 - \xi}}}$$
for all $\zeta$ sufficiently small (from \eqref{eq:intermediate_scales2} and \eqref{eq:def_gamma} again, together with \eqref{eq:quasi_mult} and \eqref{eq:uniform_arm_exp}). We deduce from Proposition~\ref{prop:largest_cluster_impurities}, and then \eqref{eq:ratio_cmax}, that for some $\kappa$ uniform in $i$,
\begin{equation} \label{eq:proba_E11}
\PP(E_{11}^{(i)}) \geq 1 - C \, \frac{L(\dol{t}^{(i)})}{r^{(i)}} \geq 1 - \frac{\kappa}{\big( \log \log \frac{1}{\zeta} \big)^2}.
\end{equation}

We can now complete the proof of Proposition \ref{prop:large_burnt_clusters}, by combining \eqref{eq:proba_E10} and \eqref{eq:proba_E11}, for all $i \in I$, with Steps~1--3 in the proof of Theorem~\ref{thm:mainFF}.
\end{proof}

\appendix

\section{Appendix: additional proofs}

\subsection{Largest cluster in a box: Bernoulli percolation} \label{sec:proof_largest}

We first establish Lemma~\ref{lem:largest_cluster_quant}, about the largest cluster in a box $\Ball_n$ for Bernoulli percolation. A non-quantitative version of this result was proved in \cite{BCKS2001}, and it is not very difficult to obtain Lemma~\ref{lem:largest_cluster_quant} by following similar arguments (see Section~5.2 of \cite{BCKS2001}, and also the proof of Lemma 4.1 in \cite{BKN2015}).

When presenting a detailed proof of Lemma~\ref{lem:largest_cluster_quant}, our goal is twofold. First, we include it for the reader's convenience, since the quantitative version is not stated in \cite{BCKS2001}, and ingredients of the proof are located in various places in that paper. But we also strived to write it so as to make its adaptation to the case of percolation with impurities as straightforward as possible (we handle this case in Section~\ref{sec:proof_largest_imp}).

As an input, we use that for some universal constants $C_1, C_2$, we have: for all $p > p_c$,
\begin{equation} \label{eq:chi_fin}
\chi^{\textrm{fin}}(p) := \EE_p \big[ |\cluster(0)| \mathbbm{1}_{|\cluster(0)| < \infty} \big] \leq C_1 L(p)^2 \theta(p)^2
\end{equation}
and
\begin{equation} \label{eq:chi_cov}
\chi^{\textrm{cov}}(p) := \sum_{v \in V} \textrm{Cov}_p(\mathbbm{1}_{0 \lra \infty}, \mathbbm{1}_{v \lra \infty}) \leq C_2 L(p)^2 \theta(p)^2.
\end{equation}
Each of these two bounds follows from a summation argument relying on \eqref{eq:exp_decay}. The bound \eqref{eq:chi_fin} is proved in Theorem~3 of \cite{Ke1987} (note that $\chi^{\textrm{fin}}(p) = \sum_{v \in V} \PP_p(0 \lra v, 0 \not \lra \infty)$), while \eqref{eq:chi_cov} is established in Section~6.4 of \cite{BCKS2001}.

\begin{proof}[Proof of Lemma \ref{lem:largest_cluster_quant}]
Let $\ve > 0$. We may assume that $\ve < \frac{1}{2}$. We want to prove that uniformly in $p > p_c$ and $n \geq 1$,
$$\PP_p \left( \frac{|\lclus_{\Ball_n}|}{\theta(p) |\Ball_n|} \notin (1 - \ve, 1 + \ve) \right) \leq C \, \frac{L(p)}{n}$$
for some $C = C(\ve)$. For that, we list the clusters in $\Ball_n$ in order of decreasing volume, and distinguish them according to whether they are contained in the infinite cluster $\Cinf$ (on the whole lattice $\TT$), or not: $|\cluster_{\Ball_n,\infty}^{(1)}| \geq |\cluster_{\Ball_n,\infty}^{(2)}| \geq \ldots$ and $|\cluster_{\Ball_n,<\infty}^{(1)}| \geq |\cluster_{\Ball_n,<\infty}^{(2)}| \geq \ldots$ (using e.g. the lexicographic order for clusters having the same volume). Clearly, we have $|\lclus_{\Ball_n}| = |\cluster_{\Ball_n,\infty}^{(1)}|$ or $|\cluster_{\Ball_n,<\infty}^{(1)}|$.

First,
\begin{equation} \label{eq:exp_cinf_ball}
\EE_p \big[ |\Cinf \cap \Ball_n | \big] = \EE_p \Bigg[ \sum_{v \in \Ball_n} \mathbbm{1}_{v \lra \infty} \Bigg] = \sum_{v \in \Ball_n} \PP_p(v \lra \infty) = |\Ball_n| \, \theta(p)
\end{equation}
by translation invariance. We also have
\begin{align*}
\textrm{Var}_p \big( |\Cinf \cap \Ball_n | \big) = \textrm{Var}_p \Bigg( \sum_{v \in \Ball_n} \mathbbm{1}_{v \lra \infty} \Bigg) & = \sum_{v,v' \in \Ball_n} \textrm{Cov}_p(\mathbbm{1}_{v \lra \infty}, \mathbbm{1}_{v' \lra \infty})\\[1mm]
& \leq \sum_{v \in \Ball_n} \sum_{v' \in V} \textrm{Cov}_p(\mathbbm{1}_{v \lra \infty}, \mathbbm{1}_{v' \lra \infty}) = |\Ball_n| \, \chi^{\textrm{cov}}(p),
\end{align*}
where the inequality follows from the fact that $\textrm{Cov}_p(\mathbbm{1}_{v \lra \infty}, \mathbbm{1}_{v' \lra \infty}) \geq 0$ for all $v, v' \in V$ (by the FKG inequality), and the last equality from translation invariance and the definition of $\chi^{\textrm{cov}}$. Combined with \eqref{eq:chi_cov}, this yields
\begin{equation} \label{eq:var_cinf_ball}
\textrm{Var}_p \big( |\Cinf \cap \Ball_n | \big) \leq  |\Ball_n| \cdot C_2 L(p)^2 \theta(p)^2 \leq C_3 \bigg( \frac{L(p)}{n} \bigg)^2 \big( |\Ball_n| \theta(p) \big)^2.
\end{equation}
We deduce from \eqref{eq:exp_cinf_ball} and \eqref{eq:var_cinf_ball} that
\begin{equation} \label{eq:cinf_ball1}
\PP_p \big( (1 - \ve) |\Ball_n| \theta(p) \leq |\Cinf \cap \Ball_n | \leq (1 + \ve) |\Ball_n| \theta(p) \big) \geq 1 - \ve^{-2} C_3 \bigg( \frac{L(p)}{n} \bigg)^2.
\end{equation}
Moreover, a similar bound holds for $\Ball_{n - \sqrt{n L(p)}}$ (we assume, without loss of generality, that $\frac{L(p)}{n} \leq \frac{1}{4}$, so that $\sqrt{n L(p)} \leq \frac{n}{2}$):
\begin{align}
\PP_p \bigg( \bigg( 1 - \frac{\ve}{2} \bigg) & \Big| \Ball_{n - \sqrt{n L(p)}} \Big| \theta(p) \leq \Big| \Cinf \cap \Ball_{n - \sqrt{n L(p)}} \Big| \leq \bigg(1 + \frac{\ve}{2} \bigg) \Big| \Ball_{n - \sqrt{n L(p)}} \Big| \theta(p) \bigg) \nonumber \\[1mm]
& \geq 1 - \bigg( \frac{\ve}{2} \bigg)^{-2} C_3 \bigg( \frac{L(p)}{n - \sqrt{n L(p)}} \bigg)^2 \geq 1 - 16 \ve^{-2} C_3 \bigg( \frac{L(p)}{n} \bigg)^2. \label{eq:cinf_ball2}
\end{align}
Note that if the event $\circuitevent \Big( \Ann_{n - \sqrt{n L(p)}, n} \Big)$ occurs, then all vertices in $\Cinf \cap \Ball_{n - \sqrt{n L(p)}}$ are connected inside $\Ball_n$, so belong to the same cluster. Moreover, it follows easily from \eqref{eq:exp_decay} that
\begin{equation} \label{eq:cinf_circuit}
\PP_p \Big( \circuitevent \Big( \Ann_{n - \sqrt{n L(p)}, n} \Big) \Big) \geq 1 - C_4 \bigg( \frac{n}{L(p)} \bigg)^{1/2} e^{- C_5 \big( \frac{n}{L(p)} \big)^{1/2}} \geq 1 - C_6 \, \frac{L(p)}{n}.
\end{equation}
On the other hand,
\begin{equation} \label{eq:cfin}
\EE_p \Big[ \big| \cluster_{\Ball_n,<\infty}^{(1)} \big| \Big] \leq |\Ball_n|^{1/2} L(p) \theta(p) + \EE_p \bigg[ \big| \cluster_{\Ball_n,<\infty}^{(1)} \big| \mathbbm{1}_{\big| \cluster_{\Ball_n,<\infty}^{(1)} \big| \geq |\Ball_n|^{1/2} L(p) \theta(p)} \bigg],
\end{equation}
and
\begin{align}
\EE_p \bigg[ \big| \cluster_{\Ball_n,<\infty}^{(1)} \big| \mathbbm{1}_{\big| \cluster_{\Ball_n,<\infty}^{(1)} \big| \geq |\Ball_n|^{1/2} L(p) \theta(p)} \bigg] & = \EE_p \Bigg[ \sum_{v \in \Ball_n} \mathbbm{1}_{\cluster(v) = \cluster_{\Ball_n,<\infty}^{(1)}} \mathbbm{1}_{| \cluster(v) | \geq |\Ball_n|^{1/2} L(p) \theta(p), v \not \lra \infty} \Bigg] \nonumber \\[2mm]
& \leq |\Ball_n| \, \PP_p \big( |\cluster(0)| \geq  |\Ball_n|^{1/2} L(p) \theta(p), 0 \not \lra \infty \big) \label{eq:cfin2}
\end{align}
(we used again translation invariance). Using Markov's inequality, we get
\begin{equation} \label{eq:cfin3}
\PP_p( |\cluster(0)| \geq |\Ball_n|^{1/2} L(p) \theta(p), 0 \not \lra \infty) \leq \frac{\chi^{\textrm{fin}}(p)}{|\Ball_n|^{1/2} L(p) \theta(p)} \leq C_7 \, \frac{L(p)}{n} \, \theta(p),
\end{equation}
from the definition of $\chi^{\textrm{fin}}$, and then \eqref{eq:chi_fin}. Combining \eqref{eq:cfin}, \eqref{eq:cfin2} and \eqref{eq:cfin3} gives
\begin{equation}
\EE_p \Big[ \big| \cluster_{\Ball_n,<\infty}^{(1)} \big| \Big] \leq C_8 \, \frac{L(p)}{n} \, |\Ball_n| \, \theta(p).
\end{equation}
Hence,
\begin{equation} \label{eq:cfin4}
\PP_p \Big( \big| \cluster_{\Ball_n,<\infty}^{(1)} \big| \leq \frac{1}{2} |\Ball_n| \theta(p) \Big) \geq 1 - 2 C_8 \, \frac{L(p)}{n}.
\end{equation}
Finally, we conclude by noting that if the events in the l.h.s. of \eqref{eq:cinf_ball1}, \eqref{eq:cinf_ball2}, \eqref{eq:cinf_circuit} and \eqref{eq:cfin4} all occur, which has a probability at least $1 - C_9 \, \frac{L(p)}{n}$ for some $C_9 = C_9(\ve)$, then
\begin{align*}
(1 - \ve) |\Ball_n| \theta(p) \leq \bigg( 1 - \frac{\ve}{2} \bigg) \Big| \Ball_{n - \sqrt{n L(p)}} \Big| \theta(p) \leq \Big| \Cinf \cap \Ball_{n - \sqrt{n L(p)}} \Big| & \leq \big|\cluster_{\Ball_n,\infty}^{(1)} \big|\\[1mm]
& \leq |\Cinf \cap \Ball_n | \leq (1 + \ve) |\Ball_n| \theta(p)
\end{align*}
and $|\lclus_{\Ball_n}| = |\cluster_{\Ball_n,\infty}^{(1)}|$.
\end{proof}

\subsection{Near-critical percolation with impurities} \label{sec:proof_impurities}

We now explain how to obtain the results stated in Section~\ref{sec:perc_impurities}, for our generalized version of percolation with impurities, i.e. under Assumption~\ref{ass:impurities}. Recall that we denote by $c$, $\gamma$ and $\upsilon$ the quantities in the r.h.s. of \eqref{eq:impurities_assumption}. For this purpose, we need to adapt the reasonings in Sections 3--5 of \cite{BN2018}.

More specifically, we first state preliminary lemmas about the existence of crossing holes, in Section~\ref{sec:app_hole}. We then show, in Section~\ref{sec:app_four}, how to extend a proof from \cite{BN2018} that some (enlarged) four-arm event is ``stable'' under the addition of impurities. This property is instrumental to derive further stability results, and we discuss consequences in Sections~\ref{sec:app_box} and \ref{sec:app_exp}, in particular Propositions~\ref{prop:crossing_impurities} and \ref{prop:exp_decay_impurities}. The proof of Proposition~\ref{prop:largest_cluster_impurities} requires more work, so we postpone it to Section~\ref{sec:proof_largest_imp}.

\subsubsection{Crossing holes} \label{sec:app_hole}

In our reasonings, we often use results regarding large holes in an annulus, analogous to Lemmas~3.2 and 3.3 in \cite{BN2018}. We just state the lemmas without providing proofs, since, compared with the corresponding results in \cite{BN2018}, only straightforward modifications are required (in order to make use of Assumption~\ref{ass:impurities} in the summations).

First, recall the upper bound on the event $\calH(A)$ that a crossing hole exists in a given annulus $A$, which was stated earlier as Lemma~\ref{lem:crossing_hole}.
\begin{lemma} \label{lem:crossing_hole2}
There exist constants $C, C'>0$ (depending only on $c$ and $\gamma$) such that the following holds. For all $m\geq 1$, for all annuli $A= A_{n_1, n_2}(z)$ with $z\in V$ and $1\leq n_1\leq \frac{n_2}{2}$,
$$\PPh^{(m)} \left(\calH(A)\right) \leq C \upsilon(m) e^{-C' \frac{n_1}{m}}.$$
\end{lemma}

The next lemma is needed in the proof of Theorem~\ref{thm:four-arm}, for the following sub-event of $\calH(A)$:
$$\ol{\ol{\calH}}(A) := \{\exists v \in V \: : \: H_v \cap \dout \Ball_{n_1}(z) \neq \emptyset, H_v \cap \din \Ball_{n_2}(z) \neq \emptyset, H_v \not \supseteq \Ball_{n_1}(z), \text{ and } H_v \cap \din \Ball_{2n_2}(z) = \emptyset\}.$$
\begin{lemma} \label{lem:crossing_hole3}
There exist constants $C, C'>0$ (depending only on $c$ and $\gamma$) such that the following holds. For all $m\geq 1$, for all annuli $A= A_{n_1, n_2}(z)$ with $z\in V$ and $1\leq n_1\leq \frac{n_2}{2}$,
$$\PPh^{(m)} \big(\ol{\ol{\calH}}(A)\big) \leq C \frac{\upsilon(m)}{m^2} n_1 n_2^{\gamma-1} \frac{1}{m^{\gamma-2}} e^{-C'\frac{n_2}{m}}.$$
\end{lemma}

\subsubsection{Four-arm stability} \label{sec:app_four}

We now state an analogous statement of Theorem~4.1 in \cite{BN2018}. Following the notations in \cite{BN2018}, for $\omega \in \Omega$ and $U\subseteq V$, we denote $(\omega^{(U)}) = (\omega^{(U)}_v)_{v\in V}$, where
$$\omega^{(U)}_v := \omega_v \mathbbm{1}_{v\not\in \bigcup_{u\in U} H_u},$$
and for an annulus $A = \Ann_{n_1, n_2}(z)$, let
$$\mathcal{W}_4(A) := \{\exists U\subseteq V \: : \: \omega^{(U)} \text{ satisfies } \mathcal{A}_4(A)\}.$$

\begin{theorem} \label{thm:four-arm}
Let $K \geq 1$. There exist $C$, $\delta_0 > 0$ (both depending on $c$, $\gamma$, $K$) such that the following holds. For all $p\in (0,1)$, and all $1\leq n_1 < n_2 \leq K(m\wedge L(p))$, if $\upsilon(m) \leq \delta_0$, then
$$\PPh_p^{(m)}\left(\mathcal{W}_4(A_{n_1, n_2})\right) \leq C \pi_4(n_1, n_2).$$
\end{theorem}

\begin{proof}[Proof of Theorem~\ref{thm:four-arm}]
The proof is almost the same as that of Theorem~4.1 in \cite{BN2018}, in the particular case when $\alpha>1$ (in the notation of that paper). Here, note that we use crucially the assumption $\gamma \in (1,2)$.
	
As in \cite{BN2018}, it suffices to consider the case $n_1=2^i$, $n_2 = 2^j$ with $j-i\geq 7$ and $2^j \leq 2K(m\wedge L(p))$. The result is proved by induction over $j$ and $(j-i)$. Let $\mathcal{D}$ be the event that $\mathcal{A}_4(\Ann_{2^{i+3}, 2^{j-3}})^c$ holds without the holes. We further define the following events:
	\begin{itemize}
		\item $\mathcal{E}_1 := \{\text{there is no big hole in $\Ann_{2^i, 2^j}$}\}$,
		
		\item $\mathcal{E}_2 :=\{\text{there is at least one big hole in $\Ann_{2^i, 2^j}$}\}$,
	\end{itemize}
	where a big hole in $\Ann_{2^i, 2^j}$ is a hole $H_v$ crossing at least one of the sub-annuli $\Ann_{2^h, 2^{h+1}} \subseteq \Ann_{2^i, 2^j}$, $i \leq h \leq j-1$. We start by writing
	\begin{align}
		\PPh_p^{(m)} \left(\mathcal{W}_4(A_{n_1, n_2}) \cap \mathcal{D}\right) & \leq \PPh_p^{(m)} \left(\mathcal{W}_4(A_{n_1, n_2}) \cap \mathcal{D} \cap \mathcal{E}_1\right) + \PPh_p^{(m)} \left(\mathcal{W}_4(A_{n_1, n_2}) \cap \mathcal{E}_2\right) \nonumber\\[1mm]
		& =: (\text{Term 1}) + (\text{Term 2}). 	\label{eq:term1term2}
	\end{align}
	Proceeding as in Section~4.2.1 of \cite{BN2018} (and using Lemma~\ref{lem:crossing_hole3}), we can obtain the following upper bounds:
	\begin{equation} \label{eq:term12}
	(\text{Term 1}) \leq C_1 C^3 \upsilon(m) \pi_4(2^i, 2^j) \quad \text{and} \quad (\text{Term 2}) \leq C_2 C \upsilon(m) \pi_4(2^i, 2^j),
	\end{equation}
	where $C$ is from the induction hypothesis, and $C_1$, $C_2$ depend only on $c$, $\gamma$, and $K$. Hence, combining \eqref{eq:term1term2}, \eqref{eq:term12}, and the hypothesis $\upsilon(m) \leq \delta_0$, we have
$$\PPh_p^{(m)} (\mathcal{W}_4(A_{n_1, n_2})\cap \mathcal{D}) \leq (C_1C^3+C_2C)\delta_0 \pi_4(2^i, 2^j).$$
As in Section~4.2.3 of \cite{BN2018}, this implies
$$\PPh_p^{(m)} (\mathcal{W}_4(A_{n_1, n_2})) \leq \hat{C} \pi_4(2^i, 2^j) + (C_1 C^3+C_2C)\delta_0 \pi_4(2^i, 2^j)$$
for some universal $\hat{C}>0$, and we choose $C = \hat{C} + 1$. If furthermore $\delta_0 > 0$ is sufficiently small, so that $(C_1 C^3+C_2C) \delta_0 < 1$, Theorem~\ref{thm:four-arm} follows by induction.
\end{proof}

\subsubsection{One-arm and box crossing stability} \label{sec:app_box}

We now state the analogs of Propositions~5.1 and 5.2 in \cite{BN2018}. Once Theorem~\ref{thm:four-arm} is established, these results can be obtained in essentially the same way as in \cite{BN2018}, so we will not repeat the proofs. Note that Proposition~\ref{prop:crossing_impurities2} below is exactly Proposition~\ref{prop:crossing_impurities}, stated earlier.

\begin{proposition}
Let $K \geq 1$. There exists $C = C(c, \gamma, K) > 0$ such that: for all $p \in (0,1)$ and $1 \leq n_1 \leq \frac{n_2}{32} \leq n_2 \leq K (m \wedge L(p))$,
\begin{equation}
\PPh_p^{(m)} \big( \arm_1(\Ann_{n_1, n_2}) \big) \geq \big( 1 - C \upsilon(m) \big) \, \PP_p \big( \arm_1(\Ann_{n_1, n_2}) \big).
\end{equation}
\end{proposition}

\begin{proposition} \label{prop:crossing_impurities2}
Let $K \geq 1$. There exists $C = C(c, \gamma, K) > 0$ such that: for all $p \in (0,1)$ and $1 \leq n \leq K (m \wedge L(p))$,
\begin{equation}
\PPh_p^{(m)} \big( \Ch([0,2n] \times [0,n]) \big) \geq \big( 1 - C \upsilon(m) \big) \, \PP_p \big( \Ch([0,2n] \times [0,n]) \big).
\end{equation}
\end{proposition}

\subsubsection{Box crossing probabilities in the supercritical regime} \label{sec:app_exp}

We use Proposition~\ref{prop:crossing_impurities2} to derive an estimate on box crossing probabilities in the supercritical regime $p > p_c$, similar to Proposition~5.3 of \cite{BN2018} but under Assumption~\ref{ass:impurities}. The result that we prove implies in particular (for $\alpha = \frac{1}{2}$) the stretched exponential property Proposition~\ref{prop:exp_decay_impurities}. In addition, we also take into account rectangles with a width $n$ between $L(p)$ and $m$, in the case when $L(p) \ll m$. We will need it later to derive, as an input for the proof of Proposition~\ref{prop:largest_cluster_impurities}, an analog of Corollary~5.4 of \cite{BN2018} (this result can only be applied when $L(p) \asymp m$, but for our purpose we also need to tackle the case $L(p) \ll m$).

\begin{proposition} \label{prop:exp_decay_impurities2}
Let $K\geq 1$ and $\alpha\in(0,1)$. There exist constants $\lambda_i > 0$, $1 \leq i \leq 5$, and $\delta_0  >0$, which depend only on $c$, $\gamma$, $K$, and $\alpha$, such that if $\upsilon(m) \leq \delta_0$ for all $m \geq 1$, then we have the following. For all $m \geq 1$, $n \geq 1$, and $p > p_c$ with $L(p) \leq K m$,
\begin{equation} \label{eq:exp_bound_impurities}
\PPh_p^{(m)} \big( \Ch([0,2n] \times [0,n]) \big) \geq
\begin{cases}
1-\lambda_1 \upsilon(m) - \lambda_2 e^{-\lambda_3 (\frac{n}{L(p)})^\alpha} & \text{if $n \leq m$,}\\[1mm]
1-\lambda_4 e^{-\lambda_5 (\frac{n}{m})^\alpha} & \text{if $n \geq m$.}
\end{cases}
\end{equation}
\end{proposition}

\begin{proof}[Proof of Proposition~\ref{prop:exp_decay_impurities2}]
Consider, for all $n \geq 1$,
$$f(n) := \PPh_p^{(m)} \big( \Cv^*([0,2n] \times [0,n]) \big),$$
and fix $\eta = \eta(\alpha) \in (0,\frac{1}{4}]$ small enough such that
$$\frac{\log{2}}{\log{2}+\log{(1+\eta)}} \geq \alpha.$$
Proceeding as in the proof of Proposition~5.3 of \cite{BN2018}, using a block argument and a similar summation as for Lemma~\ref{lem:crossing_hole2} (based on Assumption~\ref{ass:impurities}), we have
\begin{equation} \label{eq:f_iteration}
f(2(1+\eta)n) \leq 2 C \upsilon(m)e^{- C' \frac{n}{m}} + C'' f(n)^2
\end{equation}
for some constants $C, C'$ and $C''$, which depend only on $c$, $\gamma$, and (through $\eta$) $\alpha$.

We then use \eqref{eq:f_iteration} iteratively to bound $f(n)$. We start from $n_0 = K_0 L(p)$, where $K_0$ is chosen sufficiently large such that: for all $p > p_c$ with $L(p) \leq Km$,
$$f(K_0L(p)) \leq \frac{1}{20C''}.$$
Such a $K_0$ exists by Proposition~\ref{prop:crossing_impurities2} and \eqref{eq:exp_decay}, if $\delta_0$ has been chosen small enough (in terms of the constant $C = C(c, \gamma, K K_0) > 0$ appearing in Proposition~\ref{prop:crossing_impurities2}). Write $\lambda = 2(1+\eta)$. By induction, we get that if, in addition, $\delta_0 \leq (25 C C'')^{-1}$, then: for all $k \geq 0$,
$$f(\lambda^k n_0) \leq \frac{5C}{2}\upsilon(m) + \frac{1}{10 C''}2^{-2^k}$$
(using the rough bound $e^{-C'\frac{n}{m}} \leq 1$ in \eqref{eq:f_iteration}). We can then conclude the proof in a similar way as for Proposition~5.3 in \cite{BN2018}, and obtain the desired bound \eqref{eq:exp_bound_impurities} in the case $n \leq m$ (we just need this case, but in fact the bound also holds for $n \geq m$).

Finally, by decreasing the value of $\delta_0$ if necessary, we can obtain the bound \eqref{eq:exp_bound_impurities} in the case $n \geq m$ by repeating the same arguments as for Proposition~5.3 in \cite{BN2018}. We omit the details.
\end{proof}

We conclude this section by stating a result analogous to Corollary~5.4 of \cite{BN2018}, but which includes additionally the case when $L(p) \ll m$. The proof is similar to that in \cite{BN2018}, based on a standard argument using overlapping rectangles, and we leave the details to the reader (we need to apply \eqref{eq:exp_bound_impurities} carefully, depending on whether the width $n$ of the rectangle is $\leq m$ or not).
\begin{corollary} \label{cor:theta_imp}
Let $K\geq 1$ and $\ve > 0$. There exist $\ul{\lambda}$, $\ol{\lambda}$, $\kappa_0$, $\delta_0 > 0$ (depending on $c$, $\gamma$, $K$, $\ve$) such that if $\upsilon(m) \leq \delta_0$ for all $m \geq 1$, then: for all $p>p_c$ with $L(p)\leq Km$, and all $ n \geq \kappa_0 L(p)$, 
\begin{equation} \label{eq:cor_5.4}
(1 - \ve) \left(1-\ul{\lambda} \upsilon(m) \right) \theta(p) \leq \PPh_p^{(m)}(0 \lra \din \Ball_n) \leq (1 + \ve) \left(1+\ol{\lambda} \upsilon(m)\right)^{\log\left(\frac{m}{L(p)}\right)} \theta(p).
\end{equation}
\end{corollary}

\subsection{Largest cluster in a box: percolation with impurities} \label{sec:proof_largest_imp}

We now explain how to obtain Proposition~\ref{prop:largest_cluster_impurities} by adapting the proof of Lemma \ref{lem:largest_cluster_quant} to the process with impurities.

\begin{proof}[Proof of Proposition~\ref{prop:largest_cluster_impurities}]

First, we prove upper bounds analogous to \eqref{eq:chi_fin} and \eqref{eq:chi_cov}. For some constants $C_1, C_2$, we have: for all $p > p_c$ with $L(p) \leq K m$,
\begin{equation} \label{eq:chi_fin_imp}
\chi^{\textrm{fin}, (m)}(p) := \EE^{(m)}_p \big[ |\cluster(0)| \mathbbm{1}_{|\cluster(0)| < \infty} \big] \leq C_1 m^2 \theta(p)^2
\end{equation}
and
\begin{equation} \label{eq:chi_cov_imp}
\chi^{\textrm{cov}, (m)}(p) := \sum_{v \in V} \textrm{Cov}^{(m)}_p(\mathbbm{1}_{0 \lra \infty}, \mathbbm{1}_{v \lra \infty}) \leq C_2 m^2 \theta(p)^2.
\end{equation}

The bound \eqref{eq:chi_fin_imp} can be obtained from a similar summation argument as in Theorem~3 of \cite{Ke1987}, starting from
\begin{align}
\chi^{\textrm{fin}, (m)}(p) & = \sum_{v \in V} \PP^{(m)}_p(0 \lra v, 0 \not \lra \infty) \nonumber \\[1.5mm]
& = \sum_{\substack{v \in V\\ \|v\| < 4m}} \PP^{(m)}_p(0 \lra v, 0 \not \lra \infty) + \sum_{\substack{v \in V\\ \|v\| \geq 4m}} \PP^{(m)}_p(0 \lra v, 0 \not \lra \infty). \label{eq:upper_chi_fin1}
\end{align}
On the one hand,
\begin{equation} \label{eq:upper_chi_fin2}
\sum_{\substack{v \in V\\ \|v\| < 4m}} \PP^{(m)}_p(0 \lra v, 0 \not \lra \infty) \leq \sum_{\substack{v \in V\\ \|v\| < 4m}} \PP^{(m)}_p(0 \lra v) \leq \sum_{\substack{v \in V\\ \|v\| < 4m}} \PP_p(0 \lra v) \leq C m^2 \theta(p)^2,
\end{equation}
as for Bernoulli percolation. On the other hand, for $v \in V$ with $\|v\| \geq 4m$, we have (denoting $n = \frac{\|v\|}{4} \geq m$):
\begin{align*}
\PP^{(m)}_p(0 \lra v, 0 \not \lra \infty) & \leq \PP^{(m)}_p(\{0 \lra \din \Ball_n(0), v \lra \din \Ball_n(v)\} \cap \circuitevent^*_n)\\[1.5mm]
& \leq \PP^{(m)}_p(0 \lra \din \Ball_n(0), v \lra \din \Ball_n(v)) \, \PP^{(m)}_p(\circuitevent^*_n),
\end{align*}
where $\circuitevent^*_n$ is the event that there exists a vacant circuit $\circuit^*$ which surrounds $0$ and $v$ (using the FKG inequality). Hence,
$$\PP^{(m)}_p(0 \lra v, 0 \not \lra \infty) \leq \PP_p(0 \lra \din \Ball_n(0), v \lra \din \Ball_n(v)) \, \PP^{(m)}_p(\circuitevent^*_n) \leq C' \theta(p)^2 \, \PP^{(m)}_p(\circuitevent^*_n).$$
Using that $\PP^{(m)}_p(\circuitevent^*_n)$ decays exponentially fast in $\frac{n}{m}$, we obtain by summation:
\begin{equation} \label{eq:upper_chi_fin3}
\sum_{\substack{v \in V\\ \|v\| \geq 4m}} \PP^{(m)}_p(0 \lra v, 0 \not \lra \infty) \leq C'' m^2 \theta(p)^2.
\end{equation}
We deduce \eqref{eq:chi_fin_imp}, by combining \eqref{eq:upper_chi_fin1}, \eqref{eq:upper_chi_fin2}, and \eqref{eq:upper_chi_fin3}.

We now turn to \eqref{eq:chi_cov_imp}, which we establish by adapting the proof of \eqref{eq:chi_cov} from Section~6.4 of \cite{BCKS2001}. As we explain, some non-trivial modifications are required for this bound, due to the presence of impurities. We start from
\begin{equation} \label{eq:upper_chi_cov1}
\chi^{\textrm{cov}, (m)}(p) = \sum_{\substack{v \in V\\ \|v\| < 4m}} \textrm{Cov}^{(m)}_p(\mathbbm{1}_{0 \lra \infty}, \mathbbm{1}_{v \lra \infty}) + \sum_{\substack{v \in V\\ \|v\| \geq 4m}} \textrm{Cov}^{(m)}_p(\mathbbm{1}_{0 \lra \infty}, \mathbbm{1}_{v \lra \infty}).
\end{equation}

First, 
\begin{align}
\sum_{\substack{v \in V\\ \|v\| < 4m}} \textrm{Cov}^{(m)}_p(\mathbbm{1}_{0 \lra \infty}, \mathbbm{1}_{v \lra \infty}) & \leq \sum_{\substack{v \in V\\ \|v\| < 4m}} \PP^{(m)}_p(0 \lra \infty, v \lra \infty) \nonumber \\[1.5mm]
& \leq \sum_{\substack{v \in V\\ \|v\| < 4m}} \PP_p(0 \lra \infty, v \lra \infty) \leq C m^2 \theta(p)^2. \label{eq:upper_chi_cov2}
\end{align}
For $v \in V$ with $\|v\| \geq 4m$, we consider $\Ball_n(0)$ and $\Ball_n(v)$, where $n = \frac{\|v\|}{4} \geq m$. We can write
\begin{align}
\textrm{Cov}^{(m)}_p & (\mathbbm{1}_{0 \lra \infty}, \mathbbm{1}_{v \lra \infty}) = \textrm{Cov}^{(m)}_p(\mathbbm{1}_{0 \not \lra \infty}, \mathbbm{1}_{v \not \lra \infty}) \nonumber \\[1.5mm]
& = \textrm{Cov}^{(m)}_p(\mathbbm{1}_{0 \not \lra \infty, 0 \lra \din \Ball_n(0)}, \mathbbm{1}_{v \not \lra \infty, v \lra \din \Ball_n(v)}) + \textrm{Cov}^{(m)}_p(\mathbbm{1}_{0 \not \lra \din \Ball_n(0)}, \mathbbm{1}_{v \not \lra \infty, v \lra \din \Ball_n(v)}) \nonumber \\[1mm]
& \hspace{1cm} + \textrm{Cov}^{(m)}_p(\mathbbm{1}_{0 \not \lra \infty, 0 \lra \din \Ball_n(0)}, \mathbbm{1}_{v \not \lra \din \Ball_n(v)}) + \textrm{Cov}^{(m)}_p(\mathbbm{1}_{0 \not \lra \din \Ball_n(0)}, \mathbbm{1}_{v \not \lra \din \Ball_n(v)}) \nonumber \\[1.5mm]
& =: (A) + 2 (B) + (C). \label{eq:upper_chi_cov3}
\end{align}

For term $(A)$, we have
\begin{align*}
(A) = \textrm{Cov}^{(m)}_p & (\mathbbm{1}_{0 \not \lra \infty, 0 \lra \din \Ball_n(0)}, \mathbbm{1}_{v \not \lra \infty, v \lra \din \Ball_n(v)})\\[1.5mm]
& \leq \PP^{(m)}_p(0 \lra \din \Ball_n(0), v \lra \din \Ball_n(v), 0 \not \lra \infty, v \not \lra \infty)\\[1.5mm]
& \leq \PP^{(m)}_p(\{0 \lra \din \Ball_n(0), v \lra \din \Ball_n(v)\} \cap \circuitevent^*_1 \cap \circuitevent^*_2),
\end{align*}
where $\circuitevent^*_1 := \{$there exists a vacant circuit which surrounds $0$ and intersects $(\Ball_n(0))^c\}$, and $\circuitevent^*_2$ is defined in a similar way, with $v$ in place of $0$. From the FKG inequality, this is at most
$$\PP^{(m)}_p(0 \lra \din \Ball_n(0), v \lra \din \Ball_n(v)) \cdot \PP^{(m)}_p(\circuitevent^*_1 \cap \circuitevent^*_2).$$
Then,
$$\PP^{(m)}_p(0 \lra \din \Ball_n(0), v \lra \din \Ball_n(v)) \leq \PP_p(0 \lra \din \Ball_n(0), v \lra \din \Ball_n(v)) \leq (C \theta(p))^2,$$
and $\PP^{(m)}_p(\circuitevent^*_1 \cap \circuitevent^*_2)$ provides an exponentially decaying term (using Proposition~\ref{prop:exp_decay_impurities}).

As for term $(B)$, we have
\begin{align*}
(B) = \textrm{Cov}^{(m)}_p(\mathbbm{1}_{0 \not \lra \din \Ball_n(0)}, \mathbbm{1}_{v \not \lra \infty, v \lra \din \Ball_n(v)}) & = - \textrm{Cov}^{(m)}_p(\mathbbm{1}_{0 \lra \din \Ball_n(0)}, \mathbbm{1}_{v \not \lra \infty, v \lra \din \Ball_n(v)})\\[1.5mm]
& \leq \PP^{(m)}_p(0 \lra \din \Ball_n(0)) \cdot \PP^{(m)}_p(v \not \lra \infty, v \lra \din \Ball_n(v))\\[1.5mm]
& \leq C\theta(p) \cdot C \theta(p) c_1 e^{-c_2 (n/m)^{1/2}}.
\end{align*}
Here, we use that
$$\PP^{(m)}_p(v \not \lra \infty \, | \, v \lra \din \Ball_n(v)) \leq c_1 e^{-c_2 (n/m)^{1/2}}$$
(from a construction with overlapping rectangles, and the stretched exponential decay property, Proposition~\ref{prop:exp_decay_impurities}).

We finally examine term $(C)$. Observe that the events $\{0 \not \lra \din \Ball_n(0)\}$ and $\{v \not \lra \din \Ball_n(v)\}$ are not exactly independent, so we cannot say immediately (unlike in Bernoulli percolation) that their covariance is equal to $0$. Indeed, some large impurity may intersect both $\Ball_n(0)$ and $\Ball_n(v)$. However, this is unlikely for $n$ large enough, so these events are almost independent. More precisely, this term can be handled by using the following auxiliary events. Let $E_0 := \{0 \lra \din \Ball_n(0)\}$, $\tilde{E}_0 := \{0 \lra \din \Ball_n(0)$ without the impurities centered in $(\Ball_{2n})^c\}$, and $\tilde{\tilde{E}}_0 := \{0 \lra \din \Ball_n(0)$ without the impurities$\}$ (so that $E_0 \subseteq \tilde{E}_0 \subseteq \tilde{\tilde{E}}_0$), and similarly with $v$. We have
\begin{align*}
(C) = \textrm{Cov}^{(m)}_p( & \mathbbm{1}_{0 \not \lra \din \Ball_n(0)}, \mathbbm{1}_{v \not \lra \din \Ball_n(v)}) = \textrm{Cov}^{(m)}_p(\mathbbm{1}_{E_0^c}, \mathbbm{1}_{E_v^c}) = \textrm{Cov}^{(m)}_p(\mathbbm{1}_{E_0}, \mathbbm{1}_{E_v})\\[1.5mm]
& = \textrm{Cov}^{(m)}_p(\mathbbm{1}_{E_0} - \mathbbm{1}_{\tilde{E}_0}, \mathbbm{1}_{E_v}) + \textrm{Cov}^{(m)}_p(\mathbbm{1}_{\tilde{E}_0}, \mathbbm{1}_{E_v} - \mathbbm{1}_{\tilde{E}_v}) + \textrm{Cov}^{(m)}_p(\mathbbm{1}_{\tilde{E}_0}, \mathbbm{1}_{\tilde{E}_v}).
\end{align*}
Clearly, $\textrm{Cov}^{(m)}_p(\mathbbm{1}_{\tilde{E}_0}, \mathbbm{1}_{\tilde{E}_v}) = 0$, since $\tilde{E}_0$ and $\tilde{E}_v$ are independent. On the other hand,
\begin{align*}
\textrm{Cov}^{(m)}_p(\mathbbm{1}_{E_0} - \mathbbm{1}_{\tilde{E}_0}, \mathbbm{1}_{E_v}) & = - \textrm{Cov}^{(m)}_p(\mathbbm{1}_{\tilde{E}_0 \setminus E_0}, \mathbbm{1}_{E_v}) \leq \PP^{(m)}_p(\tilde{E}_0 \setminus E_0) \cdot \PP^{(m)}_p(E_v)\\[1.5mm]
& \leq \PP^{(m)}_p(\tilde{\tilde{E}}_0) \cdot \PP_p(\calH(\Ann_{n,2n})) \cdot \PP^{(m)}_p(\tilde{\tilde{E}}_v)\\[1.5mm]
& = \PP_p(0 \lra \din \Ball_n(0)) \cdot \PP_p(\calH(\Ann_{n,2n})) \cdot \PP_p(v \lra \din \Ball_n(v))
\end{align*}
(for the second inequality, we used that $\tilde{E}_0 \setminus E_0 \subseteq \tilde{\tilde{E}}_0 \cap \calH(\Ann_{n,2n})$, and the independence of $\tilde{\tilde{E}}_0$ and $\calH(\Ann_{n,2n})$). Hence,
$$\textrm{Cov}^{(m)}_p(\mathbbm{1}_{E_0} - \mathbbm{1}_{\tilde{E}_0}, \mathbbm{1}_{E_v}) \leq C \theta(p) \cdot c_1 e^{-c_2 n/m} \cdot C \theta(p),$$
and for similar reasons, the same upper bound holds for $\textrm{Cov}^{(m)}_p(\mathbbm{1}_{\tilde{E}_0}, \mathbbm{1}_{E_v} - \mathbbm{1}_{\tilde{E}_v})$.

The upper bounds that we obtained for terms $(A)$, $(B)$ and $(C)$ in \eqref{eq:upper_chi_cov3} yield, by summation,
\begin{equation}
\sum_{\substack{v \in V\\ \|v\| \geq 4m}} \textrm{Cov}^{(m)}_p(\mathbbm{1}_{0 \lra \infty}, \mathbbm{1}_{v \lra \infty}) \leq C m^2 \theta(p)^2.
\end{equation}
Combined with \eqref{eq:upper_chi_cov1} and \eqref{eq:upper_chi_cov2}, this completes the proof of \eqref{eq:chi_cov_imp}.

The proof in Section~\ref{sec:proof_largest} can now easily be adapted, using \eqref{eq:chi_fin_imp} and \eqref{eq:chi_cov_imp}, and replacing $L(p)$ by $m$ everywhere. In particular, \eqref{eq:exp_cinf_ball} can be obtained from Corollary~\ref{cor:theta_imp}, while the analog of \eqref{eq:cinf_circuit} follows from Proposition~\ref{prop:exp_decay_impurities}. We leave the details to the reader.

\end{proof}

\subsection*{Acknowledgments}

We sincerely thank Rob van den Berg for many valuable comments on an earlier version of this paper, which improved its readability.

\bibliographystyle{plain}
\bibliography{Avalanches}

\begin{thebibliography}{10}

\bibitem{Al2000}
David~J. Aldous.
\newblock The percolation process on a tree where infinite clusters are frozen.
\newblock {\em Math. Proc. Cambridge Philos. Soc.}, 128(3):465--477, 2000.

\bibitem{Ba1996}
Per Bak.
\newblock {\em How nature works: the science of self-organized criticality}.
\newblock Copernicus, New York, 1996.

\bibitem{BCKS1999}
Christian Borgs, Jennifer~T. Chayes, Harry Kesten, and Joel Spencer.
\newblock Uniform boundedness of critical crossing probabilities implies
  hyperscaling.
\newblock {\em Random Structures Algorithms}, 15(3-4):368--413, 1999.

\bibitem{BCKS2001}
Christian Borgs, Jennifer~T. Chayes, Harry Kesten, and Joel Spencer.
\newblock {The birth of the infinite cluster: finite-size scaling in
  percolation}.
\newblock {\em Comm. Math. Phys.}, 224(1):153–204, 2001.

\bibitem{BF2013}
Xavier Bressaud and Nicolas Fournier.
\newblock One-dimensional general forest fire processes.
\newblock {\em M\'{e}m. Soc. Math. Fr. (N.S.)}, (132):vi+138, 2013.

\bibitem{BH1957}
S.~R. Broadbent and J.~M. Hammersley.
\newblock Percolation processes. {I}. {C}rystals and mazes.
\newblock {\em Proc. Cambridge Philos. Soc.}, 53:629--641, 1957.

\bibitem{DSV2009}
Michael Damron, Art{\"e}m Sapozhnikov, and B{\'a}lint V{\'a}gv{\"o}lgyi.
\newblock Relations between invasion percolation and critical percolation in
  two dimensions.
\newblock {\em Ann. Probab.}, 37(6):2297--2331, 2009.

\bibitem{DrSc1992}
Barbara Drossel and Franz Schwabl.
\newblock Self-organized critical forest-fire model.
\newblock {\em Phys. Rev. Lett.}, 69:1629--1632, 1992.

\bibitem{Du2006}
Maximilian D{\"u}rre.
\newblock Existence of multi-dimensional infinite volume self-organized
  critical forest-fire models.
\newblock {\em Electron. J. Probab.}, 11:no. 21, 513--539, 2006.

\bibitem{GPS2018a}
Christophe Garban, G\'abor Pete, and Oded Schramm.
\newblock The scaling limits of near-critical and dynamical percolation.
\newblock {\em J. Eur. Math. Soc.}, 20(5):1195--1268, 2018.

\bibitem{GPS2018b}
Christophe Garban, G\'{a}bor Pete, and Oded Schramm.
\newblock The scaling limits of the minimal spanning tree and invasion
  percolation in the plane.
\newblock {\em Ann. Probab.}, 46(6):3501--3557, 2018.

\bibitem{Gr1999}
Geoffrey Grimmett.
\newblock {\em Percolation}, volume 321 of {\em Grundlehren der Mathematischen
  Wissenschaften}.
\newblock Springer-Verlag, Berlin, second edition, 1999.

\bibitem{Je1998}
Henrik~J. Jensen.
\newblock {\em Self-organized criticality: emergent complex behavior in
  physical and biological systems}, volume~10 of {\em Cambridge Lecture Notes
  in Physics}.
\newblock Cambridge University Press, Cambridge, 1998.

\bibitem{Ke1980}
Harry Kesten.
\newblock {The critical probability of bond percolation on the square lattice
  equals 1/2}.
\newblock {\em Comm. Math. Phys.}, 74(1):41–59, 1980.

\bibitem{Ke1982}
Harry Kesten.
\newblock {\em Percolation theory for mathematicians}, volume~2 of {\em
  Progress in Probability and Statistics}.
\newblock Birkh\"auser, Boston, 1982.

\bibitem{Ke1987}
Harry Kesten.
\newblock Scaling relations for {$2$}{D}-percolation.
\newblock {\em Comm. Math. Phys.}, 109(1):109--156, 1987.

\bibitem{Ki2015}
Demeter Kiss.
\newblock Frozen percolation in two dimensions.
\newblock {\em Probab. Theory Relat. Fields}, 163(3-4):713--768, 2015.

\bibitem{KMS2015}
Demeter Kiss, Ioan Manolescu, and Vladas Sidoravicius.
\newblock Planar lattices do not recover from forest fires.
\newblock {\em Ann. Probab.}, 43(6):3216--3238, 2015.

\bibitem{LSW_2001a}
Gregory~F. Lawler, Oded Schramm, and Wendelin Werner.
\newblock Values of {B}rownian intersection exponents. {I}. {H}alf-plane
  exponents.
\newblock {\em Acta Math.}, 187(2):237--273, 2001.

\bibitem{LSW_2001b}
Gregory~F. Lawler, Oded Schramm, and Wendelin Werner.
\newblock Values of {B}rownian intersection exponents. {II}. {P}lane exponents.
\newblock {\em Acta Math.}, 187(2):275--308, 2001.

\bibitem{LSW2002}
Gregory~F. Lawler, Oded Schramm, and Wendelin Werner.
\newblock One-arm exponent for critical 2{D} percolation.
\newblock {\em Electron. J. Probab.}, 7:no. 2, 13 pp., 2002.

\bibitem{Li2005}
Thomas~M. Liggett.
\newblock {\em Interacting particle systems}.
\newblock Classics in Mathematics. Springer-Verlag, Berlin, 2005.
\newblock Reprint of the 1985 original.

\bibitem{No2008}
Pierre Nolin.
\newblock Near-critical percolation in two dimensions.
\newblock {\em Electron. J. Probab.}, 13:no. 55, 1562--1623, 2008.

\bibitem{NW2009}
Pierre Nolin and Wendelin Werner.
\newblock Asymmetry of near-critical percolation interfaces.
\newblock {\em J. Amer. Math. Soc.}, 22(3):797--819, 2009.

\bibitem{Pr2012}
Gunnar Pruessner.
\newblock {\em Self-Organised Criticality: Theory, Models and
  Characterisation}.
\newblock Cambridge University Press, Cambridge, 2012.

\bibitem{Ra2009}
Bal\'{a}zs R\'{a}th.
\newblock Mean field frozen percolation.
\newblock {\em J. Stat. Phys.}, 137(3):459--499, 2009.

\bibitem{RST2019}
Bal\'{a}zs R\'{a}th, Jan~M. Swart, and Tam\'{a}s Terpai.
\newblock Frozen percolation on the binary tree is nonendogenous.
\newblock Preprint arXiv:1910.09213, 2019.

\bibitem{RT2009}
Bal\'{a}zs R\'{a}th and B\'{a}lint T\'{o}th.
\newblock Erd{\H o}s-{R}\'{e}nyi random graphs {$+$} forest fires {$=$}
  self-organized criticality.
\newblock {\em Electron. J. Probab.}, 14:no. 45, 1290--1327, 2009.

\bibitem{Sm2001}
Stanislav Smirnov.
\newblock Critical percolation in the plane: conformal invariance, {C}ardy's
  formula, scaling limits.
\newblock {\em C. R. Acad. Sci. Paris S\'er. I Math.}, 333(3):239--244, 2001.

\bibitem{SW2001}
Stanislav Smirnov and Wendelin Werner.
\newblock Critical exponents for two-dimensional percolation.
\newblock {\em Math. Res. Lett.}, 8(5-6):729--744, 2001.

\bibitem{St1943}
Walter~H. {Stockmayer}.
\newblock {Theory of molecular size distribution and gel formation in
  branched-chain polymers}.
\newblock {\em Journal of Chemical Physics}, 11:45–55, 1943.

\bibitem{BB2004}
Jacob van~den Berg and Rachel Brouwer.
\newblock Self-destructive percolation.
\newblock {\em Random Structures Algorithms}, 24(4):480--501, 2004.

\bibitem{BB2006}
Jacob van~den Berg and Rachel Brouwer.
\newblock Self-organized forest-fires near the critical time.
\newblock {\em Comm. Math. Phys.}, 267(1):265--277, 2006.

\bibitem{BLN2012}
Jacob van~den Berg, Bernardo N.~B. de~Lima, and Pierre Nolin.
\newblock A percolation process on the square lattice where large finite
  clusters are frozen.
\newblock {\em Random Structures Algorithms}, 40(2):220--226, 2012.

\bibitem{BKN2012}
Jacob van~den Berg, Demeter Kiss, and Pierre Nolin.
\newblock {A percolation process on the binary tree where large finite clusters
  are frozen}.
\newblock {\em Electron. Commun. Probab.}, 17(2):1–11, 2012.

\bibitem{BKN2015}
Jacob van~den Berg, Demeter Kiss, and Pierre Nolin.
\newblock Two-dimensional volume-frozen percolation: deconcentration and
  prevalence of mesoscopic clusters.
\newblock {\em Ann. Sci. \'Ec. Norm. Sup\'er. (4)}, 51(4):1017--1084, 2018.

\bibitem{BN2017}
Jacob van~den Berg and Pierre Nolin.
\newblock Boundary rules and breaking of self-organized criticality in 2{D}
  frozen percolation.
\newblock {\em Electron. Commun. Probab.}, 22:Paper No. 65, 15, 2017.

\bibitem{BN2015}
Jacob van~den Berg and Pierre Nolin.
\newblock Two-dimensional volume-frozen percolation: exceptional scales.
\newblock {\em Ann. Appl. Probab.}, 27(1):91--108, 2017.

\bibitem{BN2018}
Jacob van~den Berg and Pierre Nolin.
\newblock Near-critical percolation with heavy-tailed impurities, forest fires
  and frozen percolation.
\newblock {\em Probab. Theory Relat. Fields}, to appear.

\bibitem{BT2001}
Jacob van~den Berg and B\'{a}lint T\'{o}th.
\newblock A signal-recovery system: asymptotic properties, and construction of
  an infinite-volume process.
\newblock {\em Stochastic Process. Appl.}, 96(2):177--190, 2001.

\bibitem{We2009}
Wendelin Werner.
\newblock Lectures on two-dimensional critical percolation.
\newblock In {\em Statistical mechanics}, volume~16 of {\em IAS/Park City Math.
  Ser.}, pages 297--360. Amer. Math. Soc., Providence, RI, 2009.

\bibitem{WW1983}
David Wilkinson and Jorge~F. Willemsen.
\newblock Invasion percolation: a new form of percolation theory.
\newblock {\em Journal of Physics A: Mathematical and General}, 16:3365--3376,
  1983.

\end{thebibliography}

\end{document}